\documentclass[12pt]{article}
\usepackage[margin=1in]{geometry}
\RequirePackage{tgtermes}
\RequirePackage{newtxtext}
\RequirePackage{newtxmath}
\RequirePackage{bm}
\usepackage{setspace}
\usepackage{pgfplots}  % 必须加载这个宏包
\pgfplotsset{compat=1.18}
\usepackage[ruled,vlined,linesnumbered]{algorithm2e}
\SetSideCommentRight
\SetCommentSty{textnormal}
\SetKwInOut{Input}{Input}
\SetKwInOut{Output}{Output}
\usepackage{tikz}
\usepackage{natbib}
 \bibpunct[, ]{(}{)}{,}{a}{}{,}%

\usepackage{multibib}
\newcites{app}{References}

\usepackage{multirow}
\usepackage{multicol}
\usepackage{booktabs}
\usepackage{graphicx}
\usepackage{subfigure}
\usepackage{diagbox}
\usepackage{float}
\usepackage{makecell}
\usepackage{colortbl}
\usepackage{soul}
\usepackage[dvipsnames, svgnames, x11names]{xcolor}

\usetikzlibrary{patterns}
\usetikzlibrary{positioning, shapes}
\usetikzlibrary{calc}
\usetikzlibrary {shapes}
\usetikzlibrary {arrows.meta}
\usetikzlibrary{fit}

\usepackage{enumitem}
\usepackage{adjustbox}
\usepackage[colorlinks,
            linkcolor=red,
            anchorcolor=green,
            citecolor=blue,urlcolor=blue
            ]{hyperref}
\allowdisplaybreaks

\usepackage[all]{hypcap}  % 必须！修正跳转锚点位置
\hypersetup{hypertexnames=false} % 解决某些命名冲突

\usepackage{amsthm}
\usepackage{amsmath}
\newtheorem{theorem}{Theorem}
\newtheorem{lemma}[theorem]{Lemma}
\newtheorem{proposition}[theorem]{Proposition}
\newtheorem{corollary}[theorem]{Corollary}
\newtheorem{definition}{Definition}

\newtheorem{example}{Example}

\begin{document}
%%%%%%%%%%%%%%%%

% Removed commented VOLUME/NO/MONTH/YEAR metadata (INFORMS-specific)

% Removed RUNAUTHOR (INFORMS-specific)

% Removed RUNTITLE (INFORMS-specific)
% \RUNTITLE{The CRSP with Makespan Minimization}

% Full title:
\title{The Car Resequencing Problem with Makespan Minimization}

\author{
Xinyi Guo\textsuperscript{1} and Jean-Fran\c{c}ois C\^{o}t\'e\textsuperscript{*,2}\\[12pt]
{\normalsize \textsuperscript{1}\textit{Shenzhen International Graduate School, Tsinghua University, Shenzhen, China}}\\
{\normalsize \textsuperscript{2}\textit{CIRRELT, Universit\'e Laval, Qu\'ebec, Canada}}
}

\date{}

\maketitle

\begingroup
\renewcommand\thefootnote{}
\footnotetext{$^*$Corresponding author}
\footnotetext{\noindent Email addresses: \texttt{guoxinyi21@tsinghua.org.cn} (X.~Guo), \texttt{Jean-Francois.Cote@fsa.ulaval.ca} (J.-F.~C\^{o}t\'e)}
\endgroup
\setcounter{footnote}{0}

\begin{abstract}
\textbf{\textit{Problem definition}:} The car resequencing problem involves rearranging the sequence of cars between two adjacent production shops via an intermediate buffer. %This study explores a new variant that aims to minimize the resequencing makespan when the buffer features several forward lanes and one return lane (RL). Both types of lanes operate on a first-in-first-out basis, but in opposite directions. \textcolor{red}{Through the RL, cars can circulate between forward lanes.} %cars from any forward lane can loop back to re-enter another.
%The car resequencing problem involves rearranging the sequence of cars from a production shop into a new sequence for the next shop via an intermediate buffer.
This study explores a new variant that aims to minimize the resequencing makespan when the buffer features several forward lanes and a return lane (RL). Both types of lanes follow first-in-first-out rules, but operate in opposite directions. %Both types of lanes operate on a first-in-first-out basis, but in opposite directions.
Through the RL, cars can circulate between forward lanes, thereby increasing resequencing flexibility and complexity. %Cars can circulate between forward lanes via the RL, increasing resequencing flexibility and complexity.
The goal is to position cars over time to complete the target sequence change in the shortest possible time. \textbf{\textit{Methodology/results}:} We prove the problem is NP-complete and develop an exact two-stage branch-and-price approach integrated with constraint programming (BP-CP). Three versions of the master problem (MP) are proposed for the first stage, each paired with a subproblem defined on a specific graph. Using partial integer solutions from branching, a constraint-programming model built in the second stage yields either a feasible solution to the problem or a feasibility cut added to the MP. Experimental results indicate that BP-CP greatly increases the scale of solvable instances for a production batch, especially with a moderately relaxed MP. A comparative analysis shows that the RL markedly enhances resequencing flexibility, enabling smaller buffers to attain makespans comparable to those of larger buffers without the RL across varying resequencing complexities. \textbf{\textit{Managerial implications}:} Expanding lane capacity alone yields limited gains in resequencing performance and may even reduce efficiency. The RL acts
as several forward lanes, supporting a smaller and cost-effective buffer design while maintaining efficient just-in-time operations.
\end{abstract}


%%%%%%%%%%%%%%%%%%%%%%%%%%%%%%%%%%%%%%%%%%%%%%%%%%%%%%%%%%%%%%%%%%%%%%
% Text of your paper here
\section{Introduction}\label{sec1}
In car manufacturing, various car models with different attributes, such as color and configuration, are arranged in a sequence for processing on a mixed-model production line. This line runs through multiple consecutive shops, each with production costs that depend mainly on the car sequence. For example, the paint shop seeks to minimize color changes between adjacent cars, and the assembly shop aims to stabilize option consumption. More broadly, decisions regarding production sequences are critical for maintaining \textit{just-in-time} (JIT) operations and enhancing manufacturing performance \citep{kubiak2008just, bray2015production}. The \textit{car sequencing problem} typically concerns finding a sequence that minimizes costs either in a single shop or across multiple shops jointly \citep{boysen2009sequencing}.

%The \textit{car sequencing problem} concerns finding a sequence that minimizes costs in a single shop or across multiple shops jointly \citep{boysen2009sequencing}.

%In car manufacturing, various car models with different attributes, such as color and configuration, are arranged in a sequence for processing on a mixed-model production line. This line runs through multiple consecutive shops, each incurring production costs that depend mainly on the car sequence. For example, the paint shop prefers to minimize the color changes between adjacent cars, and the assembly shop aims to consume options steadily. The \textit{car sequencing problem} \citep{boysen2009sequencing} concerns determining the sequence that minimizes costs for a single shop or multiple shops jointly.

%%%%%%%To meet the growing demand for customized orders, the mixed-model production line (MMPL) is widely used in various industries to process a variety of products in a sequence. In automotive manufacturing, an MMPL runs through three consecutive shops: the body shop, the paint shop, and the assembly shop. The order of diverse car models arranged on this line significantly impacts the production costs in each shop. For the paint shop, an ideal sequence has the fewest color changes, while in the assembly shop, it aims to reduce work overload. The \textit{car sequencing problem}, thoroughly reviewed by \citet{boysen2009sequencing} and \citet{solnon2008car}, seeks a car sequence to optimize these single-shop or integrated multi-shop objectives.

Generally, the car sequence needs to be adjusted between two adjacent shops. One reason is that, driven by distinct production goals, the optimal sequence with the lowest costs preferred by each shop differs. Also, short-term disruptions like missing parts or quality defects force the removal of affected cars from the original sequence. To restore a desirable sequence, resequencing is necessary.

%%Generally, the car sequence needs to be adjusted between two adjacent shops. One reason is that, driven by distinct production goals, the optimal sequence with the lowest costs preferred by each shop differs. Also, short-term disruptions like missing parts or quality issues cause the removal of affected cars from the original sequence. To restore a desirable sequence, cars must be rearranged.

%%%%%%%Since different shops have their respective preferred sequences that result in the lowest costs, the car sequence typically requires adjustment between two adjacent shops. On the other hand, changes in the sequence are inevitable due to short-term disruptions in production, such as missing parts or quality problems, which cause the removal of affected cars from the original sequence. To regain a desirable sequence, cars must be rearranged \citep{boysen2012resequencing}.

\begin{figure}[htbp]
    \centering

 \begin{tikzpicture}[every node/.style={draw},
		rotate border/.style={shape border uses incircle, shape border rotate=#1}, transform shape,  % 允许整体变换
    scale=1.2]
        \matrix [draw=white,
		%inner xsep=0.2cm,
		column sep=0.03cm, row sep=0.03cm]
		{
			%row1
			\node [blue, text=white, minimum height=0.64cm, minimum width=0.76cm, line width=0.3pt] (n1) {};
			\pgfmatrixnextcell \node [blue, text=white, minimum height=0.64cm, minimum width=0.76cm, line width=0.3pt] (n2) {};
            \pgfmatrixnextcell \node [blue, text=white, minimum height=0.64cm, minimum width=0.76cm, line width=0.3pt] (n3) {}; \\
			%row2
			\node [blue, text=white, minimum height=0.64cm, minimum width=0.76cm, line width=0.3pt] (n4) {};
			\pgfmatrixnextcell \node [blue, text=white, minimum height=0.64cm, minimum width=0.76cm, line width=0.3pt] (n5) {};
            \pgfmatrixnextcell \node [blue, text=white, minimum height=0.64cm, minimum width=0.76cm, line width=0.3pt] (n6) {}; \\
			%row3
			\node [blue, text=white, minimum height=0.64cm, minimum width=0.76cm, line width=0.3pt] (n7) {};
			\pgfmatrixnextcell \node [blue, fill=white, text=white, minimum height=0.64cm, minimum width=0.76cm, line width=0.3pt] (n8) {};
            \pgfmatrixnextcell \node [blue, fill=white, text=white, minimum height=0.64cm, minimum width=0.76cm, line width=0.3pt] (n9) {}; \\
		};
        %%big rectangle
		\node[draw=black, line width=0.4pt, minimum height=1.85cm,  minimum width=2.5cm, rectangle, fill=none] (n01) at (0,0) {};
        %%arrows within cells
        \draw[-Stealth,
		blue,
		line width=0.4pt] ($(n1.west)+(0.08cm,0)$) -- ($(n3.east)+(-0.08cm,0)$);
        \draw[-Stealth,
		blue,
		line width=0.4pt] ($(n4.west)+(0.08cm,0)$) -- ($(n6.east)+(-0.08cm,0)$);
		\draw[-Stealth,
		blue,
		line width=0.4pt] ($(n7.west)+(0.08cm,0)$) -- ($(n9.east)+(-0.08cm,0)$);
        %%two arrows
        \draw[->,
		gray,
		line width=0.4pt] ($(n7.north)+(-0.45cm,-0.3cm)$) -- ($(n1.south)+(-0.45cm,0.3cm)$);
        \draw[->,
		gray,
		line width=0.4pt] ($(n3.south)+(0.45cm,0.3cm)$) -- ($(n9.north)+(0.45cm,-0.3cm)$);
        %% cars
		\node[draw, circle, fill=gray!30, left=of n7, xshift=+0.45cm,
        minimum size=0.32cm,
		inner sep=0.6pt, font=\scriptsize
		] (nc1) {1};
		\node[draw, circle, fill=gray!30, left=of nc1, xshift=+0.95cm, minimum size=0.32cm,
		inner sep=0.6pt, font=\scriptsize
		] (nc2) {2};
        \node[draw, circle, fill=gray!30, left=of nc2, xshift=+0.95cm, minimum size=0.32cm,
		inner sep=0.6pt, font=\scriptsize
		] (nc3) {3};
        \node[draw, circle, fill=gray!30, left=of nc3, xshift=+0.95cm, minimum size=0.32cm,
		inner sep=0.6pt, font=\scriptsize
		] (nc4) {4};
        \node[draw, circle, fill=gray!30, left=of nc4, xshift=+0.95cm, minimum size=0.32cm,
		inner sep=0.6pt, font=\scriptsize
		] (nc5) {5};
        %legend
        \node [blue, fill=white, text=white, minimum height=0.35cm, minimum width=0.52cm, line width=0.3pt, below=of nc1, yshift=+0.6cm, xshift=+0.7cm] (c1) {};
        \node [blue, fill=white, text=white, minimum height=0.35cm, minimum width=0.52cm, line width=0.3pt, right=of c1, xshift=-0.97cm] (c2) {};
        \node [blue, fill=white, text=white, minimum height=0.35cm, minimum width=0.52cm, line width=0.3pt, right=of c2, xshift=-0.97cm] (c3) {};
        \draw[-Stealth,
		blue,
		line width=0.4pt] ($(c1.west)+(0.08cm,0)$) -- ($(c3.east)+(-0.08cm,0)$);
        \node[draw=none, right=of c3, xshift=-1.1cm, blue]{\begin{minipage}[]{1.8cm}
				\centering
                \fontsize{9}{7}\selectfont  % 6pt字体，7pt行距（行距=字体大小+1pt）
				forward lane
		\end{minipage}};
        %%sequences
        \draw [-latex,line width=0.2mm] ([yshift=-7.5pt] nc5.west) -- ([yshift=-7.5pt, xshift=+5pt] nc1.east) node[above=of nc3, yshift=-0.9cm, draw=none, minimum height=0.2cm]
		{\begin{minipage}[]{1.8cm}
				\centering
                \fontsize{9}{7}\selectfont  % 6pt字体，7pt行距（行距=字体大小+1pt）
				upstream\\
				car sequence
		\end{minipage}};
        \draw [-latex,line width=0.2mm] ([yshift=-7.5pt, xshift=+103pt] nc2.west) -- ([yshift=-7.5pt, xshift=+103pt] nc1.east);
        \node [draw=none, minimum height=0.2cm, minimum width=0.52cm, line width=0.3pt, above=of n2, font=\fontsize{9}{7}\selectfont, yshift=-0.9cm] (c3) {(a) Mix bank (MB) };
    \end{tikzpicture}
    \begin{tikzpicture}[every node/.style={draw},
		rotate border/.style={shape border uses incircle, shape border rotate=#1}, transform shape,  % 允许整体变换
    scale=1.2]
    \matrix [draw=white,
		column sep=0.03cm, row sep=0.03cm]
		{
			%row1
			\node [blue, text=white, minimum height=0.64cm, minimum width=0.76cm, line width=0.3pt] (n1) {};
			\pgfmatrixnextcell \node [blue, text=white, minimum height=0.64cm, minimum width=0.76cm, line width=0.3pt] (n2) {};
            \pgfmatrixnextcell \node [blue, text=white, minimum height=0.64cm, minimum width=0.76cm, line width=0.3pt] (n3) {}; \\
			%row2
			\node [blue, text=white, minimum height=0.64cm, minimum width=0.76cm, line width=0.3pt] (n4) {};
			\pgfmatrixnextcell \node [blue, text=white, minimum height=0.64cm, minimum width=0.76cm, line width=0.3pt] (n5) {};
            \pgfmatrixnextcell \node [blue, text=white, minimum height=0.64cm, minimum width=0.76cm, line width=0.3pt] (n6) {}; \\
			%row3
			\node [red, text=white, minimum height=0.64cm, minimum width=0.76cm, line width=0.3pt] (n7) {};
			\pgfmatrixnextcell \node [red, fill=white, text=white, minimum height=0.64cm, minimum width=0.76cm, line width=0.3pt] (n8) {};
            \pgfmatrixnextcell \node [red, fill=white, text=white, minimum height=0.64cm, minimum width=0.76cm, line width=0.3pt] (n9) {}; \\
		};
        %%big rectangle
		\node[draw=black, line width=0.4pt, minimum height=1.85cm,  minimum width=2.5cm, rectangle, fill=none] (n01) at (0,0) {};
        \draw[-Stealth,
		blue,
		line width=0.4pt] ($(n1.west)+(0.08cm,0)$) -- ($(n3.east)+(-0.08cm,0)$);
        \draw[-Stealth,
		blue,
		line width=0.4pt] ($(n4.west)+(0.08cm,0)$) -- ($(n6.east)+(-0.08cm,0)$);
		\draw[-Stealth,
		red,
		line width=0.4pt] ($(n9.east)+(-0.08cm,0)$) -- ($(n7.west)+(0.08cm,0)$);
        \draw[->,
		gray,
		line width=0.4pt] ($(n7.north)+(-0.45cm,-0.3cm)$) -- ($(n1.south)+(-0.45cm,0.3cm)$);
        \draw[->,
		gray,
		line width=0.4pt] ($(n3.south)+(0.45cm,0.3cm)$) -- ($(n9.north)+(0.45cm,-0.3cm)$);
        %% cars
		\node[draw, circle, fill=gray!30, left=of n7, xshift=+0.45cm,
        minimum size=0.32cm,
		inner sep=0.6pt, font=\scriptsize
		] (nc1) {1};
        \node[above=of nc1, draw=none, yshift=-1cm, font=\fontsize{9}{7}\selectfont] {$1$};
		\node[draw, circle, fill=gray!30, left=of nc1, xshift=+0.95cm, minimum size=0.32cm,
		inner sep=0.6pt, font=\scriptsize
		] (nc2) {2};
        \node[above=of nc2, draw=none, yshift=-1cm, font=\fontsize{9}{7}\selectfont] {$2$};
        \node[draw, circle, fill=gray!30, left=of nc2, xshift=+0.95cm, minimum size=0.32cm,
		inner sep=0.6pt, font=\scriptsize
		] (nc3) {3};
        \node[above=of nc3, draw=none, yshift=-1cm, font=\fontsize{9}{7}\selectfont] {$3$};
        \node[draw, circle, fill=gray!30, left=of nc3, xshift=+0.95cm, minimum size=0.32cm,
		inner sep=0.6pt, font=\scriptsize
		] (nc4) {4};
        \node[above=of nc4, draw=none, yshift=-1cm, font=\fontsize{9}{7}\selectfont] {$4$};
        \node[draw, circle, fill=gray!30, left=of nc4, xshift=+0.95cm, minimum size=0.32cm,
		inner sep=0.6pt, font=\scriptsize
		] (nc5) {5};
        \node[above=of nc5, draw=none, yshift=-1cm, font=\fontsize{9}{7}\selectfont] {$5$};
        \draw [-latex,line width=0.2mm] ([yshift=-7.5pt] nc5.west) -- ([yshift=-7.5pt, xshift=+5pt] nc1.east) node[above=of nc3, yshift=-0.5cm, draw=none, minimum height=0.2cm]
		{\begin{minipage}[]{1.8cm}
				\centering
                \fontsize{9}{7}\selectfont  % 6pt字体，7pt行距（行距=字体大小+1pt）
				upstream\\
				car sequence
		\end{minipage}};
        %% cars
		\node[draw, circle, fill=gray!30, right=of n9, xshift=-0.45cm,
        minimum size=0.32cm,
		inner sep=0.6pt, font=\scriptsize
		] (nc11) {1};
        \node[above=of nc11, draw=none, yshift=-1cm, font=\fontsize{9}{7}\selectfont] {$\pi_5$};
		\node[draw, circle, fill=gray!30, right=of nc11, xshift=-0.95cm, minimum size=0.32cm,
		inner sep=0.6pt, font=\scriptsize
		] (nc21) {5};
        \node[above=of nc21, draw=none, yshift=-1cm, font=\fontsize{9}{7}\selectfont] {$\pi_4$};
        \node[draw, circle, fill=gray!30, right=of nc21, xshift=-0.95cm, minimum size=0.32cm,
		inner sep=0.6pt, font=\scriptsize
		] (nc31) {2};
        \node[above=of nc31, draw=none, yshift=-1cm, font=\fontsize{9}{7}\selectfont] {$\pi_3$};
        \node[draw, circle, fill=gray!30, right=of nc31, xshift=-0.95cm, minimum size=0.32cm,
		inner sep=0.6pt, font=\scriptsize
		] (nc41) {3};
        \node[above=of nc41, draw=none, yshift=-1cm, font=\fontsize{9}{7}\selectfont] {$\pi_2$};
        \node[draw, circle, fill=gray!30, right=of nc41, xshift=-0.95cm, minimum size=0.32cm,
		inner sep=0.6pt, font=\scriptsize
		] (nc51) {4};
        \node[above=of nc51, draw=none, yshift=-1cm, font=\fontsize{9}{7}\selectfont] {$\pi_1$};
        \draw [-latex,line width=0.2mm] ([yshift=-7.5pt, xshift=-5pt] nc11.west) -- ([yshift=-7.5pt, xshift=+2pt] nc51.east) node[above=of nc31, yshift=-0.5cm, draw=none, minimum height=0.2cm]
		{\begin{minipage}[]{1.8cm}
				\centering
                \fontsize{9}{7}\selectfont  % 6pt字体，7pt行距（行距=字体大小+1pt）
				downstream\\
				car sequence
		\end{minipage}};
        \node[draw, dashed, circle, fill=gray!30, above=of n8,
        yshift=-1.42cm,
        minimum size=0.32cm,
		inner sep=0.6pt, font=\scriptsize
		] (nc1-1) {1};
        \node[draw, dashed, circle, fill=gray!30, above=of n3,
        yshift=-1.42cm,
        minimum size=0.32cm,
		inner sep=0.6pt, font=\scriptsize
		] (nc2-1) {2};
        \node[draw, dashed, circle, fill=gray!30, above=of n2,
        yshift=-1.42cm,
        minimum size=0.32cm,
		inner sep=0.6pt, font=\scriptsize
		] (nc3-1) {3};
        \node[draw, dashed, circle, fill=gray!30, above=of n5,
        yshift=-1.42cm,
        minimum size=0.32cm,
		inner sep=0.6pt, font=\scriptsize
		] (nc4-1) {4};
        \node[draw, dashed, circle, fill=gray!30, above=of n4,
        yshift=-1.42cm,
        minimum size=0.32cm,
		inner sep=0.6pt, font=\scriptsize
		] (nc5-1) {5};
        %legend
        \node [red, fill=white, text=white, minimum height=0.35cm, minimum width=0.52cm, line width=0.3pt, below=of nc2, yshift=+0.6cm] (c1) {};
        \node [red, fill=white, text=white, minimum height=0.35cm, minimum width=0.52cm, line width=0.3pt, right=of c1, xshift=-0.97cm] (c2) {};
        \node [red, fill=white, text=white, minimum height=0.35cm, minimum width=0.52cm, line width=0.3pt, right=of c2, xshift=-0.97cm] (c3) {};
        \draw[-Stealth,
		red,
		line width=0.4pt] ($(c3.east)+(-0.08cm,0)$) -- ($(c1.west)+(0.08cm,0)$);
        \node[draw=none, right=of c3, xshift=-1.1cm, red]{\begin{minipage}[]{2.3cm}
				\centering
                \fontsize{9}{7}\selectfont  % 6pt字体，7pt行距（行距=字体大小+1pt）
				return lane (RL)
		\end{minipage}};
        \node [draw=none, minimum height=0.2cm, minimum width=0.52cm, line width=0.3pt, above=of n2, font=\fontsize{9}{7}\selectfont, yshift=-0.9cm] (c3) {(b) Mix bank with a return lane (MB-RL)};
    \end{tikzpicture}

    \caption{Illustration of the MB and MB-RL of the same size with 3 lanes and a lane capacity of 3. \label{fig1}}

\end{figure}
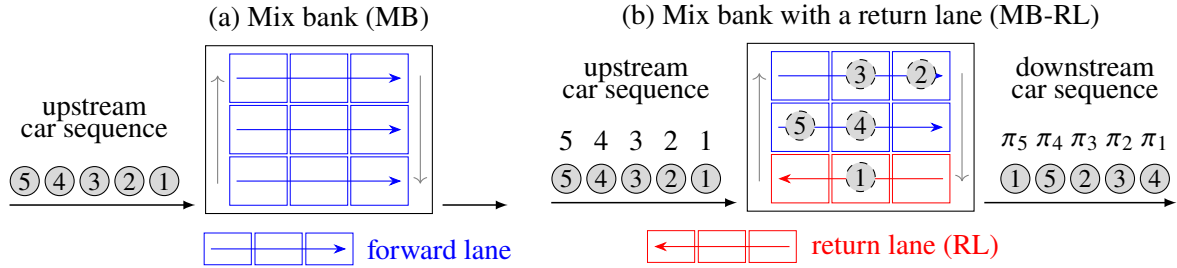

%a new sequence that minimizes production costs in the subsequent shop.

The \textit{car resequencing problem} (CRSP) typically addresses how to rearrange a car sequence from one shop into a new sequence that minimizes the production costs in the next shop. Resequencing can be \textit{virtual}, in which preset models are reassigned to car bodies without altering their relative positions. Alternatively, \textit{physical resequencing} uses a buffer between two adjacent shops, identified as the upstream and downstream shops. A car from the upstream shop can be temporarily stored in the buffer, allowing its successors to proceed downstream first. The \textit{mix bank} (MB) is the most common buffer due to its cost efficiency \citep{taube2018resequencing}. As shown in Figure \ref{fig1}(a), it consists of several forward lanes, with cars entering and exiting on a \textit{first-in-first-out} (FIFO) basis. However, the resequencing capability of an MB is strictly limited by its size and FIFO rules. For example, an MB with 3 lanes, each able to hold 3 cars, cannot transform the 5-car upstream sequence $[1,2,3,4,5]$ (led by car 1) into the desired downstream sequence $[4,3,2,5,1]$ (led by car 4).

%%The resulting \textit{car resequencing problem} (CRSP) typically studies how to rearrange a car sequence in one shop into a new one that minimizes production costs in the next shop. This resequencing can be \textit{virtual}, achieved by reassigning preset models to car bodies without altering their relative positions, or \textit{physical}, utilizing a buffer between two adjacent shops \citep{boysen2012resequencing}, called the upstream and downstream shops. In the latter, a car leaving the upstream shop can be temporarily in the buffer, allowing its successors to proceed downstream first. The \textit{mix bank} (MB) is the most common buffer due to its cost efficiency \citep{taube2018resequencing}. As shown in Figure \ref{fig1}(a), it features several forward lanes where cars enter and exit on a \textit{first-in-first-out} (FIFO) basis. However, its resequencing capability is strictly limited by the buffer size and FIFO rules. For example, an MB with 3 lanes each holding 3 cars cannot change an upstream sequence $[1,2,3,4,5]$ starting with car 1 into a desired downstream sequence $[4,3,2,5,1]$ starting with car 4.
%with an upstream sequence T1 starting with car 1, the desired downstream sequence T2 starting with car 4 cannot be achieved using an MB with three forward lanes of capacity two.

To improve the flexibility of the MB, carmakers tend to adopt a new buffer: the \textit{mix bank with a return lane} (MB-RL), which is examined in this work. This buffer extends the MB with a \textit{return lane} (RL), where cars move against those in the forward lanes. Through the RL, the first car in any forward lane can move back to the buffer entrance for reassignment to a forward lane. With this loop structure, cars can switch among forward lanes. For illustration, Figure \ref{fig1}(b) shows an MB-RL of the same size as the MB in Figure \ref{fig1}(a). By requiring cars 1 and 2 to enter the RL and later re-enter the forward lanes, this MB-RL can achieve the sequence change given in the previous example.
%By requiring only car 3 to enter the RL and later re-enter a forward lane, this MB-RL achieves the sequence change in the previous example.

%%%%%To improve the flexibility of the MB, carmakers now adopt a new buffer: the \textit{mix bank with a return lane} (MB-RL), which is examined in this work. This buffer extends the MB using a \textit{return lane} (RL) where cars move against those in the forward lanes. The first car in any forward lane can enter the RL, move back toward the upstream shop, and then be reassigned to another forward lane at the buffer entrance. With this loop structure, cars can switch among forward lanes. An MB-RL of the same size as the MB in Figure \ref{fig1}(a) is illustrated in Figure \ref{fig1}(b). It enables the sequence change in the previous example by requiring car 3 to enter the RL and later re-enter a forward lane.

%The first car in any forward lane can enter the RL. Cars returning via the RL are then reassigned to forward lanes at the MB-RL entrance. As a result, cars can loop between the RL and forward lanes, allowing them to re-enter forward lanes repeatedly. The later the cars enter forward lanes, the farther downstream positions they can occupy. Figure \ref{fig1}(b) shows an MB-RL of the same size as the MB in Figure \ref{fig1}(a). It enables the sequence change in the previous example by only requiring car 3 to enter the RL and re-enter a forward lane at a later time.

The operation of the MB-RL poses a critical challenge: minimizing the \textit{resequencing makespan}, defined as the total time needed to transform the upstream sequence into the downstream one. By circulating, any order of the cars within a loop formed by a forward lane and the RL can be achieved. Nevertheless, allowing cars to loop blindly may cause production delays for the downstream shop, especially damaging a make-to-order production mode with strict due dates \citep{boysen2009sequencing}.

%Through a circulating motion,
%%%%%By circulating, any order of cars within a loop formed by a forward lane and the RL can be achieved. Nevertheless, allowing cars to loop blindly may cause production delays for the downstream shop, especially damaging a make-to-order production mode with strict due dates \citep{boysen2009sequencing}. Therefore, a key challenge in operating the MB-RL is to reduce the \textit{resequencing time}, defined as the total time needed to transform the upstream sequence into the downstream one.

%%%%%The operation of the MB-RL presents a critical challenge: minimizing \textit{resequencing makespan}, which is the total time taken in the MB-RL to adjust the upstream car sequence to the downstream one. Prolonged time spent on resequencing poses the risk of production delays for the downstream shop, especially in a make-to-order production mode with strict due dates \citep{boysen2009sequencing}. It is impractical to constantly transport cars in loops within the MB-RL to attain the desired downstream sequence.

In this work, we focus on improving resequencing efficiency when using the MB-RL to rearrange the cars produced in a one-hour shift. The problem addressed is called the \textit{car resequencing problem with makespan minimization} (CRSP-MS). Unlike the traditional CRSP, where the car sequence is the decision to be made, the CRSP-MS treats both upstream and downstream sequences as fixed, allowing exclusive attention to the dynamic scheduling of cars. %allowing exclusive attention to dynamic car scheduling.
By determining the position of each car over time, the objective of the CRSP-MS is to minimize the \textit{resequencing makespan}.

%%%%%%%In this work, we focus on the resequencing efficiency of the CRSP with an MB-RL. The new version of the CRSP addressed is referred to as the \textit{car resequencing problem with makespan minimization} (CRSP-MS). By determining the position of each car at each timestamp, the goal of the CRSP-MS is to minimize the \textit{resequencing makespan} needed for the MB-RL to change the upstream sequence into a given downstream one for cars produced in a one-hour shift. Since previous studies on the CRSP typically decouple sequence determination from buffer scheduling \citep{guo2025logic}, we exclude sequence determination from CRSP-MS and treat both sequences as known to focus exclusively on the dynamic scheduling process in the MB-RL.

Beyond formally defining the CRSP-MS, this work focuses on its exact solution. Using a \textit{time-space network} (TSN) representation, we develop a two-stage \textit{branch-and-price approach integrated with constraint programming} (BP-CP) to increase the number of cars that can be rearranged within a practical production batch (e.g., 50 cars, as in \citet{wu2021mathematical} and \citet{guo2025logic}). In the first stage, %For the \textit{column generation} (CG) in the first stage,
three versions of the \textit{master problem} (MP) that differ in lower-bound tightness are proposed. Each MP has a counterpart \textit{shortest-path pricing subproblem} defined on a specific TSN graph to generate new columns. After \textit{column generation} (CG) converges, integrality is enforced only on the makespan and the number of times each car enters the RL. When the related decision variables are integers, we solve a \textit{constraint programming} (CP) model in the second stage to find a feasible CRSP-MS solution. If the CP model proves infeasible, a feasibility cut is added to the MP, and CG restarts. The algorithm is enhanced with a valid lower bound and a preprocessing technique. Computational results demonstrate the efficiency of BP-CP, which manages up to 120 cars within one hour. The comparison with the traditional MB confirms the superior resequencing capability and the comparable efficiency of the MB-RL across various resequencing difficulties.

Our contribution is threefold. First, we formalize the problem of minimizing the \textit{resequencing makespan} when using the emerging buffer MB-RL to rearrange cars. The introduced CRSP-MS addresses the practical challenges faced by carmakers in improving the operational efficiency of the MB-RL \citep{moon2005simulation, yu2023analysis}, and we establish its computational complexity. Methodologically, then, our BP-CP framework efficiently handles large instances and serves as a reference for applying \textit{branch-and-price} (BP) to highly symmetric problems that require exhaustive branching. Although tailored for scheduling cars in the MB-RL, it also applies to analogous systems, such as (re)humping operations in railway shunting yards \citep{boysen2012shunting} and loop sorters in warehouses \citep{boysen2024order}. Managerially, we provide the first known numerical comparison between the MB-RL and the MB. The findings show that adding an RL significantly enhances resequencing capability while preserving makespan performance, enabling managers to achieve higher throughput with a smaller buffer footprint and support cost-effective buffer design decisions.

%Our contribution is threefold. First, we formalize the problem of minimizing the \textit{resequencing makespan} in the context of using the emerging buffer MB-RL to rearrange cars. The introduced CRPS-MS extends the classical CRSP, tackling the practical challenges carmakers face in enhancing the operational efficiency of the MB-RL \citep{moon2005simulation, yu2023analysis}, and its computational complexity is established.  Methodologically, then, our BP-CP framework offers a reference for applying BP to highly symmetric problems that require exhaustive branching. Although tailored for scheduling cars in the MB-RL, it also applies to analogous systems, such as (re)humping operations in railway shunting yards \citep{boysen2012shunting}. Finally, we present the first known numerical comparison between the MB-RL and the MB regarding resequencing capability and efficiency.

The rest of this paper is organized as follows. Section \ref{sec2} reviews related studies on the physical CRSP. Section \ref{sec3} formally defines the CRSP-MS and introduces its TSN representation. Section \ref{sec4} describes the proposed BP-CP approach, including its fundamental components, enhancements, and overall implementation. Section \ref{sec5} presents the computational performance of BP-CP and compares the MB-RL with the MB to derive managerial insights into its operational benefits. Section \ref{sec6} finally concludes this work.

%%The rest of this paper is organized as follows. Section \ref{sec2} reviews related studies on the physical CRSP. Section \ref{sec3} formally defines the CRSP-MS and introduces its TSN representation. Section \ref{sec4} describes the proposed BP-CP approach, including its fundamental components, enhancements, and overall implementation. The computational performance of BP-CP and a comparison between the MB-RL and the MB are presented in Section \ref{sec5}. Section \ref{sec6} finally concludes this work.

%%%%%%%The remainder of this paper is organized as follows. Section \ref{sec2} reviews related studies on the physical CRSP. The CRSP-MS is formally defined in Section \ref{sec3-1}, followed by the methods for checking its feasibility in Section \ref{sec3-2} and a compact formulation in Section \ref{sec3-3}. Section \ref{sec4} details our three-stage BP-CP approach, including the formulations for each stage, algorithmic enhancements, and overall implementation. The computational analysis, which refers to the proposed methods, the impact of buffer size, as well as a comparison between MB and MB-RL, is presented in Section \ref{sec5}. Finally, Section \ref{sec6} concludes this work.

\section{Literature Review}
\label{sec2}
An extensive survey of the CRSP is provided by \citet{boysen2012resequencing}. Since \textit{virtual resequencing} operates without physical buffers, as previously mentioned, our review focuses on research related to \textit{physical resequencing}, particularly where the MB or the MB-RL is used.

\textit{Physical resequencing} commonly uses three types of buffers: the \textit{automated storage and retrieval system} (AS/RS) \citep{lubben2023online}, the \textit{pull-off table} \citep{boysen2011car}, and the MB. Both the AS/RS and the \textit{pull-off table} provide independently accessible cells that temporarily store cars removed from the production line. While one car is offline, its successors keep moving forward until another offline car is ready for reinsertion. At equal capacity, the two buffers offer the same resequencing capability. The MB is the most widely adopted of the three due to its simple operation and lower cost \citep{taube2018resequencing%,sun2024integrating
}. However, its flexibility is constrained not only by its size but also by FIFO rules, which fix the exit order for cars in the same lane.
%fix/control/set

%Three typical buffers are used in \textit{physical resequencing}: the \textit{automated storage and retrieval system} (AS/RS), the \textit{pull-off table}, and the MB. An AS/RS features hundreds of cells to store cars, with each cell accessible individually \citep{lubben2023online}. The \textit{pull-off table} allows cars to be temporarily taken off-line from the MMPL. The successors behind an off-line car proceed to move forward until an off-line car is later reinserted into the final sequence \citep{boysen2011car}. The buffer size fundamentally limits the resequencing capability of both the AS/RS and the \textit{pull-off table}. Given the same number of storage cells, both buffers exhibit equal flexibility. The MB is used the most among the three buffers due to its operational simplicity and relatively low construction costs \citep{taube2018resequencing%,sun2024integrating
%}.
%However, its flexibility is restricted by both buffer size and FIFO rules. When retrieving cars from forward lanes, adherence to FIFO rules affects the precedence in which cars exit the MB.

%Compared to the MB, the MB-RL is more flexible but requires more sophisticated operations due to the permission to reassign cars to forward lanes, which limits the research focus on it.

Compared to the MB, the MB-RL is more flexible but requires more sophisticated operations due to the permission for cars to enter forward lanes multiple times, which limits research on it. To our knowledge, \citet{moon2005simulation} are the first to study the CRSP with an MB-RL at a Korean carmaker through simulations. For various buffer sizes, they apply simple rules to rearrange cars to reduce color changes in the obtained sequence. Later studies \citep{boysen2012resequencing, bysko2020automotive} recognize the flexibility of the RL as a future research topic. \citet{yu2023analysis} recently address the MB-RL in a multi-objective CRSP and enhance its operational efficiency by reducing RL usage.
%To our knowledge, \citet{moon2005simulation} first study the CRSP with an MB-RL through simulations at a Korean carmaker. For various buffer sizes, they apply simple rules to reduce color changes in the car sequence. Later studies \citep{boysen2012resequencing, bysko2020automotive} recognize the flexibility offered by the RL as a future research topic. \citet{yu2023analysis} recently address the MB-RL in a multi-objective CRSP and enhance its operational efficiency by reducing the overall use of the RL.

%%%%%%%%Compared to the MB, the MB-RL offers stronger resequencing capability. Nevertheless, its more complex internal scheduling, which results from cars re-entering forward lanes via the RL, has limited its study in the literature. To our knowledge, \citet{moon2005simulation} are the first to address the CRSP with an MB-RL. In their simulation study at a Korean carmaker, simple rules are used to minimize color changes across various MB-RL sizes. Later studies \citep{boysen2012resequencing, bysko2020automotive} regard the additional resequencing flexibility offered by the RL as a future research topic. More recently, \citet{yu2023analysis} employ the MB-RL for a multi-objective CRSP, where the MB-RL operational efficiency, as one optimization goal, is explicitly considered by minimizing the total number of RL usages.

Regarding the optimization objective, improving resequencing efficiency in buffers has received little attention, despite its importance for JIT production. Classical CRSP studies focus on minimizing production costs in the paint or assembly shop. In the paint shop, high cleaning costs for spray nozzles make it crucial to build large blocks of cars in the same color \citep{sun2017study}. For the assembly shop, the purpose of resequencing is to reduce work overload at stations. Only a few multi-objective CRSP models consider resequencing efficiency as a goal. Besides the aforementioned work by \citet{yu2023analysis}, \citet{leng2023multi} improve upon it by reducing due-date violations. This work is the first to specifically optimize resequencing efficiency in the MB-RL.
%Resequencing efficiency has only been considered a goal in a few multi-objective CRSP models.

%%%In the paint shop, high cleaning costs for spray nozzles make it essential to build large blocks of cars in the same color \citep{sun2017study}. In the assembly shop, resequencing aims to reduce work overload by minimizing violations of specific sequencing rules or smoothing option consumption.

%The CRSP in this context is also known as the color batching problem.

%%%%%%%%%%%%%%%Despite extensive research on physical CRSP, the resequencing efficiency of buffers has received little attention. Traditionally, the CRSP aims to minimize production costs by rearranging cars before they enter the paint shop or assembly shop. In the paint shop, high costs caused by cleaning spraying nozzles make color-change minimization, or building large blocks of the same color cars, a primary goal. This version of the CRSP is also known as the \textit{color batching problem} \citep{sun2017study}. In the assembly shop, resequencing aims to reduce work overload by minimizing violations of specific sequencing rules or to keep a steady consumption of options. Resequencing efficiency has only emerged as an optimization goal in a few multi-objective CRSP models. Beyond the aforementioned work by \citet{yu2023analysis}, which minimizes RL usage, \citet{leng2023multi} improve resequencing efficiency of the MB by reducing violations of due dates.

Existing solution methods for the CRSP with an MB or MB-RL typically assume that the buffer is large enough to hold all cars, allowing the problem to be decomposed into two static subproblems. The \textit{fill subproblem} assigns cars to lanes, and the \textit{release subproblem} determines the order in which they exit \citep{boysen2013decomposition}. Most studies focus on heuristic methods, whereas exact approaches are limited and usually address only one subproblem. For example, \citet{elahi2015optimizing} propose a multi-commodity network flow model that can solve the fill subproblem exactly for 100 cars in real time. \citet{ko2016paint} formulate the release problem as a traveling salesman problem. Using dynamic programming, they find the optimal sequence for 20 cars within an hour. This scale of solvable instances is later expanded to 56 by \citet{hong2018accelerated} with tighter bounds. In a CRSP designed to minimize the total costs across multiple shops, \citet{guo2025logic} discard the buffer-size assumption and develop a three-level logic-based Benders decomposition approach. By determining optimal upstream and downstream sequences at the first two levels and assigning cars to lanes at the third level, their method enables the resequencing of 120 cars per batch.

Among the few methods designed for the CRSP with an MB-RL, \citet{moon2005simulation} develop rule-based storage and retrieval algorithms. \citet{yu2023analysis} propose an integer programming model solved by a state-transition algorithm. To our knowledge, no exact approaches have been reported.

%%Following the problem setting of Guo et al.,
%Following \citet{guo2025logic}, our CRSP-MS does not limit the number of cars to be rearranged by the MB-RL size. The optimal upstream and downstream sequences are predefined, and the problem focuses on scheduling cars to enter and exit the MB-RL efficiently. To address the open question of quantifying the role of the RL in resequencing performance \citep{boysen2012resequencing}, we develop an exact solution method and conduct a comparative analysis between the MB-RL and the MB.
%%pending question

Following \citet{guo2025logic}, our CRSP-MS does not limit the number of cars to be rearranged by the MB-RL size. The optimal upstream and downstream sequences are predefined, and the problem focuses on efficiently scheduling cars to enter and exit the MB-RL. To address the open question of quantifying the value of the RL in resequencing \citep{boysen2012resequencing}, we develop an exact solution method and conduct a comparative analysis between the MB-RL and the MB. The results offer managers evidence-based insights into the advantages of resequencing with the MB-RL.

\section{The Car Resequencing Problem with Makespan Minimization}\label{sec3}

\subsection{Problem Definition}\label{sec3-1}
Consider two adjacent shops connected by an MB-RL and a sequence of $d$ cars to be resequenced, indexed by the set $V=\{1,2,\dots,d\}$. Each car $k\in V$ has a designated position $p_k\in P=\{1,2,\dots,d\}$ in the downstream shop. The upstream sequence $\Pi^U=[1,2,\dots,d]$ defines the order in which these cars enter the MB-RL, and the downstream sequence $\Pi^D=[\pi_1,\dots,\pi_p,\dots,\pi_d]$ specifies the order in which they exit the MB-RL to the downstream shop, where $\pi_p\in V$ denotes the upstream index of the car assigned to downstream position $p\in P$. The CRSP-MS involves transforming $\Pi^U$ into $\Pi^D$ via the MB-RL in the shortest possible time. %Consider two adjacent shops connected by an MB-RL, each with a predefined sequence of $d=|V|=|\{1,2,\dots,d\}|$ cars. The upstream sequence $\Pi^U=[1,2,\dots,d]$ defines the order in which the $d$ cars enter the MB-RL from the upstream shop, and the downstream sequence $\Pi^D=[\pi_1,\dots,\pi_p,\dots,\pi_d]$ specifies the order in which they exit the MB-RL to the downstream shop, where $\pi_p\in V$ indicates the upstream index of the car assigned to downstream position $p\in P=\{1,2,\dots,d\}$. The CRSP-MS involves transforming $\Pi^U$ into $\Pi^D$ via the MB-RL in the shortest possible time.
%seeks to transform

%%%%%Consider a set of $d=|V|=|\{1,2,\dots,d\}|$ cars with two given sequences and an MB-RL connecting two adjacent shops. The upstream sequence $\Pi^U=[1,2,\dots,d]$ indicates the order in which the cars exit the upstream shop $s$. The downstream sequence $\Pi^D=[\pi_1,\pi_2,\dots,\pi_p,\dots,\pi_d]$ specifies their positions in the downstream shop $e$, where each $\pi_p$ is the index of the upstream car assigned to downstream position $p\in P=\{1,2,\dots,d\}$. The CRSP-MS involves transforming $\Pi^U$ into $\Pi^D$ via the MB-RL in the shortest possible time.

\begin{figure}[htbp]
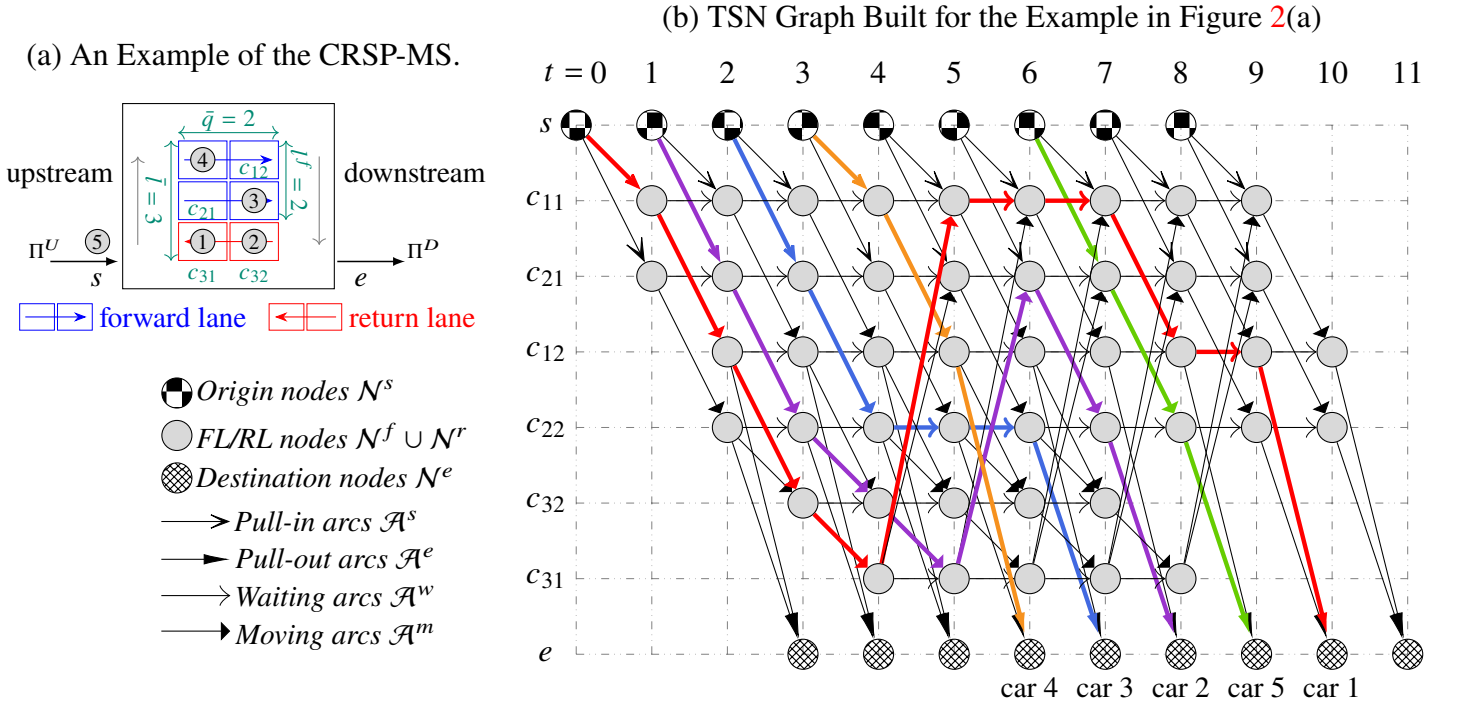

    \centering
\makebox[\textwidth][c]{%

}% end makebox

    \caption{An Example of the CRSP-MS and Its TSN Graph Representation. \label{fig2}}

\end{figure}

As shown in Figure \ref{fig2}(a), the MB-RL consists of $\bar{l}$ lanes: one RL indexed by $\bar{l}$ and $l^f=\bar{l}-1$ lanes. Let $L^f=\{1,2,\dots,l^f\}$ denote the set of forward lanes, and $L=L^f\cup \{\bar{l}\}$ be that of all lanes. Each lane has $\bar{q}$ cells, and each cell $q\in Q=\{1,2,\dots,\bar{q}\}$ can hold one car. %As shown in Figure \ref{fig2}(a), \textcolor{red}{the MB-RL consists of $\bar{l}$ lanes: one RL numbered $\bar{l}$ and $l^f=\bar{l}-1$ forward lanes. Each lane in $L=L^f\cup \{\bar{l}\}=\{1,2,\dots,l^f\}\cup \{\bar{l}\}$ has a capacity of $\bar{q}$ cells, and each cell $q\in Q=\{1,2,\dots,\bar{q}\}$ can hold one car.} %As shown in Figure \ref{fig2}(a), the MB-RL comprises $\bar{l}=|L|=|\{1,2,\dots,\bar{l}\}|$ lanes: one RL numbered $\bar{l}$ and $l^f=|L^f|=|\{1,2,\dots,\bar{l}-1\}|$ forward lanes. Each lane has $\bar{q}=|Q|=|\{1,2,\dots,\bar{q}\}|$ cells, and each cell can hold one car.
The internal conveyor belt in each lane moves cars from cell to cell at a constant speed, with the RL and forward lanes operating in opposite directions. At each moment, a transfer car at the MB-RL entrance can dispatch one car from either the upstream shop or the RL into a forward lane, and another at the MB-RL exit can retrieve the first car from any forward lane to send it to either the downstream shop or the RL. The two transfer cars operate independently, allowing one car to enter a forward lane while another leaves simultaneously. A car may enter the RL multiple times. The number of cars in any lane must never exceed the capacity $\bar{q}$.

%In the direction from upstream to downstream (i.e., from left to right),
Numbering from left to right, we denote the $q$-th ($\forall q\in Q$) cell of lane $l\in L$ as $c_{lq}\in C=\{c_{lq}\mid l\in L, q\in Q\}$ and define $S=C\cup \{s,e\}$ as the set of all locations where cars may stay, with $s$ and $e$ representing the upstream and downstream shops. Recall that car 1 is the first car in $\Pi^{U}$, and car $\pi_{d}$ is the last in $\Pi^{D}$. The resequencing process starts when car 1 departs from the upstream shop and ends when car $\pi_{d}$ reaches the downstream shop. Let $\tau$ denote the maximum duration of this process. We discretize it into $\tau$ time units (TUs), forming the set $T=\{0,1,\dots,\tau\}$ of discrete timestamps. Each of the following five movements is assumed to consume one TU: \textup{(i)} from $s$ to $c_{l1}$ of forward lane $l\in L^f$, \textup{(ii)} from $c_{\bar{l}1}$ of the RL to $c_{l1}$ of forward lane $l\in L^f$, (iii) from $c_{l\bar{q}}$ of forward lane $l\in L^f$ to $e$, \textup{(iv)} from $c_{l\bar{q}}$ of forward lane $l\in L^f$ to $c_{\bar{l}\bar{q}}$ of the RL, \textup{(v)} from $c_{lq}$ to $c_{l,q+1}$ ($q\in Q\backslash \{\bar{q}\}$) of forward lane $l\in L^f$ or from $c_{\bar{l}q}$ to $c_{\bar{l},q-1}$ ($q\in Q\backslash\{1\}$) of the RL. %\textup{(i)}/\textup{(ii)} from the upstream shop/the leftmost cell of the RL to the leftmost cell of a forward lane, \textup{(iii)}/\textup{(iv)} from the rightmost cell of a forward lane to the downstream shop/the rightmost cell of the RL, and \textup{(v)} between adjacent cells within the same lane.
The goal of the CRSP-MS is to minimize the \textit{resequencing makespan}, i.e., the total time required to transport each car $k\in V$ to its designated downstream position $p_k\in P$, by determining the location of each car at each timestamp in $T$. The computational complexity of the CRSP-MS is established in the following theorem (see the online supplement for the proof).

%%%Recall that car 1 is the first car in $\Pi^{U}$, and car $\pi_{d}$ is the last car in $\Pi^{D}$. The resequencing process starts when car 1 departs from the upstream shop and ends when car $\pi_{d}$ reaches the downstream shop. Let $\tau$ represent the maximum duration of this process. We discretize it into $\tau$ time units (TUs), forming the set $T=\{0,1,\dots,\tau\}$ of discrete timestamps. Each of the following five movements is assumed to consume one TU: \textup{(i)}/\textup{(ii)} from the upstream shop/the leftmost cell of the RL to the leftmost cell of a forward lane, \textup{(iii)}/\textup{(iv)} from the rightmost cell of a forward lane to the downstream shop/the rightmost cell of the RL, and \textup{(v)} between adjacent cells within the same lane. The objective of the CRSP-MS is to minimize the \textit{resequencing makespan}, i.e., the total time required to transport each car $k\in V$ to its designated downstream position $p_k\in P$, by determining the position of each car at each timestamp in $T$. The computational complexity of the problem is established in the following theorem (see online supplement for the proof).

%The CRSP-MS is NP-hard in the strong sense, even if no car enters the RL.
\begin{theorem}\label{theo1}
    The CRSP-MS is NP-complete in the strong sense, even if no car enters the RL.
\end{theorem}
%%%%%\begin{proof}
    %We prove this theorem by reduction from an unrestricted 3-Partition problem.
%%%%%    See \ref{app1}.
%%%%%\end{proof}

To quantify the difficulty of changing the sequence from $\Pi^{U}$ into $\Pi^{D}$, we introduce a \textit{resequencing complexity index}, $\sigma$, based on the concept of \textit{conflicting} cars from \citet{guo2025logic}. Two cars $k_1$ and $k_2$ in $V$ are considered \textit{conflicting} if their order in $\Pi^{U}$ is reversed to that in $\Pi^{D}$, i.e., $k_1<k_2$ and $p_1>p_2$, or $k_1>k_2$ and $p_1<p_2$, $p_1$, $p_2\in P$. On this basis, the index $\sigma$ is defined as follows:
%Then, we define the index o1 as follows:

\begin{definition}\label{def1}
    %Let $\sigma_k$ be the number of cars in $\{k+1,k+2,\dots,d\}$ that conflict with car $k\in V$. The \textit{resequencing complexity index} $\sigma$ for transforming $\Pi^U$ into $\Pi^P$ is given by $\sigma=\sum_{k\in V}{\sigma_k}$.
    Let $\sigma_k$ be the number of \textit{conflicting} cars in $\{k+1,k+2,\dots,d\}$ with car $k\in V$. The \textit{resequencing complexity index} $\sigma$ for transforming $\Pi^U$ into $\Pi^P$ is given by $\sigma=\sum_{k\in V}{\sigma_k}$.
\end{definition}

%%%%As cars enter the downstream shop one after another, any delay in the entry of one car spreads to those behind it in $\Pi^D$, and then prolongs the makespan. Based on this delay propagation, Proposition \ref{pro2} provides a straightforward sufficient condition under which a car does not enter the RL in the optimal solution to the CRSP-MS.
As cars enter the downstream shop one after another, any delay in entry spreads to subsequent cars in $\Pi^D$ and then prolongs the makespan. Proposition \ref{pro2} provides a sufficient condition for a car not to enter the RL in an optimal CRSP-MS solution (see online supplement for the proof).
%Given this delay propagation, Proposition \ref{pro2} provides a straightforward sufficient condition \textcolor{red}{for a car not to enter the RL in an optimal CRSP-MS solution.}

%%%%%As cars arrive at the downstream shop sequentially, any delay in the arrival of one car spreads to its successors in $\Pi^D$, thereby prolonging the makespan. Based on this propagation of delays, Proposition \ref{pro1} offers a sufficient condition for cars not to enter the RL in the optimal solution.

\begin{proposition}\label{pro2}
    %%%%For any car $\pi_p$ in the downstream sequence $\Pi^D=[\pi_1,\pi_2,\dots,\pi_d]$, let $\pi(p)$ denote the maximum index of cars in the subsequence $[\pi_1,\pi_2,\dots,\pi_p]$, i.e., $\pi(p)=\max\{\pi_1,\pi_2,\dots,\pi_p\}$. If $\pi_p=\pi(p)$, then car $\pi_p$ does not enter the RL in the optimal solution.
    %%Consider car k1 in relation to the downstream sequence.
    %%%%%%%%%%%%%%%Refer to the downstream sequence $\Pi^{D}$ to consider car $\pi_{p}\in V$ ($p\in P$). Let $\pi(p)$ the maximum upstream index of cars in the downstream subsequence %be the car with the maximum upstream index in the downstream subsequence
    %%%%%%%%%%%%%%%$[\pi_1,\pi_2,\dots,\pi_p]$, i.e., $\pi(p)=\max\{\pi_1,\pi_2,\dots,\pi_p\}$. If $\pi_p=\pi(p)$, then car $\pi_p$ does not enter the RL in the optimal CRPS-MS solution and is called RL-unused.
    Consider car $\pi_{p}\in V$ in the downstream sequence $\Pi^{D}=[\pi_1,\dots,\pi_p,\dots,\pi_{d}]$. If $\pi_p$ has the largest upstream index among the first $p\in P$ cars in $[\pi_1,\pi_2,\dots,\pi_p]$, i.e., $\pi_p=\max\{\pi_1,\pi_2,\dots,\pi_p\}$, then there exists an optimal solution to the CRSP-MS in which car $\pi_p$ never enters the RL and is said to be RL-unused.
\end{proposition}

\noindent Cars 4 and 5 in Figure \ref{fig1}(b) are RL-unused because $4=\pi_1=\max \{4\}$ and $5=\pi_4=\max\{4,3,2,5\}$. Let $V^{R}_{\times}=\{k\in V \mid \text{car } k \text{ is RL-unused.}\}$ denote the set of all RL-unused cars. By Proposition \ref{pro2}, $V^{R}_{\times}$ always contains the first car in the downstream sequence, i.e., $\pi_1\in V^{R}_{\times}$. %\noindent Cars 4 and 5 in Figure \ref{fig1}(b) can thus be RL-unused ($4=\pi_1=\max \{4\}$ and $5=\pi_4=\max\{4,3,2,5\}$). Let $V^{R}_{\times}=\{k\in V \mid \text{car } k \text{ is RL-unused.}\}$ denote the set of all RL-unused cars. It evidently includes the first car $\pi_1$ in $\Pi^{D}$, i.e., $\pi_1\in V^{R}_{\times}$.
%Moreover, the optimal departure and arrival timestamps from the upstream shop and at the downstream shop for specific cars can be derived as follows.
Moreover, the optimal departure timestamps from the upstream shop for cars in $\{1,2,\dots,\pi_1\}$ can be derived as follows.

%%%%%%%%%%According to Proposition \ref{pro1}, in the example shown in Figure \ref{fig4}, cars 4 and 5 do not enter the RL in the optimal solution. The following definition describes these cars as \textit{RL-unused}.

%%%%%%%%%%\begin{definition}\label{def2}
%%%%%%%%%%    A car is said to be RL-unused if it satisfies the condition %specified
%%%%%%%%%%    in Proposition \ref{pro1}.
    %A car is said to be RL-unused if it does not enter the RL in the optimal solution.
%%%%%%%%%%\end{definition}

%%%%%%%%%%Let $V^{R}_{\times}=\{k\in V \mid \text{car } k \text{ is RL-unused.}\}$ denote the set of RL-unused cars. Evidently, it includes the first car in $\Pi^D$. Furthermore, the departure and arrival timestamps of specific cars in the optimal solution can be determined (see \ref{app1} for the proof), as stated below.

\begin{corollary}\label{coro1}
    %%The first car in the downstream sequence, i.e., car $\pi_1$, is an RL-unused car. In the optimal solution, it departs from the upstream shop at timestamp $\pi_1-1$ and arrives at the downstream shop at timestamp $\pi_1+\bar{q}$.
    %The first car in the downstream sequence $\Pi^D$, i.e., car $\pi_1$, is RL-unused.
    In the optimal CRSP-MS solution, each car $k\in \{1,2,\dots,\pi_1\}$ departs from the upstream shop at timestamp $k-1$.%, and car $\pi_1$ reaches the downstream shop at timestamp $\pi_1+\bar{q}$.
\end{corollary}

Before solving the CRSP-MS, it is essential to verify whether $\Pi^{U}$ can be transformed into $\Pi^{D}$ using the provided buffer, which poses a \textit{resequencing feasibility problem} (RSFP). In this study, we examine the RSFP under a given MB-RL and a given MB. For the MB-RL case, Lemma \ref{lema1} first confirms a feasibility condition for moving a car to its target downstream position. By checking this condition for all cars, we develop a polynomial-time algorithm for the RSFP, detailed in the online supplement. For the MB, we prove that the RSFP is %strongly NP-hard
strongly NP-complete and can be addressed by solving the car-to-lane assignment problem of \citet{guo2025logic} (a mixed-integer programming model). Both their model and our NP-completeness %strong NP-hardness
proof are given in the online supplement. For clarity, Proposition \ref{pro1} summarizes these resequencing feasibility-checking methods by buffer type.

%Note that before solving the CRSP-MS, it is essential to verify whether $\Pi^{U}$ can be transformed into $\Pi^{D}$ using the given buffer, which defines the \textit{resequencing feasibility problem} (RSFP). In this study, we examine the RSFP with two buffer types: a given MB-RL and a given MB. Our polynomial-time algorithm, the proof of strong NP-hardness, and the assignment formulation of \citet{guo2025logic} are provided in the online supplement.

\begin{lemma}\label{lema1}
    Let $\overline{V}(k)=\{v\in V \mid v<k;p_v>p_k\in P\}$ denote the set of cars that must remain in the MB-RL when car $k\in V$ is about to enter from the upstream shop. It is feasible to move car $k\in V$ to its downstream position $p_k\in P$ via the MB-RL if and only if $|\overline{V}(k)|<\bar{l}\times \bar{q}$, where $\bar{l}$ is the total number of lanes and $\bar{q}$ is the lane capacity.
\end{lemma}

\begin{proposition}\label{pro1}
    \textup{(i)} The RSFP with an MB-RL is solvable in $\mathcal{O}(d^2)$ time, where $d$ is the number of cars. \textup{(ii)} The RSFP with an MB is solvable via the assignment model of \citet{guo2025logic}.
\end{proposition}
%The RSFP with an MB-RL is solvable in O(d²) time using our polynomial-time algorithm. The RSFP with an MB is solvable via the assignment model of Guo et al.

\subsection{Time-Space Network Representation}\label{sec3-2}
%\subsection{\textcolor{red}{Resequencing feasibility check}}\label{sec3-2}
This section further defines the CRSP-MS on a TSN graph, a common representation for scheduling problems that involve both temporal and spatial decisions \citep{hane1995fleet}. %\citep{yuan2017novel}.
In a TSN graph, each node represents a location at a given timestamp. An arc connects two nodes if movement between their locations is feasible within the corresponding time interval.

%To construct the TSN graph $\mathcal{G}=(\mathcal{N},\mathcal{A})$ for the CRSP-MS, let $\tau^m=\bar{q}+1$ be the direct travel time (in TUs) from the upstream shop ($s$) to the downstream shop ($e$), and $\tau^r=2\bar{q}$ that from $s$ to the leftmost cell $c_{\bar{l}1}$ of the RL.
To build the TSN graph $\mathcal{G}=(\mathcal{N},\mathcal{A})$ for the CRSP-MS, let $\tau^m=\bar{q}+1$ (in TUs) be the non-waiting travel time from the upstream shop ($s$) to the downstream shop ($e$) without the RL, and $\tau^r=2\bar{q}$ be that from $s$ to the leftmost cell $c_{\bar{l}1}$ of the RL on the first visit. The node set $\mathcal{N}$ is divided into four subsets: $\mathcal{N}=\mathcal{N}^{s}\cup \mathcal{N}^e\cup \mathcal{N}^{f}\cup \mathcal{N}^{r}$, to distinguish the type of location $i\in C$ where a car stays at a specific timestamp $t\in T$. \textit{Origin nodes} $\mathcal{N}^s=\{(s,t)\mid t=0,\dots,\tau-\tau^m\}$ correspond to $s$, \textit{destination nodes} $\mathcal{N}^e=\{(e,t)\mid t=\tau^m,\dots,\tau\}$ to $e$, \textit{FL nodes} $\mathcal{N}^{f}=\{(c_{lq},t)\mid l\in L^f,q\in Q,t=q,\dots,\tau-\tau^m+q\}$ to forward lanes, and \textit{RL nodes} $\mathcal{N}^r=\{(c_{\bar{l}q},t)\mid c_{\bar{l}q}\in C, q\in Q, t=\tau^r-q+1,\dots,\tau-\bar{q}-q\}$ to the RL.

The arc set $\mathcal{A}$ represents all possible location changes of a car from timestamps $t$ to $t+1$ ($t\in T\backslash\{\tau\}$). It consists of four types: \textup{(i)} \textit{pull-in arcs} $\mathcal{A}^s=\{(s,t;c_{l1},t+1)\mid (s,t)\in \mathcal{N}^s,(c_{l1},t+1)\in \mathcal{N}^f\}$ for entering a forward lane from $s$, \textup{(ii)} \textit{pull-out arcs} $\mathcal{A}^e=\{(c_{l\bar{q}},t-1;e,t\mid (c_{l\bar{q}},t-1)\in \mathcal{N}^f, (e,t)\in \mathcal{N}^e\}$ for exiting a forward lane to $e$, \textup{(iii)} \textit{waiting arcs} $\mathcal{A}^w=\{(c_{lq},t;c_{lq},t+1)\mid (c_{lq},t),(c_{lq},t+1)\in \mathcal{N}^f\cup \mathcal{N}^r\}$ for staying in the same cell within the MB-RL, and \textup{(iv)} \textit{moving arcs} $\mathcal{A}^m=\mathcal{A}^{mc}\cup \mathcal{A}^{fr}\cup \mathcal{A}^{rf}$ for cell-to-cell transitions. Specifically, $\mathcal{A}^{mc}=\{(c_{lq},t;c_{l,q+1},t+1)\mid (c_{lq},t),(c_{l,q+1},t+1)\in \mathcal{N}^f\}\cup \{(c_{\bar{l}q},t;c_{\bar{l},q-1},t+1)\mid (c_{\bar{l}q},t),(c_{\bar{l},q-1},t+1)\in \mathcal{N}^r\}$ %$\mathcal{A}^{mc}=\{(c_{lq},t;c_{l,q+1},t+1)\mid (c_{lq},t),(c_{l,q+1},t+1)\in \mathcal{N}^f\cup \mathcal{N}^r\}$
is for intra-lane movements, $\mathcal{A}^{fr}=\{(c_{l\bar{q}},t;c_{\bar{l}\bar{q}},t+1)\mid (c_{l\bar{q}},t)\in \mathcal{N}^f,(c_{\bar{l}\bar{q}},t+1)\in \mathcal{N}^r\}$ for transfers from a forward lane to the RL, and $\mathcal{A}^{rf}=\{(c_{\bar{l}1},t;c_{l1},t+1)\mid (c_{\bar{l}1},t)\in \mathcal{N}^r, (c_{l1},t+1)\in \mathcal{N}^f\}$ for transfers from the RL to a forward lane.

%%%%%To represent all possible changes in the location of a car occurring during one TU, the arc set A1 is divided as A1 = A11 + A12 + A13 + A14 to distinguish four types of possible movements.

%%%%%To construct the complete TSN graph G1 that fully represents the CRSP-MS, let t1 denote the travel time (in TUs) from the upstream shop (s) to the downstream shop, and t2 that from s to the leftmost cell c1 of the RL. By spatial location type, the node set N1 is partitioned into four subsets.

In $\mathcal{G}$, each \textit{path} from an \textit{origin node} to a \textit{destination node} defines a feasible trajectory for a car moving from the upstream to the downstream shop. The CRSP-MS is then equivalent to finding such a \textit{path} for each car in $V$ to minimize the arrival timestamp at the downstream shop of the last car $\pi_d$. In addition to following the given upstream and downstream sequences, the solution must adhere to %the MB-RL operational rules:
the FIFO rules, lane capacity limits, and restrict at most one car in (out) a forward lane at a time. As an example, using the MB-RL in Figure \ref{fig2}(a) to achieve the sequence change in Figure \ref{fig1}(b), Figure \ref{fig2}(b) shows the built $\mathcal{G}$ and a feasible solution with a makespan of 10 TUs.
%For the example in Figure \ref{fig2}(a), Figure \ref{fig2}(b) shows the built $\mathcal{G}$ and a feasible solution yielding a makespan of 10 TUs.

%Besides following the predefined upstream and downstream sequences ($\Pi^{U}$ and $\Pi^{D}$), the solution must adhere to the MB-RL operational rules: FIFO principles, lane capacity limits, and the restriction that at most one car may enter (exit) a forward lane at a time.

For each car in $V$, we define a binary decision variable for each arc in $\mathcal{A}$ to indicate whether the arc is part of the solution \textit{path} for that car. The CRSP-MS is then formulated as a multi-commodity network flow model (detailed in the online supplement). However, this compact formulation becomes computationally intractable for practical-scale instances (see Section \ref{sec5-2}). To handle larger-scale cases, we next develop an exact approach that integrates BP with CP.

%%%%%%%For each car in $V$, we define a binary decision variable for each arc in $\mathcal{A}$ to indicate whether the arc is part of the solution \textit{path} for that car. Using these variables, the CRSP-MS is formulated as a multi-commodity network flow model (detailed in the online supplement). As shown in Section \ref{sec5-2}, this compact formulation becomes computationally intractable for real-world-scale instances. To handle larger-scale cases, we develop an exact approach that integrates \textit{branch-and-price-and-cut} (BPC) with CP in the next section.

%%%%For each car $k\in V$, we define a binary decision variable for each arc in $\mathcal{A}$ to indicate whether the arc is part of the solution \textit{path} for car $k$. Using these variables, the CRPS-MS is formulated as a multi-commodity network flow model (detailed in the online supplement). %By solving this compact formulation, we obtain the optimal solution with a minimum makespan of 10 for the example in Figure 1(a), as illustrated in Figure 1(b).
%%%%However, results in Section 3.1 indicate that this formulation becomes computationally intractable for real-world cases. To handle larger-scale instances, we develop an exact approach that integrates branch-and-price-and-cut with constraint programming in the following section.

\section{A branch-and-price approach integrated with constraint programming}\label{sec4}
Our BP-CP approach has two stages. This section first explains its motivation and overall framework. Then, it details the components of each stage, the search for integer solutions, and the algorithmic enhancements, concluding with a complete implementation procedure.
%Our BP-CP approach comprises three stages. This section first explains its motivation and overall framework. Then, it details the components of each stage, the search for integer solutions, and the algorithmic enhancements, concluding with a full implementation procedure.

%%%%The BP-CP approach proposed in this section consists of three stages. We first explain the motivation behind it and outline its framework. Then, all its functional components are detailed, including formulations for each stage, search for integer solutions, and algorithmic enhancements. Finally, we conclude with how to implement the entire approach.

\subsection{Overview of the proposed approach}\label{sec41}
The BP framework is highly regarded as an exact method for solving problems defined on TSN graphs. In BP, each node in the \textit{branch-and-bound} (BB) tree is solved using CG. During each CG iteration, the \textit{restricted master problem} (RMP), a linear relaxation of the MP restricted to a subset of variables (i.e., columns), is solved first. Its dual solution is then used to solve \textit{pricing subproblems} (SPs) to identify new columns with negative reduced costs (for minimization problems). After adding these columns, the RMP is re-solved. This process repeats until strong duality conditions are met. Once CG ends, the current node is fully explored. If the RMP solution is integral, it provides a feasible solution to the original problem. Otherwise, branching is performed to generate new BB nodes in pursuit of integer solutions. The optimal solution is found after exploring all the BB nodes.

%%%%The BP and BPC methods are widely used for problems defined on TSN graphs \citep{zhang2024computing}. In BP, each node in the \textit{branch-and-bound} (BB) tree is solved using CG %\citep{lubbecke2005selected}
%%%%. During each iteration of CG, the RMP, a linear relaxation of the original problem but with only a subset of variables (i.e., columns), is solved first. Using its solution, \textit{pricing subproblems} (SPs) are then solved to identify new columns with negative reduced costs (for minimization problems). After adding these columns, the RMP is resolved. This process repeats until strong duality conditions are met, and BPC further enhances it with cutting planes \citep{costa2019exact}. Once CG ends, the current node is explored. If its optimal RMP solution is integral, it yields a feasible solution to the original problem. If not, branching is performed to create new BB nodes in pursuit of integer solutions. The optimal solution is obtained after exploring all the BB nodes.

%The BP or BPC methods are widely used for multi-commodity flow problems \citep{santini2018branch, zhang2024computing}. In BP, each node in the \textit{branch-and-bound} (BB) tree is solved through an iterative procedure called \textit{column generation} (CG) \citep{lubbecke2005selected}.

%However, the strong symmetries in the CRSP-MS makes it difficult to applying classical BP or BPC.
However, applying classical BP to the CRSP-MS is hindered by strong symmetries in the problem. The homogeneity of forward lanes, for example, allows cars to interchange their assigned lanes, resulting in many structurally distinct but makespan-identical columns. These columns lead to highly fractional RMP solutions, necessitating exhaustive branching to obtain integrality.

%%%%%%However, the bottleneck in applying classical BP or BPC to the CRSP-MS lies in the inherent problem symmetries, which induce numerous redundant columns generated and highly fractional RMP solutions. For example, due to the homogeneity of forward lanes and their internal cells, swapping the chosen forward lanes among cars or having a car wait in different cells of the same lane results in distinct but makespan-identical columns. These columns cause excessive RMP variables to take fractional values, necessitating exhaustive branching to find integer solutions.

To mitigate branching explosion caused by symmetries, we develop a two-stage BP-based method (i.e., BP-CP) that enforces integrality on only a subset of variables. Once they take integer values, they serve as input for a second-stage feasibility problem to find a solution to the original problem.
%To mitigate branching explosion from symmetries, we develop a three-stage BPC-based approach (i.e., BP-CP) that enforces integrality only on a subset of variables. Once they take integer values, they serve as input for a third-stage feasibility problem to find a solution to the original problem.

%%%To mitigate the branching explosion caused by symmetries, our BP-CP approach enforces integrality on only a subset of variables via the BB tree. The integer values of these variables then serve as input for a feasibility problem to find a feasible solution to the original problem.

%%%%%%%%%determines the specific forward lane entered by each car at each time step, along with corresponding entry and exit timestamps for all lanes

Specifically, in BP-CP, the BB node exploration proceeds in two stages: CG in the first stage ($\mathcal{S}_1$) and feasible-solution search in the second stage ($\mathcal{S}_2$). In $\mathcal{S}_1$, the RMP and SPs are solved iteratively. After CG converges, we check whether the variables related to the makespan and the number of times each car enters the RL are integers. If so, these values define an assignment problem in $\mathcal{S}_2$ that determines the specific forward lane each car enters at each time, along with the corresponding lane-entry and lane-exit timestamps. A feasible assignment yields a CRSP-MS solution and closes the node, whereas an infeasible one results in a feasibility cut for the MP and a restart of CG. If the relevant variables are not all integers, we branch from this node to create new child nodes.

Given that our assignment problem concerns time-related decisions, we model it as a CP feasibility problem. CP features rich variable types and constraints for more flexible modeling. The \textit{interval variables}, each defined by three integer attributes: \textit{start time}, \textit{end time}, and \textit{duration}, can be efficiently managed in scheduling problems using CP functions and filtering algorithms \citep{kasapidis2025unified}. In our case, the \textit{start} and \textit{end times} indicate when a car enters and leaves a lane.

%%%%%Given that our assignment problem in $\mathcal{S}_3$ concerns time-related decisions, we model it as a CP feasibility problem. CP features various variable and constraint types that allow for more flexible modeling and stronger inference techniques. The \textit{interval variables}, each defined by three integer attributes related to time: \textit{start time}, \textit{end time}, and \textit{duration} of the interval, find their particular use in scheduling problems \citep{kasapidis2025unified}. These variables capture temporal relations and can be efficiently managed using CP %ordering
%%%%%precedence constraints and filtering algorithms. In our context, the \textit{start time} and \textit{end time} indicate when a car enters and leaves a lane.

Unlike previous studies that use CP in BP to construct SPs \citep{hashemi2016constraint} or across all BP components such as CG and row generation \citep{zhang2024computing}, our BP-CP uses a CP model to find feasible solutions to the original problem. Its two-stage mechanism resembles the BP of \citet{lalonde2022branch} for a multiple knapsack problem. They require only the variables for item and bin assignments to be integers and, with these values, solve a feasibility problem that finalizes the item-to-knapsack assignment. We now detail each component of our approach.
\subsection{First Stage: Master Problem and Two Reduced Variants}\label{sec4-1}
Three MP formulations that use \textit{path}-based variables are proposed. The first, $\mathcal{M}_1$, is based on graph $\mathcal{G}$ introduced in Section \ref{sec3-2} and provides a complete formulation of the CRSP-MS. The other two, $\mathcal{M}_2$ and $\mathcal{M}_3$, are derived from $\mathcal{M}_1$ by relaxing selected constraints and are based on two types of compressed TSN graphs constructed for each car.

As preliminary steps, for each car $k\in V$, let $\tau^s_k=k-1$ be its earliest departure timestamp from the upstream shop and $\tau^e_k=\tau-(d-p_k)$ its latest arrival timestamp at the downstream. The maximum number of times it can enter the RL during resequencing, $\bar{n}_k$, is the largest integer $n$ such that $n\cdot\tau^r\leqslant\tau^e_k-\tau^s_k$. By Proposition \ref{pro2}, if car $k$ is RL-unused, then $\bar{n}_k=0$ and its TSN graphs contain no RL-related arcs. The set of all possible numbers of times car $k$ enters the RL is $N_k=\{0,1,\dots,\bar{n}_k\}$.
%If $\bar{n}_k=0$, then car $k$ never enters the RL, and its TSN graphs have no RL-related arcs. The set of all possible numbers of times car $k$ enters the RL is $N_k=\{0,1,\dots,\bar{n}_k\}$.

%%As a preliminary applicable to the graphs built for each car k1,

%%%%Then, $\mathcal{M}_1$ and its associated RMP are formulated using variables defined for the \textit{paths} of each car $k\in V$ in $\mathcal{G}$. To classify all \textit{paths} for car $k$, let $\mathcal{G}^1_k=(\mathcal{N}^1_k,\mathcal{A}^1_k)$ be the subgraph formed by nodes and arcs in $\mathcal{G}$ whose related timestamps lie between $\tau^s_k$ and $\tau^e_k$. The car-specific node set $\mathcal{N}^1_k$ and arc set $\mathcal{A}^1_k$ retain the same node and arc type as in $\mathcal{G}$, fully defined in the online supplement. For the example in Figure 1(b), Figure 3(a) shows the subgraph $\mathcal{G}^1_3$ for car 3, whose $\tau^s_3=2$ and $\tau^e_3=9$.
%As an example, Figure 3(a) shows the subgraph $\mathcal{G}^1_3$ of $\mathcal{G}$ for car 3 in Figure 1(b), whose $\tau^s_3=2$ and $\tau^e_3=9$.
%Figure 3(a) shows the subgraph of $\mathcal{G}$ for car 3 in Figure 1(b), i.e., $\mathcal{G}^1_3$.

%Then, $\mathcal{M}_1$ and its associated RMP are formulated using variables defined for the \textit{paths} of each car $k\in V$ in $\mathcal{G}$.
Then, $\mathcal{M}_1$ is formulated using variables defined for the \textit{paths} of each car $k\in V$ in $\mathcal{G}$. To identify all \textit{paths} for car $k$, define $\mathcal{G}^{1}_k=(\mathcal{N}^1_k,\mathcal{A}^1_k)$ as the subgraph of $\mathcal{G}$ with the nodes and arcs whose timestamps fall between $\tau^s_k$ and $\tau^e_k$. %with nodes and arcs between $\tau^s_k$ and $\tau^e_k$.
%let $\mathcal{G}^{1}_k=(\mathcal{N}^1_k,\mathcal{A}^1_k)$ be the subgraph of $\mathcal{G}$ containing nodes and arcs with timestamps between $\tau^s_k$ and $\tau^e_k$.
%let $\mathcal{G}^{1}_k=(\mathcal{N}^1_k,\mathcal{A}^1_k)$ be the subgraph formed by nodes and arcs in $\mathcal{G}$ with timestamps between $\tau^s_k$ and $\tau^e_k$.
The car-specific node set is: $\mathcal{N}^1_k=\mathcal{N}^s_k\cup \mathcal{N}^e_k\cup \mathcal{N}^{1f}_k \cup \mathcal{N}^{1r}_k$, where $\mathcal{N}^s_k=\{(s,t)\in \mathcal{N}^s\mid t = \tau^s_k,\dots,\tau^e_k-\tau^m\}$ represents \textit{origin nodes}, $\mathcal{N}^e_k=\{(e,t)\in \mathcal{N}^e\mid t=\tau^s_k+\tau^m,\dots,\tau^e_k\}$ \textit{destination nodes}, $\mathcal{N}^{1f}_k=\{(c_{lq},t)\in \mathcal{N}^f\mid q\in Q, t=\tau^s_k+q,\dots,\tau^e_k-(\tau^m-q)\}$ \textit{FL nodes}, and $\mathcal{N}^{1r}_k=\{(c_{\bar{l}q},t)\in \mathcal{N}^r\mid q\in Q, t=\tau^s_k + \tau^r-q+1,\dots,\tau^e_k-\bar{q}-q\}$ \textit{RL nodes}. Correspondingly, the four arc types are adapted into subsets $\mathcal{A}^s_k$, $\mathcal{A}^e_k$, $\mathcal{A}^w_k$, and $\mathcal{A}^m_k=\mathcal{A}^{mc}_k\cup \mathcal{A}^{fr}_k\cup \mathcal{A}^{rf}_k$, defined in detail in the online supplement, to form the arc set $\mathcal{A}^1_k$. For the same example used in Figure \ref{fig2}(b), the subgraph $\mathcal{G}^1_1$ for car 1, whose $\tau^s_1=0$ and $\tau^e_1=11$, is identical to $\mathcal{G}$ shown in that figure.
%For the same example used in Figure \ref{fig2}(b), the subgraph $\mathcal{G}^1_1$ for car 1, whose $\tau^s_1=0$ and $\tau^e_1=11$, is identical to $\mathcal{G}$ in Figure \ref{fig2}(b).
%For car 1 in Figure \ref{fig1}(b), with $\tau^s_1=0$ and $\tau^e_1=11$, its subgraph $\mathcal{G}^1_1$ is identical to $\mathcal{G}$ in Figure \ref{fig2}(b).
%As an example, Figure \ref{fig3}(a) shows the subgraph $\mathcal{G}^1_1$ for car 1 in Figure \ref{fig2}(a), whose $\tau^s_1=0$ and $\tau^e_1=11$.

%%Correspondingly, replacing original nodes sets used in the definition of arc sets of G1, the four arc types are specialized to car k1 as: N1 = N12 + N13 + N14 + N15.
%%With these definition of node sets, the four types of arcs are correspondingly specialized to car k1 as:

For each car $k\in V$, let $\Omega_k$ denote the set of feasible \textit{paths} in the TSN graph that enable it to travel from the upstream to the downstream shop. In $\mathcal{G}^1_k$, such a \textit{path} $r\in \Omega_k$ is a sequence of arcs starting with a \textit{pull-in arc} and ending with a \textit{pull-out arc}, characterized by five parameters: \textup{(i)} the departure timestamp $s^k_r$ from the upstream shop, \textup{(ii)} the arrival timestamp $t^k_r$ at the downstream shop, \textup{(iii)} a binary indicator $\rho^{kr}_{it}$ for whether node $(i,t)\in \mathcal{N}^1_k$ is visited under $r$, \textup{(iv)} a binary indicator $\Tilde{\rho}^{kr}_{it}=1$ if and only if node $(i,t)\in \mathcal{N}^1_k$ is visited under $r$ but not via a \textit{waiting arc}, and \textup{(v)} the number $\zeta^k_r$ of times car $k$ enters the RL under $r$. Let binary variable $x^k_r=1$ if and only if car $k$ selects \textit{path} $r\in \Omega_k$, and binary variable $z^k_n=1$ if it enters the RL exactly $n\in N_k$ times, and 0 otherwise. The formulation of $\mathcal{M}_1$ is written as: %\textcolor{red}{The first \textit{path}-based formulation of the CRSP-MS, i.e., $\mathcal{M}_1$, is written as:}%The first \textit{path}-based IP model for the CRSP-MS, i.e., $\mathcal{M}_1$, is formulated as follows.
%\vspace{-6mm}%The first \textit{path}-based IP model, which fully formulates the CRSP-MS, i.e., $\mathcal{M}_1$, is then given as follows.
\begin{subequations}\label{RMP1}
\begin{align}
  \label{m1-obj}(\mathcal{M}_1)\,\ \min&\sum_{r\in \Omega_{\pi_d}}{t^{\pi_d}_r x^{\pi_d}_r} \\%[-2pt]
        s.t.&\quad \nonumber \\
        \label{m1-con1}&\sum_{r\in \Omega_{k}}{x^k_r}=1,\,\ &&k\in V, \\%[-2pt]
        \label{m1-con2}&\sum_{r\in \Omega_{k}}{s^k_r  x^k_r}-\sum_{r\in \Omega_{k-1}}{s^{k-1}_r x^{k-1}_r}\geqslant 1, &&k\in V\backslash\{1\}, \\%[-12pt]
        \label{m1-con3}&\sum_{r\in \Omega_{\pi_{p+1}}}{t^{\pi_{p+1}}_r x^{\pi_{p+1}}_r}-\sum_{r\in \Omega_{\pi_{p}}}{t^{\pi_{p}}_r x^{\pi_{p}}_r}\geqslant 1, &&p\in P\backslash \{d\}, \\%[-3pt]
        \label{m1-con4}&\sum_{k\in V}{\sum_{r\in \Omega_k}{\rho^{kr}_{it} \cdot x^k_r}}\leqslant 1, &&(i,t)\in \mathcal{N}^{f}\cup \mathcal{N}^{r}, \\%[-3pt]
        \label{m1-con5}&\sum_{l\in L^f}{\sum_{k\in V}{\sum_{r\in \Omega_k}{\Tilde{\rho}^{kr}_{c_{l1},t} \cdot x^k_r}}}\leqslant 1, &&t\in \{1,\dots,\tau-\bar{q}\}, \\%[-3pt]
        \label{m1-con6}&\sum_{k\in V}{\sum_{r\in \Omega_{k}}{(\rho^{kr}_{et}+\Tilde{\rho}^{kr}_{c_{\bar{l}\bar{q}},t})\cdot x^{k}_{r}}}\leqslant 1,  &&t\in \{\tau^m,...,\tau \}, \\%[-3pt]
        \label{m1-con7}&\sum_{n\in N_k}{z^k_n}=1, &&k\in V, \\%[-3pt]
        \label{m1-con8}&\sum_{r\in \Omega_k}{\zeta^k_r\cdot x^k_r}-\sum_{n\in N_k}{n\cdot z^k_n}=0, &&k\in V, \\%[-3pt]
        \label{m1-con9}&\left[\texttt{Feasibility Cuts}\right], &&  \\%[-3pt]
        \label{m1-con10}&x^{k}_{r}, z^{k}_{n}\in \{0,1\}, &&r\in \Omega_k, n\in N_k, k\in V.
\end{align}
%%%%\setcounter{RMPone}{\value{equation}}
%record how many equations are put in this "align"
\end{subequations}

The objective function \eqref{m1-obj} minimizes the arrival timestamp of car $\pi_d$ at the downstream shop, i.e., the \textit{resequencing makespan}. Constraints \eqref{m1-con1} ensure that exactly one \textit{path} is selected per car.
Recall that upstream car $\pi_p\in V$ moves to downstream position $p\in P$, and $\pi_{p+1}\in V$ to $p+1\in P$. %Recall that upstream cars $\pi_p$ and $\pi_{p+1}$ move to downstream positions $p$ and $p+1$, respectively.
Constraints \eqref{m1-con2} and \eqref{m1-con3} require the cars to depart and arrive in the given upstream and downstream orders. Constraints \eqref{m1-con4} limit each \textit{FL} or \textit{RL node} in $\mathcal{G}$ to at most one visit to ensure that FIFO rules and lane capacity limits are respected. Constraints \eqref{m1-con5}/\eqref{m1-con6} enforce that at most one car can enter/exit a forward lane at a time. Constraints \eqref{m1-con7} require each car to enter the RL a certain number of times, with relevant variables linked by constraints \eqref{m1-con8}. As stated earlier, a feasibility cut is added to the MP if the integer values of partial variables are proven infeasible by solving the CP model in $\mathcal{S}_2$. Constraints \eqref{m1-con9} incorporate all such generated cuts, each introduced in Section \ref{sec4-3} following the CP model. Finally, constraints \eqref{m1-con10} impose binary requirements on decision variables.

Model $\mathcal{M}_1$ generally includes numerous variables as the graph expands. We consider its restricted version, in which the \textit{path} set $\Omega_k$ is replaced by a smaller subset $\Omega^{\prime}_k\subset \Omega_k$ for each car $k\in V$, and all integrality requirements on variables are removed. The resulting linear program (LP), denoted by $\mathcal{M}^{\text{LP}}_1$, is the first RMP in our BP-CP. It is solved repeatedly as new \textit{paths} with negative reduced costs are dynamically added to $\Omega^{\prime}_k$, until no such \textit{paths} remain.

%%Model $\mathcal{M}_1$ is regarded as the first so-called \textit{master problem} (MP), typically containing numerous variables as the graph size rapidly grows. By replacing the \textit{path} set $\Omega_k$ with a smaller one, $\Omega^{\prime}_k$, for each car $k\in V$, we get a restricted $\mathcal{M}_1$, i.e., the first RMP in our BP-CP. Its linear relaxation, from relaxing the integrality requirements, is denoted by $\mathcal{M}^{\text{LP}}_1$. During CG, $\mathcal{M}^{\text{LP}}_1$ is continually updated as new \textit{paths} with negative reduced costs are added to $\Omega_k^{\prime}$ ($\forall k\in V$) until no such \textit{paths} exist.

%%%%The model $\mathcal{M}_1$ is regarded as the first so-called \textit{master problem}(MP) in our approach. It typically contains an exponential number of variables (or columns) because the number of feasible \textit{paths} increases rapidly with graph size, while most variables remain 0 in the optimal solution. To address this issue, following the principle of BP methods, we consider a relaxed version of $\mathcal{M}_1$, where the \textit{path} set $\Omega_k$ for each car $k\in V$ is replaced by a smaller set $\Omega^{'}_k$, and the integrality requirements on the variables are relaxed. The resulting linear programming model, denoted by $\mathcal{M}^{\text{LP}}_1$, serves as the first RMP in our BP-CP. During CG, $\mathcal{M}^{\text{LP}}_1$ is iteratively updated as new paths with negative reduced costs are added to $\Omega^{'}_k$, until no such paths exist.

As shown in Section \ref{sec5-2}, $\mathcal{M}_1$ yields the tightest lower bound on the makespan by fully modeling the CRSP-MS. However, its use in BP-CP poses two challenges stemming from the structure of $\mathcal{G}$. The first is the strong symmetries inherent in $\mathcal{G}$, as described in Section \ref{sec41}. The second is that the performance of BP-CP with $\mathcal{M}_1$ degrades significantly as the problem size increases. Each \textit{FL} or \textit{RL node} in $\mathcal{G}$ induces a constraint of type \eqref{m1-con4}, so larger graphs add numerous such constraints, which slow down the solution of $\mathcal{M}^{\text{LP}}_1$ and the CG process.

To address these issues, we drop constraints \eqref{m1-con4} from $\mathcal{M}_1$ to speed up the solution of $\mathcal{M}^{\text{LP}}_1$, while accepting a weaker linear relaxation as a trade-off. This removal renders some \textit{FL} and \textit{RL nodes} in $\mathcal{G}$ redundant, allowing us to eliminate them, and \textit{paths} no longer need to track the exact position of a car at every timestamp. Building on this idea, we next introduce two reduced variants of $\mathcal{M}_1$, each based on a type of compressed TSN graph.

%%To address these issues, we consider dropping constraints \eqref{m1-con4} from $\mathcal{M}_1$ to speed up the solution of $\mathcal{M}^{\text{LP}}_1$, whereas accepting a weaker linear relaxation as a trade-off. This removal makes some \textit{FL} and \textit{RL nodes} in $\mathcal{G}$ become redundant, allowing them to be eliminated, and \textit{paths} no longer need to track the exact position of a car at every timestamp. Building on this idea, we next introduce two reduced variants of $\mathcal{M}_1$, each defined on a compressed TSN graph, along with their related RMPs.

%%%%%To tackle these challenges, we consider dropping constraints \eqref{m1-con4} from $\mathcal{M}_1$ to accelerate CG, while accepting a weaker linear relaxation as a trade-off. Removing these constraints eliminates the need for each \textit{path} to specify the exact position of a car at every timestamp, allowing arcs in $\mathcal{G}^1_k$ that represent consecutive movements of car $k\in V$ within a lane to be merged. The resulting graph is more compact and helps reduce problem symmetries. Building on this idea, we next present two reduced variants of $\mathcal{M}_1$, each based on a compressed TSN graph, and their related RMPs, to progressively address the symmetries from homogeneous lane cells and forward lanes.

%\subsubsection{First reduced RMP}
\subsubsection{First reduced MP}
The main idea behind the first reduction is to ignore detailed intra-lane car movements and to keep only the nodes and arcs that capture which lane each car enters each time, along with the corresponding lane entry and exit timestamps.
%The main idea behind the first reduction of $\mathcal{M}^1$ and the compression of $\mathcal{G}$ is to ignore detailed intra-lane car movements and to keep only nodes and arcs that capture which lane each car enters each time, along with the corresponding lane entry and exit timestamps.

%Let G2 denote the TSN graph compressed by G1 for each car k1 after the first reduction.
%We start by compressing $\mathcal{G}^1_k$ for each car $k\in V$. Let $\mathcal{G}^2_k=(\mathcal{N}^2_k,\mathcal{A}^2_k)$ denote the resulting TSN graph. Its node set $\mathcal{N}^2_k$ consists of four subsets: $\mathcal{N}^2_k=\mathcal{N}^s_k\cup \mathcal{N}^e_k\cup \mathcal{N}^{2F}_k\cup \mathcal{N}^{2R}_k$.
Let $\mathcal{G}^2_k=(\mathcal{N}^2_k,\mathcal{A}^2_k)$ denote the compressed TSN graph for each car $k\in V$ after the first reduction. The node set $\mathcal{N}^2_k$ consists of four subsets: $\mathcal{N}^2_k=\mathcal{N}^s_k\cup \mathcal{N}^e_k\cup \mathcal{N}^{2F}_k\cup \mathcal{N}^{2R}_k$. The \textit{origin nodes} $\mathcal{N}^s_k$ and \textit{destination nodes} $\mathcal{N}^e_k$ remain unchanged from $\mathcal{G}^1_k$ to record the upstream-departure and downstream-arrival timestamps. To identify which forward lane a car enters and its entry timestamp, the \textit{FL node} set $\mathcal{N}^{2F}_k=\bigcup_{l\in L^f}{\mathcal{N}^{2l}_k}$ only retains nodes in $\mathcal{G}^1_k$ associated with the leftmost cells of forward lanes, with subset $\mathcal{N}^{2l}_k=\{(c_{l1},t)\in \mathcal{N}^{1f}_k\}$ specific to forward lane $l\in L^f$. Likewise, the \textit{RL node} set $\mathcal{N}^{2R}_k=\{(c_{\bar{l}\bar{q}},t)\in \mathcal{N}^{1r}_k\}$ contains only nodes from $\mathcal{G}^1_k$ associated with the rightmost cell of the RL to capture the RL entry timestamp. The exit timestamp $t\in T$ from a forward lane or the RL can be derived by a node from $\mathcal{N}^e_k\cup \mathcal{N}^{2R}_k$ or $\mathcal{N}^{2F}_k$ with timestamp $t+1\in T$.

%With the compressed node set $\mathcal{N}^2_k$,
With $\mathcal{N}^2_k$, the \textit{pull-out}, \textit{waiting}, and \textit{moving arcs} in $\mathcal{G}^1_k$ can be merged into three macro-types, each describing a macro-movement of cars and further subdivided by specific forward lanes. %each representing a car macro-movement and further divided by a specific forward lane.
For each forward lane $l\in L^f$, a \textit{macro-pull-out arc} in $\mathcal{A}^{le}_k=\{(c_{l1},t;e,t^{\prime})\mid (c_{l1},t)\in \mathcal{N}^{2l}_k,(e,t^{\prime})\in \mathcal{N}^e_k, t^{\prime}-t\geqslant\bar{q}\}$ indicates a direct movement from its leftmost cell $c_{l1}$ to the downstream shop. An \textit{enter-RL arc} in $\mathcal{A}^{lR}_k=\{(c_{l1},t;c_{\bar{l}\bar{q}},t^{\prime})\mid (c_{l1},t)\in \mathcal{N}^{2l}_k, (c_{\bar{l}\bar{q}},t^{\prime})\in \mathcal{N}^{2r}_k, t^{\prime}-t\geqslant \bar{q} \}$ represents entering the rightmost cell $c_{\bar{l}\bar{q}}$ of the RL from the leftmost cell $c_{l1}$ of lane $l$. An \textit{exit-RL arc} in $\mathcal{A}^{Rl}_k=\{(c_{\bar{l}\bar{q}},t;c_{l1},t^{\prime})\mid (c_{\bar{l}\bar{q}},t)\in \mathcal{N}^{2r}_k, (c_{l1},t^{\prime})\in \mathcal{N}^{2l}_k, t^{\prime}-t\geqslant \bar{q}\}$ signifies exiting from the rightmost cell $c_{\bar{l}\bar{q}}$ of the RL into the leftmost cell $c_{l1}$ of lane $l$. Collecting these arcs over all forward lanes yields the \textit{macro-pull-out}, \textit{enter-RL}, and \textit{exit-RL arc} sets $\mathcal{A}^{2e}_k=\bigcup_{l\in L^f}{\mathcal{A}^{le}_k}$, $\mathcal{A}^{FR}_k=\bigcup_{l\in L^f}{\mathcal{A}^{lR}_k}$, and $\mathcal{A}^{RF}_k=\bigcup_{l\in L^f}{\mathcal{A}^{Rl}_k}$. %By combining these arc subsets across all forward lanes, the \textit{macro-pull-out}, \textit{enter-RL}, and \textit{exit-RL arc} sets $\mathcal{A}^{2e}_k=\bigcup_{l\in L^f}{\mathcal{A}^{le}_k}$, $\mathcal{A}^{FR}_k=\bigcup_{l\in L^f}{\mathcal{A}^{lR}_k}$, and $\mathcal{A}^{RF}_k=\bigcup_{l\in L^f}{\mathcal{A}^{Rl}_k}$ are obtained.
Along with the unchanged \textit{pull-in arcs} $\mathcal{A}^{s}_k$ from $\mathcal{G}^1_k$, these arcs form the arc set $\mathcal{A}^2_k$ of $\mathcal{G}^2_k$, given by: $\mathcal{A}^{2}_k=\mathcal{A}^{s}_k\cup \mathcal{A}^{2e}_k\cup \mathcal{A}^{FR}_k\cup \mathcal{A}^{RF}_k$. Figure \ref{fig3}(a) illustrates the $\mathcal{G}^2_1$ built for car 1 in Figure \ref{fig2}.
%%%For car 3 in Figure 1(a), Figure 2(b) shows the G2 constructed for it.

%%%By combining these arc subsets over all forward lanes, the \textit{macro-pull-out}, \textit{enter-RL}, and \textit{exit-RL arc} sets $\mathcal{A}^{2e}_k=\bigcup_{l\in L^f}{\mathcal{A}^{le}_k}$, $\mathcal{A}^{FR}_k=\bigcup_{l\in L^f}{\mathcal{A}^{lR}_k}$, and $\mathcal{A}^{RF}_k=\bigcup_{l\in L^f}{\mathcal{A}^{Rl}_k}$ are formed. Together with the unchanged \textit{pull-in arcs} $\mathcal{A}^{s}_k$ %$\mathcal{A}^{s}_k=\{(s,t;c_{l1},t+1)\mid (s,t)\in \mathcal{N}^s_k,(c_{l1},t+1)\in \mathcal{N}^{2f}_k\}$
%%%from $\mathcal{G}^1_k$, these arcs form the arc set $\mathcal{A}^{2}_k=\mathcal{A}^{s}_k\cup \mathcal{A}^{2e}_k\cup \mathcal{A}^{FR}_k\cup \mathcal{A}^{RF}_k$.

\begin{figure}[htbp]
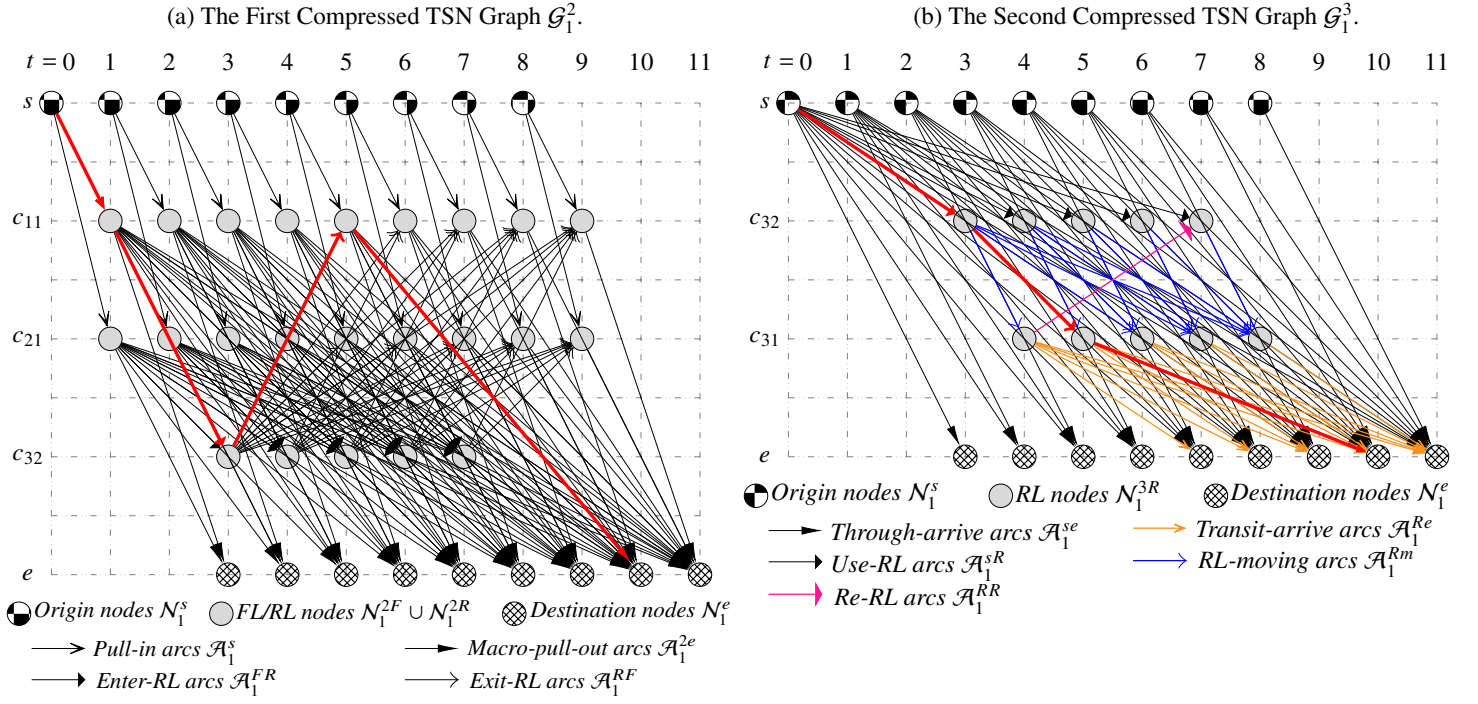

    \centering

 \noindent\makebox[\linewidth][c]{
    \centering
    \scalebox{0.78}{
}}

    \caption{Illustration of the First and the Second Compressed TSN Graphs for Car 1 in Figure \ref{fig2}. \label{fig3}}

\end{figure}

Based on $\mathcal{G}^2_k$, the feasible \textit{path} $r\in \Omega_k$ for each car $k\in V$ starts with a \textit{pull-in arc} in $\mathcal{A}^s_k$ and ends with a \textit{macro-pull-out arc} in $\mathcal{A}^{2e}_k$, with a series of \textit{enter-RL} and \textit{exit-RL arcs} in between. In addition to the parameters $s^k_r$, $t^k_r$, and $\zeta^k_r$ inherited from feasible \textit{paths} in $\mathcal{G}^1_k$, such a \textit{path} $r$ in $\mathcal{G}^2_k$ is further characterized by four binary indicators: \textup{(i)} $\rho^{kr}_{It}=1$ if car $k$ arrives at the leftmost cell of a forward lane at timestamp $t\in T$, \textup{(ii)} $\rho^{kr}_{Ot}=1$ if car $k$ arrives at the downstream shop or the rightmost cell of the RL at timestamp $t\in T$, \textup{(iii)} $\rho^{kr}_{Rt}=1$ if car $k$ is in the RL at timestamp $t\in T$, and \textup{(iv)} $\rho^{kr}_{lt}=1$ if car $k$ is in forward lane $l\in L^f$ at timestamp $t\in T$; otherwise, these indicators are 0. With these definitions and the same decision variables $x^k_r$ and $z^k_n$ as in $\mathcal{M}_{1}$, the first reduced MP, $\mathcal{M}_2$, is formulated as follows.
%the first reduced variant of $\mathcal{M}_{1}$, denoted by $\mathcal{M}_{2}$, is formulated as the following IP model.
%\vspace{-0.7cm}
\begin{subequations}\label{RMP2}
\begin{align}
  (\mathcal{M}_2)\,\ \min&\,\ \eqref{m1-obj} \nonumber \\%[-3pt]
        s.t.&\quad \nonumber \\
        &\,\ \eqref{m1-con1}-\eqref{m1-con3}, \eqref{m1-con7}-\eqref{m1-con10}, \nonumber \\%[-3pt]
        \label{m2-con2}&\,\ \sum_{k\in V}{\sum_{r\in \Omega_k}{\rho^{kr}_{It} \cdot x^k_r}}\leqslant 1, &&t\in \{1,\dots,\tau-\bar{q}\}, \\
        %\label{m2-con2}&\,\ \sum_{k\in V}{\sum_{r\in \Omega_k}{\rho^{kr}_{It} \cdot x^k_r}}\leqslant 1, \nonumber \\%[-15pt]
        %\label{m2-con2}&\qquad \qquad \qquad \quad \quad t\in \{1,\dots,\tau-\bar{q}\}, \\%[-3pt]
        \label{m2-con3}&\,\ \sum_{k\in V}{\sum_{r\in \Omega_k}{\rho^{kr}_{Ot} \cdot x^k_r}}\leqslant 1, &&t\in \{\tau^m,\dots,\tau\}, \\%[-3pt]
        \label{m2-con4}&\,\ \sum_{k\in V}{\sum_{r\in \Omega_k}{\rho^{kr}_{Rt}\cdot x^k_r}}\leqslant \bar{q}, &&t\in \{\tau^m,\dots,\tau-\tau^m\}, \\
        \label{m2-con5}&\,\ \sum_{k\in V}{\sum_{r\in \Omega_k}{\rho^{kr}_{lt}\cdot x^k_r}}\leqslant \bar{q}, && l\in L^f, t\in \{1,\dots,\tau-1\}.
\end{align}
%%%%%\setcounter{RMPone}{\value{equation}}
%record how many equations are put in this "align"
\end{subequations}

Constraints \eqref{m2-con2}/\eqref{m2-con3} correspond to constraints \eqref{m1-con5}/\eqref{m1-con6} in $\mathcal{M}_1$, ensuring that at most one car reaches/leaves a forward lane per timestamp. Constraints \eqref{m2-con4}/\eqref{m2-con5} impose lane capacity limits on the RL/each forward lane such that the number of cars in that lane never exceeds $\bar{q}$. As a relaxation of $\mathcal{M}_1$, $\mathcal{M}_2$ does not consider FIFO rules. The LP of $\mathcal{M}_2$ defined over a small set $\Omega^{\prime}_k\subset \Omega_k$ for each car $k\in V$, serves as the first reduced RMP (also the second RMP) in our BP-CP, denoted by $\mathcal{M}^{\text{LP}}_2$.
%%As a relaxation of $\mathcal{M}_1$, $\mathcal{M}_2$ does not enforce FIFO rules when cars leave lanes. Similarly, the LP of $\mathcal{M}_2$ defined over a small set $\Omega^{\prime}_k\subset \Omega_k$ for each car $k\in V$ serves as the first reduced RMP (also the second RMP) in our BP-CP, denoted by $\mathcal{M}^{\text{LP}}_2$. Similarly, the restricted $\mathcal{M}_2$ using a smaller \textit{path} set $\Omega^{\prime}_k$ for each car $k\in V$ serves as the first reduced RMP (also the second RMP) in our BP-CP. Its linear relaxation without integrality requirements on variables is denoted by $\mathcal{M}^{\text{LP}}_2$.

%%%Corresponding to constraints \eqref{m1-con5} and \eqref{m1-con6} in $\mathcal{M}_1$, constraints \eqref{m2-con2} and \eqref{m2-con3} in $\mathcal{M}_2$ ensure that at most one car can enter and leave a forward lane per timestamp, respectively. Instead of incorporating constraints \eqref{m1-con4}, lane capacity is enforced in $\mathcal{M}_2$ through constraints \eqref{m2-con4} and \eqref{m2-con5}, which require that the number of cars in any lane never exceeds $\bar{q}$. As a relaxation of $\mathcal{M}_1$, $\mathcal{M}_2$ allows cars to ignore FIFO rules when leaving lanes.

%%%By replacing $\Omega_k$ with its smaller subset $\Omega^{'}_k$ and dropping the integrality constraints \eqref{m1-con10}, we obtain the first reduced RMP, which is also the second RMP in our BP-CP, denoted by $\mathcal{M}^{\text{LP}}_2$.

%\subsubsection{Second reduced RMP}
\subsubsection{Second reduced MP}
We further reduce $\mathcal{M}_2$ to mitigate symmetry caused by homogeneous forward lanes. The idea is to track only whether and when each car $k\in V$ enters and exits the RL, without distinguishing the specific forward lanes it enters. From the RL entry and exit timestamps, we can infer when the car departs from the preceding forward lane and arrives at the subsequent one. As a result, all nodes and arcs related to forward lanes can be removed from $\mathcal{G}^2_k$.

%We further reduce $\mathcal{M}_2$ and compress $\mathcal{G}^2_k$ for each car $k\in V$ to mitigate the symmetry effects caused by homogeneous forward lanes. The idea is to focus only on whether and when car $k$ enters and exits the RL, without distinguishing which forward lanes it enters. Once the timestamps for entering and exiting the RL are known, the departure timestamp from the preceding forward lane and the arrival timestamps at the subsequent forward lane can be inferred. As a result, it is possible to remove all nodes and arcs related to forward lanes from $\mathcal{G}^2_k$.

Let $\mathcal{G}^3_k=(\mathcal{N}^3_k,\mathcal{A}^3_k)$ denote the second compressed TSN graph for car $k\in V$. The node set $\mathcal{N}^3_k$ retains all \textit{origin nodes} $\mathcal{N}^s_k$, \textit{destination nodes} $\mathcal{N}^e_k$, and \textit{RL-nodes} $\mathcal{N}^{2R}_k$ from the original $\mathcal{G}^2_k$. Beyond these, it features those \textit{RL-nodes} from $\mathcal{G}^1_k$ associated with the leftmost cell of the RL to indicate car $k$ exiting the RL, %exiting the RL for car $k$,
called \textit{RL-left nodes} $\mathcal{N}^{3R}_k=\{(c_{\bar{l}1},t)\in \mathcal{N}^{1r}_k\}$. Formally, $\mathcal{N}^3_k=\mathcal{N}^s_k\cup \mathcal{N}^e_k\cup \mathcal{N}^{2R}_k\cup \mathcal{N}^{3R}_k$.

%%%%Let $\mathcal{G}^3_k=(\mathcal{N}^3_k,\mathcal{A}^3_k)$ denote the second compressed TSN graph for car $k$. The node set $\mathcal{N}^3_k$ includes all \textit{origin} and \textit{destination nodes} ($\mathcal{N}^s_k$ and $\mathcal{N}^e_k$) from graph $\mathcal{G}^1_k$ (or $\mathcal{G}^2_k$). Besides these, other nodes in $\mathcal{N}^3_k$ are divided into two subsets to identify when car $k$ enters and exits the RL:

%%%%\begin{itemize}
%%%%    \item \textit{RL-last nodes} $\mathcal{N}^{3R\bar{q}}_{k}=\{(c_{\bar{l}q},t)\in \mathcal{N}^{1R}_k\mid q=\bar{q}\}$: indicate entry into the last cell $c_{\bar{l}\bar{q}}$ of the RL.

%%%%    \item \textit{RL-first nodes} $\mathcal{N}^{3R1}_{k}=\{(c_{\bar{l}q},t)\in \mathcal{N}^{1R}_k\mid q=1\}$: indicate exit from the first cell $c_{\bar{l}1}$ of the RL.

    %\item $\mathcal{N}^{R}_{1k}=\{(c_{\bar{l}1},t)\in C\times T | t=\tau^s_k+\tau^r, \dots, \tau^e_k-\tau^d\}$: set of \textit{first-cell nodes} in $\mathcal{G}^3_k$.

    %\item $\mathcal{N}^R_{\bar{q}k}=\{(c_{\bar{l}\bar{q}},t)\in C\times T | t=\tau^s_k+\tau^d, \dots, \tau^e_k-\tau^r\}$: set of \textit{last-cell nodes} in $\mathcal{G}^3_k$.
%%%%\end{itemize}

%With $\mathcal{N}^3_k$, the arc set $\mathcal{A}^3_k$ models the upstream-to-downstream movement of car $k$ using the number of RL entries and \textcolor{red}{the associated RL entry and exit timestamps.}
With $\mathcal{N}^3_k$, the arc set $\mathcal{A}^3_k$ models the upstream-to-downstream movement of car $k$ using the number of times it enters the RL and the associated RL entry and exit timestamps. Specifically, a feasible \textit{path} $r\in \Omega_k$ for car $k$ in $\mathcal{G}^3_k$ has two forms. One is a single \textit{through-arrive arc} $\mathcal{A}^{se}_k=\{(s,t;e,t^{\prime})\mid (s,t)\in \mathcal{N}^s_k,(e,t^{\prime})\in \mathcal{N}^e_k,t^{\prime}-t\geqslant \tau^m\}$, representing a direct arrival at the downstream shop without passing through the RL. The other is a sequence of consecutive arcs indicating an arrival after using the RL. The other is a sequence of consecutive arcs indicating an arrival after using the RL. The latter includes \textit{use-RL arcs} $\mathcal{A}^{sR}_k=\{(s,t;c_{\bar{l}\bar{q}},t^{\prime})\mid (s,t)\in \mathcal{N}^s_k,(c_{\bar{l}\bar{q}},t^{\prime})\in \mathcal{N}^{2R}_k,t^{\prime}-t\geqslant \tau^m\}$ for initial entry into the RL from the upstream shop, \textit{re-RL arcs} $\mathcal{A}^{RR}_k=\{(c_{\bar{l}1},t;c_{\bar{l}\bar{q}},t^{\prime})\mid (c_{\bar{l}1},t)\in \mathcal{N}^{3R}_k,(c_{\bar{l}\bar{q}},t^{\prime})\in \mathcal{N}^{2R}_k,t^{\prime}-t\geqslant \tau^m\}$ for re-entry into the RL after moving through a forward lane, \textit{transit-arrive arcs} $\mathcal{A}^{Re}_k=\{(c_{\bar{l}1},t;e,t^{\prime})\mid (c_{\bar{l}1},t)\in \mathcal{N}^{3R}_k,(e,t^{\prime})\in \mathcal{N}^e_k,t^{\prime}-t\geqslant \tau^m\}$ for final downstream-arrival after the last use of the RL, and \textit{RL-moving arcs} $\mathcal{A}^{Rm}_k=\{(c_{\bar{l}\bar{q}},t;c_{\bar{l}1},t^{\prime})\mid (c_{\bar{l}\bar{q}},t)\in \mathcal{N}^{2R}_k,(c_{\bar{l}1},t^{\prime})\in \mathcal{N}^{3R}_k,t^{\prime}-t\geqslant \bar{q}-1\}$ for movement within the RL.
%The latter includes \textit{use-RL arcs} $\mathcal{A}^{sR}_k=\{(s,t;c_{\bar{l}\bar{q}},t^{\prime})\mid (s,t)\in \mathcal{N}^s_k,(c_{\bar{l}\bar{q}},t^{\prime})\in \mathcal{N}^{2R}_k,t^{\prime}-t\geqslant \tau^m\}$ for initial entry into the RL from the upstream shop, \textit{re-RL arcs} $\mathcal{A}^{RR}_k=\{(c_{\bar{l}1},t;c_{\bar{l}\bar{q}},t^{\prime})\mid (c_{\bar{l}1},t)\in \mathcal{N}^{3R}_k,(c_{\bar{l}\bar{q}},t^{\prime})\in \mathcal{N}^{2R}_k,t^{\prime}-t\geqslant \tau^m\}$ for re-entry into the RL after moving through a forward lane, \textit{transit-arrive arcs} $\mathcal{A}^{Re}_k=\{(c_{\bar{l}1},t;e,t^{\prime})\mid (c_{\bar{l}1},t)\in \mathcal{N}^{3R}_k,(e,t^{\prime})\in \mathcal{N}^e_k,t^{\prime}-t\geqslant \tau^m\}$ for final arrival after using the RL, and \textit{RL-moving arcs} $\mathcal{A}^{Rm}_k=\{(c_{\bar{l}\bar{q}},t;c_{\bar{l}1},t^{\prime})\mid (c_{\bar{l}\bar{q}},t)\in \mathcal{N}^{2R}_k,(c_{\bar{l}1},t^{\prime})\in \mathcal{N}^{3R}_k,t^{\prime}-t\geqslant \bar{q}-1\}$ for movement within the RL. These arcs together define $\mathcal{A}_3=\mathcal{A}^{se}_k\cup \mathcal{A}^{sR}_k\cup \mathcal{A}^{RR}_k \cup \mathcal{A}^{Re}_k \cup \mathcal{A}^{Rm}_k$.
Correspondingly, Figure \ref{fig3}(b) illustrates the $\mathcal{G}^3_1$ built for the same car 1 in Figure \ref{fig2}.

Omitting all forward-lane structures in $\mathcal{G}^3_k$ graphs allows us to remove constraints \eqref{m2-con5} related to them from $\mathcal{M}_2$. The resulting formulation is $\mathcal{M}_3$, together with its linear relaxation defined on set $\Omega^{\prime}_{k}\subset \Omega_k$ ($\forall k\in V$). %The resulting model $\mathcal{M}_3$ and its linear relaxation with set $\Omega^{\prime}_{k}\subset \Omega_k$ ($\forall k\in V$) are the third MP and RMP in our BP-CP, respectively.
Because $\mathcal{M}_3$ excludes both FIFO and forward-lane capacity constraints, it yields the weakest lower bound for the CRSP-MS. Nonetheless, it has the shortest solution time and reduces computational effort at each BB node, as shown in Section \ref{sec5-2}.

For each RMP, the set $\Omega^{\prime}_{k}$ ($\forall k\in V$) initially contains a few feasible \textit{paths} in the corresponding TSN graph. During CG, it repeatedly incorporates new \textit{paths} that can lower the objective value until none exist. We next outline how to find these \textit{paths} using \textit{pricing} SPs.

\subsection{First Stage: Pricing Subproblems}\label{sec4-2}
At each CG iteration, $d$ \textit{pricing} SPs are solved using the dual-variable values from the current RMP solution. Each SP, associated with a car $k\in V$ and denoted by SP$_k$, is a \textit{shortest-path problem} on the TSN graph for that car. It seeks a \textit{path} with the minimum reduced cost from an \textit{origin node} in $\mathcal{N}^s_k$ to a \textit{destination node} in $\mathcal{N}^e_k$. If this \textit{path} has a negative cost, it is added to set $\Omega_k^{\prime}$. When no such \textit{paths} with negative reduced costs are found for any car in $V$, the CG completes with the optimal RMP solution. For all three RMP versions and their respective TSN graphs for car $k$, the reduced-cost calculation for a feasible \textit{path} in $\Omega_k$ is given in the online supplement.

To prepare for solving each SP$_k$, we add a dummy \textit{source node} connected to all \textit{origin nodes} and a dummy \textit{terminal node} connected to all \textit{destination nodes} with zero-cost arcs in every TSN graph for car $k$. The SP$_k$ is then equivalent to finding the shortest \textit{path} from the \textit{source} to the \textit{terminal}.

%%%%For each version of the RMP and its corresponding TSN graph for car $k\in V$, the calculation of the reduced cost for any \textit{path} $r\in \Omega_k$ is provided in the online supplement. To prepare for solving SP$_k$, we add a dummy \textit{source node} connected to all \textit{origin nodes} and a dummy \textit{terminal node} connected to all \textit{destination nodes} with zero-cost arcs in all three graph types for car $k$. Then, solving SP$_k$ becomes finding the shortest \textit{path} from the \textit{source} to the \textit{terminal node}.

In support of branching on the number of times car $k\in V$ enters the RL (see Section \ref{sec4-4}), we solve each SP$_k$ on three TSN graphs using a \textit{labeling algorithm}. This algorithm produces $\bar{n}_k+1$ shortest \textit{paths} from the \textit{source} to the \textit{terminal}, recalling that $\bar{n}_k$ is the maximum number of times car $k$ can enter the RL during resequencing. The $n$-th ($0\leqslant n\leqslant\bar{n}_k$) \textit{path} is the shortest among those along which car $k$ enters the RL exactly $n$ times. Details of this \textit{labeling algorithm} for our classical \textit{shortest-path} SP$_k$ are provided in the online supplement.

%%%In support of branching on the number of times car $k\in V$ enters the RL (as introduced in Section \ref{sec4-4}), we solve each SP$_k$ on three TSN graphs using a \textit{labeling algorithm}. This algorithm produces $\bar{n}_k+1$ shortest \textit{paths} from the \textit{source} to the \textit{terminal}, recalling that $\bar{n}_k$ is the maximum number of times car $k$ can enter the RL during resequencing. The $n$-th ($0\leqslant n\leqslant\bar{n}_k$) \textit{path} is the shortest among those along which car $k$ enters the RL $n$ times. %The $n$-th ($0\leqslant n\leqslant\bar{n}_k$) shortest \textit{path} indicates that car $k$ enters the RL $n$ times on that specific \textit{path}.
%%%Because our \textit{labeling algorithm} is designed for a classic \textit{single-source-single-sink shortest-path problem}, we detail it in the online supplement.

%%%After running the \textit{labeling algorithm} for each car $k\in V$, if the \textit{paths} obtained satisfy the branching rules in Section \ref{sec4-4} applied to the current BB node and have negative reduced costs, they are added to set $\Omega^{\prime}_k$. If there exists an updated $\Omega^{\prime}_k$ for a car $k\in V$, the RMP is then resolved; otherwise, the CG iterative process between $\mathcal{S}_1$ and $\mathcal{S}_2$ terminates.

After running the \textit{labeling algorithm} for each car $k\in V$, the \textit{paths} that have negative reduced costs and satisfy the branching rules in Section \ref{sec4-4} of the current BB node are added to set $\Omega^{\prime}_k$. If there exists an updated $\Omega^{\prime}_k$ for a car $k\in V$, the RMP is then solved again; otherwise, the CG process terminates.

%%%%Accounting for arc availability affected by branching rules, we perform a depth-first search to calculate the in-degree of each node before applying the \textit{labeling algorithm}. If the \textit{terminal} has zero in-degree, i.e., it is unreachable from the \textit{source}, then SP$_k$ is confirmed infeasible, allowing us to prune the current BB node.

\subsection{Acquisition of partial integer solution}\label{sec4-4}
%\subsection{Branching scheme}\label{sec4-4}
Upon completion of CG, we require integrality only for the RMP variables related to the makespan and the number of times each car enters the RL. The makespan integrality is ensured by removing \textit{incompatible paths} from set $\Omega^{\prime}_{\pi_d}$ of the last car $\pi_d$ and disabling the relevant arcs in its TSN graphs. For RL-entry counts, integrality is maintained through two branching rules.

%%%For the RMP solution at the end of the $\mathcal{S}_1$-$\mathcal{S}_2$ iteration, our BP-CP approach only requires the variables related to the makespan and the number of times each car enters the RL to be integers. If not, we enforce makespan integrality by removing incompatible \textit{paths} from $\Omega^{\prime}_{\pi_d}$ of the last car $\pi_d$ and marking relevant arcs in its TSN graphs as inaccessible. The integrality of the RL-entry counts is ensured through two branching rules.
%If either of these is fractional

\begin{itemize}
    %\item \textit{Removing incompatible columns and arcs.}
    \item \textit{Removing incompatible \textit{paths} and arcs.} The optimal RMP objective value $\widehat{m}$ implies a lower bound $\lceil \widehat{m}\rceil$ on the makespan. It follows that any \textit{paths} in $\Omega^{\prime}_{\pi_d}$ leading car $\pi_d$ to the downstream shop before $\lceil \widehat{m}\rceil$ are \textit{incompatible} and should be removed. To prevent their regeneration, downstream-arrival arcs $(i,t;e,t^{\prime})$ with $t^{\prime}<\lceil \widehat{m}\rceil$ are marked as inaccessible in the TSN graphs of car $\pi_d$, including \textit{pull-out arcs} in $\mathcal{G}^1_{\pi_d}$, \textit{macro-pull-out arcs} in $\mathcal{G}^2_{\pi_d}$, as well as \textit{through-arrive} and \textit{transit-arrive arcs} in $\mathcal{G}^3_{\pi_d}$. After these updates, CG restarts at the current BB node.

    %%%The optimal RMP objective value $\widehat{m}$ implies that the minimum makespan is at least $\lceil \widehat{m}\rceil$. Therefore, any \textit{paths} in $\Omega^{'}_{\pi_d}$ that lead car $\pi_d$ to the downstream shop before $\lceil \widehat{m}\rceil$ are \textit{incompatible} and should be removed from $\Omega^{'}_{\pi_d}$. To prevent their regeneration in $\mathcal{S}_2$, we mark the arrival-related arcs of the form $(i,t;e,t^{\prime})$, where $t^{\prime}<\lceil \widehat{m}\rceil$, as inaccessible in the TSN graphs of car $\pi_d$. %Specifically, these involve \textit{pull-out arcs} in $\mathcal{G}^{1}_{\pi_d}$, \textit{macro-pull-out arcs} in $\mathcal{G}^2_{\pi_d}$, as well as \textit{through-arrive} and \textit{transit-arrive arcs} in $\mathcal{G}^3_{\pi_d}$.
    %%%After these updates, the $\mathcal{S}_1$-$\mathcal{S}_2$ iteration at the current BB node restarts by solving the RMP.

    %The path-removal and arc-marking trigger no new BB nodes created, but a restart iteration between S1 and S2.
\end{itemize}

Once CG ends with no \textit{incompatible paths} in $\Omega^{\prime}_{\pi_d}$, the makespan $\widehat{m}$ is integral. We then verify and enforce the integrality of the RL-entry count $\widetilde{n}_k=\sum_{n\in N_k}{n\cdot \bar{z}^k_n}$ for each car $k\in V$, where $\bar{z}^k_n$ is the value of variable $z^k_n$ ($\forall n\in N_k$). First, the downstream sequence $\Pi^{D}$ is partitioned into $|V^{R}_{\times}|+1$ consecutive segments by the cars in $V^{R}_{\times}$, which serve as partition boundaries. These segments are then examined sequentially. For each, denote $S$ as the set of all cars in it with fractional RL-entry counts. %denote $S$ as the set of all cars within that segment that have fractional RL-entry counts.
Let $h_s=\sum_{k\in S}{\widetilde{n}_k}$ be the sum of RL-entry counts for cars in $S$. If $h_s$ is integral, we proceed to the next segment; otherwise, two branches are created using the following branching rules.

\begin{itemize}
    \item \textit{Branching on the total number of times cars in a segment enter the RL.} On the left branch, a constraint $\sum_{k\in S}{\sum_{n\in N_k}{n\cdot z^k_n}}\geqslant \lceil h_s\rceil$ is added to the RMP to ensure that the cars in $S$ enter the RL at least $\lceil h_s\rceil$ times in total, leaving set $\Omega^{\prime}_{k}$ and the \textit{pricing} SP$_k$ ($\forall k\in V$) unchanged. On the right branch, a constraint $\sum_{k\in S}{\sum_{n\in N_k}{n\cdot z^k_n}}\leqslant \lfloor h_s\rfloor$ is added to limit these cars to at most $\lfloor h_s \rfloor$ RL entries. For each car $k\in S$, only \textit{paths} with up to $\lfloor h_s \rfloor$ RL visits are retained in and added to $\Omega^{\prime}_{k}$. When $h_s=0$, all RL-related arcs in the TSN graphs of the car are disabled. We use Example \ref{ex1} to illustrate this branching procedure.
    %%%%On the left branch, a constraint $\sum_{k\in S}{\sum_{n\in N_k}{n\cdot z^k_n}}\geqslant \lceil h_s\rceil$ is added to the RMP to ensure that the cars in $S$ enter the RL at least $\lceil h_s\rceil$ times overall, without revising set $\Omega^{\prime}_{k}$ or the \textit{pricing} SP$_k$ ($\forall k\in V$). On the right branch, an added constraint $\sum_{k\in S}{\sum_{n}{n\cdot z^k_n}}\leqslant \lfloor h_s\rfloor$ restricts these cars to at most $\lfloor h_s \rfloor$ RL entries. Only \textit{paths} with up to $\lfloor h_s \rfloor$ RL entries can remain in set $\Omega^{\prime}_{k}$, $\forall k\in V$, and be added during the $\mathcal{S}_1$-$\mathcal{S}_2$ iteration. When $h_s=0$, all RL-related arcs in the TSN graphs are disabled.
\end{itemize}

\begin{example}\label{ex1}
    Given a downstream sequence $\Pi^D=[4,7,10,9,5,1,8,6,3,2]$ of 10 cars, the RL-entry counts $\widetilde{n}_k$ for each car $k\in V$ are: $\widetilde{n}_1 = 0.652777$, $\widetilde{n}_2 = 1$, $\widetilde{n}_3 = 1$, $\widetilde{n}_4 = 0$, $\widetilde{n}_5 = 0.222222$, $\widetilde{n}_6 = 0.674999$, $\widetilde{n}_7 = 0$, $\widetilde{n}_8 = 0.2$, $\widetilde{n}_9 = 0.125$, and $\widetilde{n}_{10}=0$. According to the partitioning procedure, $\Pi^D$ is divided into 4 segments by $V^R_{\times}=\{4,7,10\}$: $[4]$, $[7]$, $[10]$, and $[9,5,1,8,6,3,2]$. After removing the cars with integral $\widetilde{n}_k$ from each segment, we obtain the sets $\emptyset$, $\emptyset$, $\emptyset$, and $\{9,5,1,8,6\}$. %After removing the cars with integral $\widetilde{n}_k$ values from each segment, we obtain the corresponding sets $\emptyset$, $\emptyset$, $\emptyset$, and $\{9,5,1,8,6\}$.
    The first three have $h_s = 0$ and induce no branching. For the last set, $h_s = 1.874998$ (fractional). We hence perform branching on this set, yielding two branches $\sum_{k\in S}{\sum_{n\in N_k}{n\cdot z^k_n}}\geqslant 2$ and $\sum_{k\in S}{\sum_{n\in N_k}{n\cdot z^k_n}}\leqslant 1$.
    %According to our partition method, $\Pi^D$ is divided into 4 segments by $V^R_{\times}=\{4,7,10\}$: $\emptyset$, $\emptyset$, $\emptyset$, and $\{9,5,1,8,6\}$. Starting from the first $\emptyset$, we finally identify $\{9,5,1,8,6\}$ where the cars have fractional RL-entry counts of 1.874998 as set $S$, yielding two branches: $\sum_{k\in S}{\sum_{n\in N_k}{n\cdot z^k_n}}\geqslant 2$ and $\sum_{k\in S}{\sum_{n\in N_k}{n\cdot z^k_n}}\leqslant 1$.
\end{example}

If the above branching is not performed, we then check the integrality of the $\boldsymbol{z}$-variable values. For each car $k\in V$, we find the variable $z^{k}_{\widehat{n}_k}$ in the RMP solution that has the maximum value among those in set $\{z^k_n \mid n\in N_k\}$. If their values meet the \textit{cover-set inequality} $\sum_{k\in V}{z^k_{\widehat{n}_k}}\geqslant d-1$, then branching on the $\boldsymbol{z}$ variables is deemed unnecessary, and we obtain an \textit{RL-entry plan} specifying that car $k$ enters the RL $\widehat{n}_k$ times. If not, among all fractional $\boldsymbol{z}$ variables, we select the one $z^{k^{*}}_{n^{*}}$ corresponding to the car with the largest upstream index $k^{*}=\max\{k\in V\mid \exists n\in N_k,
\bar{z}^k_n \text{ is fractional}\}$ and the highest RL-entry count $n^*=\max\{n\in N_{k^*}\mid \bar{z}^{k^*}_n \text{ is fractional}\}$, and branch on it as follows.

%If the above branching is not performed, it then proceeds to examine the integrality of the $\boldsymbol{z}$-variable values. For each car $k\in V$, we identify the variable $z^k_{\widetilde{n}_k}$ in set $\{z^k_n\in \{0,1\}\mid n\in N_k\}$ with the maximum value in the RMP solution.
%%%%%%%If the above branching is not carried out, it then examines the integrality of the $\boldsymbol{z}$-variable values. %For each car $k\in V$, we find the variable $z^{k}_{\widehat{n}_k}$ with the maximum value in the RMP solution among the variables in set $\{z^k_n\in \{0,1\} \mid n\in N_k\}$.
%%%%%%%For each car $k\in V$, we find the variable $z^{k}_{\widehat{n}_k}$ in the RMP solution with the maximum value among the variables in set $\{z^k_n\in \{0,1\} \mid n\in N_k\}$. If their values meet the \textit{cover-set inequality} $\sum_{k\in V}{z^k_{\widehat{n}_k}}\geqslant d-1$, then branching on $\boldsymbol{z}$ variables is deemed unnecessary, and we obtain an \textit{RL-entry plan} specifying that car $k$ enters the RL $\widehat{n}_k$ times. If not, among all fractional $\boldsymbol{z}$ variables, we select the one $z^{k^{*}}_{n^{*}}$ corresponding to the car with the largest upstream index $k^{*}=\max\{k\in V\mid \exists n\in N_k,
%%%%%%%\bar{z}^k_n \text{ is fractional}\}$ and the highest RL-entry count $n^*=\max\{n\in N_{k^*}\mid \bar{z}^{k^*}_n \text{ is fractional}\}$, and branch on it as follows.

\begin{itemize}
    \item \textit{Branching on the $\boldsymbol{z}$ variables.}
    Branching on variable $z^{k^*}_{n^*}$ yields two BB nodes by fixing $z^{k^*}_{n^*}=1$ or $z^{k^*}_{n^*}=0$ in the RMP. Under $z^{k^*}_{n^*}=1$, only \textit{paths} with exactly $n^*$ RL visits remain in set $\Omega^{\prime}_{k^*}$. If $n^{*}=0$, all RL-related arcs in the TSN graphs of car $k^*$ are disabled. Under $z^{k^*}_{n^*}=0$, all \textit{paths} with exactly $n^*$ RL visits are excluded from $\Omega^{\prime}_{k^*}$ and barred from future addition. No changes to the TSN graphs are needed, but all \textit{through-arrive arcs} in $\mathcal{G}^3_{k^*}$ are inaccessible when $n^{*}=0$.

    %%Branching on variable $z^{k^*}_{n^*}$ yields two BB nodes by fixing $z^{k^*}_{n^*}=1$ or $z^{k^*}_{n^*}=0$ in the RMP. Under $z^{k^*}_{n^*}=1$, only \textit{paths} with exactly $n^*$ RL entries are included in set $\Omega^{\prime}_{k^*}$. If $n^{*}=0$, all RL-related arcs are disabled in the graphs. Under $z^{k^*}_{n^*}=0$, all \textit{paths} with exactly $n^*$ RL-entries are excluded from $\Omega^{\prime}_{k^*}$ and prevented from future addition. No structural changes to the graphs, but all \textit{through-arrive arcs} in $\mathcal{G}^3_{k^*}$ are made inaccessible when $n^{*}=0$.
\end{itemize}

If no branching occurs up to this point, an integer makespan $\widehat{m}$ and an \textit{RL-entry plan} $u=\{(1,\widehat{n}_1),\dots,(k,\widehat{n}_k),\dots,(d,\widehat{n}_d)\}$ are obtained, where each pair $(k,\widehat{n}_k)$ indicates that car $k\in V$ enters the RL $\widehat{n}_k\in N_k$ times. The remaining decisions for the CRSP-MS are then obtained by solving a feasibility problem in the second stage $\mathcal{S}_2$, as explained in the next section.

%%%If no BB nodes are generated up this point, an integer makespan $\widehat{m}$ and an \textit{RL-entry plan} $u=\{(1,\widehat{n}_1),\dots,(k,\widehat{n}_k),\dots,(d,\widehat{n}_d)\}$ are obtained, where each pair $(k,\widehat{n}_k)$ indicates that car $k\in V$ enters the RL $\widehat{n}_k\in N_k$ times. The remaining decisions for the CRSP-MS are then completed in BP-CP by solving a feasibility problem in the third stage $\mathcal{S}_3$, as explained in the next section.

\subsection{Second stage: Constraint Programming for the Car-to-Lane Assignment Problem}\label{sec4-3}
%The feasibility problem in $\mathcal{S}_3$ assigns each car to a specific forward lane for each time it needs to enter, along with the corresponding entry and exit timestamps.
%The feasibility problem in $\mathcal{S}_2$ assigns each car to a forward lane at each time it needs to enter, and determines the corresponding lane-entry and lane-exit timestamps. To model it, we create $\widehat{n}_k+1$ \textit{dummy cars} for each car $k\in V$ if it enters the RL $\widehat{n}_k$ times, i.e., for each pair $(k,\widehat{n}_k)\in u$, where the $n$-th ($1\leqslant n\leqslant \widehat{n}_k+1$) \textit{dummy car}, denoted by $o^k_n$, represents the $n$-th time car $k$ enters a forward lane.

The feasibility problem in $\mathcal{S}_2$ assigns each car to a specific forward lane at each time it needs to enter, and determines the corresponding entry and exit timestamps. To model it, we create $\widehat{n}_k+1$ \textit{dummy cars} for each car $k\in V$ if it enters the RL $\widehat{n}_k$ times (i.e., for each pair $(k,\widehat{n}_k)\in u$), where the $n$-th ($1\leqslant n\leqslant \widehat{n}_k+1$) \textit{dummy car}, denoted by $o^k_n$, indicates the $n$-th time car $k$ enters a forward lane.

For each car $k\in V$, let $\mathcal{D}^k=\{o^k_1,\dots,o^k_n,\dots,o^k_{\widehat{n}_k+1}\}$ be the ordered set of \textit{dummy cars} created for it. The first \textit{dummy car} ($o^k_1$) represents its initial entry into a forward lane from the upstream shop, and the last one ($o^k_{\widehat{n}_k+1}$) the final entry before reaching the downstream shop. Each intermediate \textit{dummy car} $o^k_n$, where $2\leqslant n\leqslant \widehat{n}_k$, corresponds to an entry from the RL into a forward lane, followed by a subsequent re-entry into the RL. The set of all \textit{dummy cars} is denoted by $\mathcal{D}=\bigcup_{k\in V}{\mathcal{D}^k}$.

%%%Let $\mathcal{D}^k=\{o^k_1,\dots,o^k_n,\dots,o^k_{\widehat{n}_k+1}\}$ be the ordered set of all \textit{dummy cars} for each car $k\in V$, and $\mathcal{D}=\bigcup_{k\in V}{\mathcal{D}^k}$ the set of all such \textit{dummy cars}. For each car $k$, the first \textit{dummy car} ($o^k_1$) represents its entry into a forward lane from the upstream shop, and the last one ($o^k_{\widehat{n}_k+1}$) corresponds to its final exit through a forward lane to the downstream shop. Each intermediate \textit{dummy car} $o^k_n$, where $1<n\leqslant\widehat{n}_k$, indicates that car $k$ enters a forward lane from the RL and will later re-enter the RL. If $\widehat{n}_k=0$, the only \textit{dummy car} in $\mathcal{D}^k$, i.e., $o^k_1$, is car $k$ itself, which never enters the RL.

We formulate this car-to-lane assignment problem as a CP model. Each \textit{dummy car} is represented by an \textit{interval decision variable} in the IBM CP Optimizer, with \textit{start} and \textit{end times} denoting the timestamps at which the car enters and is scheduled to exit a lane. %The \textit{start} and \textit{end times} of this variable correspond to the timestamps marking when the car enters and is scheduled to exit a lane, respectively.
To capture lane assignment, we also define an \textit{optional} lane-specific \textit{interval variable} for each \textit{dummy car}. Its absence in the CP solution indicates that the \textit{dummy car} is not assigned to the related forward lane. For each car $k\in V$, all decision variables in the model are defined as follows. (The domains of integer variables and of the \textit{start} and \textit{end times} of \textit{interval variables} are given in the online supplement.)
%For \textit{interval variables}, the value ranges for their \textit{start} and \textit{end times} are provided in the online supplement.
%If this variable is determined to be absent in the CP solution, it indicates that the dummy car is not assigned to the related forward lane. All decision variables used in the model are defined as follows. For interval variables, the value ranges for their start and end times are provided in the online supplement.

%%%%We formulate the assignment problem as a CP model. Each \textit{dummy car}, associated with two timestamps marking when it enters and leaves a forward lane, can be represented as an \textit{interval variable} in IBM CP Optimizer. The \textit{start time} and \textit{end time} of this variable correspond to the timestamps when the car arrives at the first cell of a forward lane and when it is scheduled to leave that lane, respectively. The \textit{duration} between these times, i.e., the length of the variable, can be used to impose lane-capacity constraints.

%%%%To model lane assignment, a lane-specific \textit{interval variable} can be defined as \textit{optional} for each \textit{dummy car}, indicating whether that car is assigned to a specific lane. If such a variable is determined to be absent in the solution of the CP model, then the associated \textit{dummy car} does not enter that lane. In this case, all constraints applied to that variable become inactive.

%%%%As a result, the decision variables used in our assignment CP model are defined as below:

\begin{itemize}[leftmargin=0pt,itemsep=2pt,topsep=2pt]
    \item $y^k_n$: a \textit{present interval variable} for \textit{dummy car} $o^k_n\in \mathcal{D}^k$, %$\forall k\in V$, %with \textit{start time} defined in X1 and \textit{end time} defined in X2.

    \item $\widetilde{y}^{l}_{kn}$: an \textit{optional interval variable} indicates assigning \textit{dummy car} $o^k_n\in \mathcal{D}^k$ to forward lane $l\in L^f$,

    %%\item $\widetilde{y}^{l}_{kn}$: an \textit{optional interval variable} indicates that \textit{dummy car} $o^k_n\in \mathcal{D}^k$ is assigned to forward lane $l\in L^f$, %$\forall k\in V$, %with \textit{start time} defined in X1 and \textit{end time} defined in X2.

    \item $g^k_n$: a \textit{present interval variable} denotes the $n$-th ($n\in \{1,\dots,\widehat{n}_k\}$) entry of car $k$ into the RL,

    %\item $g^k_n$: a \textit{present interval variable} denotes the $n$-th entry of \textit{dummy car} $o^k_n\in \mathcal{D}^k\backslash \{o^k_{\widehat{n}_k+1}\}$ into the RL, %$\forall k\in V$, %with \textit{start time} defined in X1 and \textit{end time} defined in X2.

    \item $\phi^{in}_{kn}$: an integer variable represents the arrival timestamp of \textit{dummy car} $o^k_n\in \mathcal{D}^k$ at the leftmost cell of a forward lane, %$\forall k\in V$, %defined between x1 and x2.

    \item $\phi^{out}_{kn}$: an integer variable represents the departure timestamp of \textit{dummy car} $o^k_n\in \mathcal{D}^k$ from the rightmost cell of a forward lane, %$\forall k\in V$.%, defined between x1 and x2.
\end{itemize}

The car-to-lane assignment CP model ($\mathcal{M}^{\text{CP}}$) is then written as follows.

%%The car-to-lane assignment problem can then be written as the following CP model ($\mathcal{M}^{\text{CP}}$).

\vspace{-0.5cm}

\begin{subequations}\label{CPModel}
\begin{align}
  \label{cp-con1}(\mathcal{M}^{\text{CP}})\,\ &\texttt{StartOf}(y^k_n)=\phi^{in}_{kn}, &&o^k_n \in \mathcal{D}^{k}, k\in V, \\
        \label{cp-con2}&\texttt{EndOf}(y^k_n)=\phi^{out}_{kn}, &&o^k_n \in \mathcal{D}^{k}, k\in V, \\
        \label{cp-con3}&\texttt{AllDiff}\big( \{\phi^{in}_{kn}\mid o^k_n \in \mathcal{D}^{k}, k\in V \} \big), \\
        \label{cp-con4}&\texttt{AllDiff}\big( \{ \phi^{out}_{kn}\mid o^k_n \in \mathcal{D}^{k}, k\in V \} \big), \\
        \label{cp-con5}&\texttt{Alternative}\big(y^k_n, \{\widetilde{y}^l_{kn}\mid l\in L^f  \} \big), &&o^k_n \in \mathcal{D}^{k}, k\in V, \\
        \label{cp-con6}&\texttt{EndAtStart}\big(y^k_n, g^k_n, 1 \big), &&o^k_n \in \mathcal{D}^k \backslash \{o^k_{\widehat{n}_k+1}\}, k\in V, \\
        \label{cp-con7}&\texttt{EndAtStart}\big(g^k_n, y^k_{n+1}, 1 \big), &&o^k_n \in \mathcal{D}^k \backslash \{o^k_{\widehat{n}_k+1}\}, k\in V,\\
        \label{cp-con8}&\texttt{StartBeforeStart}\big(y^k_1, y^{k+1}_1,1 \big), &&k\in V\backslash \{d\}, \\
        \label{cp-con9}&\texttt{EndBeforeEnd}\big(y^{\pi_p}_{\widehat{n}_{\pi_p}+1}, y^{\pi_{p+1}}_{\widehat{n}_{\pi_{p+1}}+1}, 1 \big), &&p\in P\backslash \{d\}, \\
        \label{cp-con10}&\texttt{EndBeforeStart}\big( y^k_n,y^k_{n+1}, \tau^m \big), &&o^k_n \in \mathcal{D}^k \backslash \{o^k_{\widehat{n}_k+1}\}, k\in V, \\
        \label{cp-con11}&\sum_{k\in V}{\sum_{o^k_n \in \mathcal{D}^k}{ \texttt{Pulse}\big(\widetilde{y}^l_{kn}, 1 \big)}}\leqslant \bar{q}, &&l\in L^f, \\
        \label{cp-con12}&\sum_{k\in V}{\sum_{o^k_n \in \mathcal{D}^k \backslash \{o^k_{\widehat{n}_k+1}\}}{\texttt{Pulse}\big(g^k_n, 1 \big)}}\leqslant \bar{q}, \\
        &\texttt{If} \,\ \big(\phi^{in}_{k,n_k} - \phi^{in}_{v,n_v} \big) \times \big(\phi^{out}_{k,n_k} - \phi^{out}_{v,n_v} \big)<0 \,\ \Longrightarrow \nonumber \\%[-5pt]
        \label{cp-con13}&\texttt{PresentOf}\big(\widetilde{y}^{l}_{k,n_k} \big) + \texttt{PresentOf}(\widetilde{y}^{l}_{v,n_v}) \leqslant 1, &&l\in L^f, o^k_{n_k}, o^v_{n_v} \in \mathcal{D}, \\
        \label{cp-con14}&\big(\texttt{StartOf}(y^k_{n_k+1}) < \texttt{StartOf}(y^v_{n_v+1}) \big)= \nonumber \\%[-5pt]
        &\big(\texttt{EndOf}(y^k_{n_k}) < \texttt{EndOf}(y^v_{n_v}) \big), &&k<v\in V, n_k=1,\dots,\widehat{n}_k, n_v=1,\dots,\widehat{n}_v.
\end{align}
%%%%\setcounter{RMPone}{\value{equation}}
%record how many equations are put in this "align"
\end{subequations}

\vspace{-6mm}

For each \textit{dummy car}, constraints \eqref{cp-con1} and \eqref{cp-con2} synchronize the \textit{start} and \textit{end times} of its \textit{interval variables} with the corresponding integer variables. %Constraints \eqref{cp-con3} and \eqref{cp-con4} ensure distinct lane-entry and lane-exit timestamps so that no two cars enter (exit) a forward lane at the same timestamp.
Constraints \eqref{cp-con3}/\eqref{cp-con4} ensure distinct lane-entry/lane-exit timestamps so that no two cars enter/exit a forward lane at the same timestamp. Constraints \eqref{cp-con5} select exactly one \textit{optional interval variable} from set $\{\widetilde{y}^l_{kn}\mid l\in L^f\}$ to be \textit{present} for each \textit{dummy car} $o^k_n\in \mathcal{D}$. This variable shares attributes with variable $y^k_n$, and its lane index $l\in L^f$ determines the assigned forward lane. Constraints \eqref{cp-con6}/\eqref{cp-con7} specify that the $n$-th exit timestamp of car $o^k_n$ from a forward lane/the RL equals the $n$-th/$n+1$-th entry timestamp into the RL/a forward lane minus one TU. Constraints \eqref{cp-con8} and \eqref{cp-con9} maintain the given upstream and downstream sequences, with at least one TU of separation between consecutive cars. For successive forward-lane entries of car $k\in V$, constraints \eqref{cp-con10} ensure the required transit time ($\tau^m$ TUs) through the RL. Constraints \eqref{cp-con11} and \eqref{cp-con12} impose lane capacity limits using function $\texttt{Pluse($a,1$)}$, which returns 1 for any timestamp within the interval defined by variable $a$. Their left-hand sides count the number of cars in each lane at any timestamp. To respect FIFO rules, constraints \eqref{cp-con13} forbid two cars from using the same forward lane if their entry and exit orders are opposite. Function $\texttt{PresentOf($a$)}$ returns 1 only if the \textit{optional} variable $a$ is determined to be \textit{present}. Constraints \eqref{cp-con14} ensure consistency between the orders in which cars enter the forward lanes and their previous exits from the RL.

A feasible $\mathcal{M}^{\text{CP}}$ solution yields a CRSP-MS solution. If $\mathcal{M}^{\text{CP}}$ is infeasible, the makespan $\widehat{m}$ along with the RL-entry plan $u$ is confirmed infeasible. Any shorter makespan $t<\widehat{m}$ paired with this $u$ is also invalid because it further reduces the \textit{interval}-\textit{variable} domains. Consequently, the infeasible pair $(\widehat{m},u)$ can be eliminated by adding the following lifted \textit{no-good cut} \eqref{cut2} to the MP, where $\Omega_{\pi_d}^{\prime}(t)$ denotes the set of all \textit{paths} for the last car $\pi_d$ with a downstream-arrival timestamp $t$.

%%%If model $\mathcal{M}^{\text{CP}}$ yields a feasible assignment, then we obtain a CRSP solution. Otherwise, if infeasible, the combination of the given makespan $\widehat{m}$ and RL-usage plan $u$ is removed by adding the following valid cut \eqref{cut1} to all three RMPs.

%%%\begin{equation}\label{cut1}
%%%    \sum_{r\in \Omega^{'}_{\pi_d}(\widehat{m})}{x^{\pi_d}_r}+\sum_{(k,\widehat{n}_k)\in u}{z^k_{\widehat{n}_k}}\leqslant d
%%%\end{equation}
%%%\noindent where set $\Omega^{'}_{k^*}(\widehat{m})$ comprises all \textit{paths} by which the last car $\pi_d$ arrives at the downstream shop at timestamp $\widehat{m}$. Cut \eqref{cut1} acts as a standard \textit{no-good cut} and can be enhanced in our case by including all \textit{paths} where car $\pi_d$ arrives earlier than $\widehat{m}$. The reason is that a shorter makespan reduces the domain of \textit{interval variables}. The lifted cut is written as follows:

%\vspace{-3mm}
\begin{equation}\label{cut2}
    \sum_{t\leqslant \widehat{m}}{\sum_{r\in \Omega^{\prime}_{\pi_d}(t)}{x^{\pi_d}_r}}+\sum_{(k,\widehat{n}_k)\in u}{z^{k}_{\widehat{n}_k}}\leqslant d
\end{equation}

Once a cut is added, the RMP is solved again using CG.
%Once a cut is added, the RMPs are solved again, along with the restarted $\mathcal{S}_1$-$\mathcal{S}_2$ iteration.

\subsection{Enhancements}\label{sec4-5}
%\subsection{Probing techniques}\label{sec4-5}
%\subsection{Preprocessing}\label{sec4-5}
This section presents two enhancements for BP-CP: a valid lower bound and a preprocessing scheme with an associated valid inequality for the MP. The latter selects \textit{promising cars} likely to use the RL that help probe CRSP-MS solutions both before and during BP-CP execution.
%for probing CRSP-MS solutions both before and during BP-CP execution.

%%%%%This section introduces three algorithmic enhancements for BP-CP: a valid lower bound for the CRSP and two preprocessing techniques, each with its respective valid inequalities added to the RMPs. The first preprocessing is to select \textit{promising cars} likely to use the RL and then use them to probe CRSP solutions both before and during BP-CP execution. The second is to solve a series of MIPs to prefix the departure and arrival timestamps for specific cars.

\begin{itemize}
    \item \textit{A valid lower bound for the CRSP-MS.} %Let $\underline{\tau}^e_k$ denote the lower bound on the arrival timestamp of each car $k\in V$ at the downstream shop, which equals $k+\bar{q}$ when the upstream and downstream sequences are identical.
    Let $\underline{\tau}^e_k$ denote the lower bound on the downstream-arrival timestamp for each car $k\in V$, which equals $k+\bar{q}$ when the upstream and downstream sequences are identical. As Corollary \ref{coro1} states, in the optimal CRSP-MS solution, the first car $\pi_1$ in $\Pi^{D}$ departs the upstream shop at timestamp $\pi_1-1$ and reaches the downstream shop at its lower bound $\underline{\tau}^e_1=\pi_1+\bar{q}$. This $\underline{\tau}^e_1$ yields a tighter bound $\underline{\tau}^e_2=\max\{\underline{\tau}^e_1+1,\tau_2+\bar{q}\}$ for the next car $\pi_2$, and the process continues recursively for all remaining cars. The bound $\underline{\tau}^e_{\pi_d}$ for the last car $\pi_d$ offers a valid lower bound on the \textit{resequencing makespan}, as summarized below.
\end{itemize}

\begin{proposition}\label{pro3}
    %\textup{(i)} In the optimal CRSP-MS solution, the first downstream car $\pi_1$ departs at timestamp $\pi_1-1$ and arrives at timestamp $\pi_1+\bar{q}$.
    \textup{(i)} A lower bound $\underline{\tau}^e_{\pi_p}$ on the downstream-arrival timestamp of car $\pi_p\in V$ ($\forall p\in P\backslash \{1\}$) is obtained recursively from $\max \{\underline{\tau}^e_{\pi_{p-1}}+1,\pi_p+\bar{q}\}$, where $\underline{\tau}^e_{\pi_1}=\pi_1+\bar{q}$. %\textup{(ii)} A lower bound $\underline{\tau}^e_{\pi_p}$ on the arrival timestamp of each car $\pi_p\in V\backslash \{\pi_1\}$ at its downstream position $p\in P\backslash \{1\}$ is obtained recursively as $\max \{\underline{\tau}^e_{\pi_{p-1}}+1,\pi_p+\bar{q}\}$, where $\underline{\tau}^e_{\pi_1}=\pi_1+\bar{q}$.
    \textup{(ii)} The bound $\underline{\tau}^e_{\pi_d}$ for the last car $\pi_d$ is a valid lower bound for the CRSP-MS. \textup{(iii)} If the CRSP-MS is feasible with a given MB and no RL, then the minimum makespan under that MB is exactly $\underline{\tau}^e_{\pi_d}$.
\end{proposition}

%%%%%%\begin{proposition}\label{pro3}
%%%%%%    \textup{(i)} In the optimal CRSP-MS solution, car $\pi_1$ departs at timestamp $\pi_{1}-1$ and arrives at timestamp $\pi_{1}+\bar{q}$. %\textup{(i)} In the optimal CRSP solution, the first car $\pi_{1}$ in the downstream shop departs at timestamp $\pi_{1}-1$ and arrives at timestamp $\pi_{1}+\bar{q}$.
%%%%%%    \textup{(ii)} A lower bound $\underline{\tau}^e_{\pi_p}$ on the arrival timestamp of each car $\pi_p \in V\backslash\{\pi_1\}$ at its downstream position $p\in P\backslash \{1\}$ is calculated recursively as $\max \{\underline{\tau}^e_{\pi_{p-1}}+1, \pi_{p}+\bar{q}\}$, where $\underline{\tau}^e_{\pi_1}=\pi_1+\bar{q}$. \textup{(iii)} The lower bound $\underline{\tau}^e_{\pi_d}$ for the last car $\pi_d$ is a valid lower bound for the CRSP-MS. \textup{(iv)} If the CRSP-MS is feasible under a given MB, then the minimum resequencing makespan equals the lower bound $\underline{\tau}^e_{\pi_d}$.
%%%%%%    \end{proposition}

\begin{itemize}
    \item \textit{Selection of promising cars likely to enter the RL.} As introduced in Section \ref{sec3-1}, two \textit{conflicting} cars cannot be assigned to the same forward lane without entering the RL. Therefore, for any subset $\mathcal{C}$ of cars in which every pair conflicts, if $|\mathcal{C}|$ exceeds the number of forward lanes, i.e., $|\mathcal{C}|>l^f$, then at least $|\mathcal{C}|-l^f$ cars in $\mathcal{C}$ must enter the RL.
    %Two cars cannot be assigned to the same forward lane without entering the RL if their upstream indexes and downstream positions are in opposite order, i.e., $k_1>k_2$, $p_1<p_2$, or $k_1<k_2$, $p_1>p_2$, $\forall k_1,k_2\in V$, $\forall p_1,p_2\in P$, defining them as \textit{conflicting} cars. For any subset $\mathcal{C}$ of cars where every pair conflicts, if $|\mathcal{C}|$ exceeds the number of forward lanes, i.e., $|\mathcal{C}|>l^f$, then at least $|\mathcal{C}|-l^f$ cars in $\mathcal{C}$ must enter the RL.
%Consider a subset $\mathcal{C}$ of cars where no two cars can be assigned to the same forward lane without using the RL. If the size of $\mathcal{C}$ exceeds all available forward lanes, i.e., $|\mathcal{C}|>|L^f|$, then at least $|\mathcal{C}|-|L^f|$ cars in $\mathcal{C}$ must use the RL.
\end{itemize}

From \citet{guo2025logic}, we identify such subsets via an undirected graph where each car is a vertex and edges connect \textit{conflicting} pairs. Each \textit{clique} in this graph represents a subset $\mathcal{C}$ in which all cars must enter distinct forward lanes unless they enter the RL. Using the \textit{Bron-Kerbosch algorithm}, we find only the \textit{maximal cliques} that cannot be extended further by adding more vertices.

%%%%Following \citet{guo2025logic}, such subsets can be identified via an undirected graph in which each car is considered a vertex. An edge links two cars if their upstream indexes and downstream positions are in opposite orders, i.e., $k_1>k_2$, $p_1<p_2$, or $k_1<k_2$, $p_1>p_2$, $\forall k_1, k_2\in V$, $\forall p_1,p_2\in P$, defining them as \textit{conflicting} cars. A subset $\mathcal{C}$ then forms a \textit{clique} in this graph, where every two cars are connected and must enter distinct forward lanes unless using the RL. As in \citet{guo2025logic}, we run the \textit{Bron-Kerbosch algorithm} to find only the \textit{maximal cliques} that cannot be further extended by adding more vertices (i.e., cars).

Let $\mathcal{C}^{\text{all}}$ be the set of these \textit{maximal cliques} whose sizes exceed $l^f$, and arrange the cars in each of these \textit{cliques} in ascending order based on their upstream indexes. %Let $\mathcal{C}^{\text{all}}$ be the set of these \textit{maximal cliques} whose sizes exceed $l^f$, with cars in each ordered by ascending upstream indexes.
In each \textit{clique} $\mathcal{C}\in \mathcal{C}^{\text{all}}$, at least $|\mathcal{C}|-l^f$ cars enter the RL, which is enforced by adding the following inequality to the MP at the root node before BP-CP execution.
%For each clique, the requirement that at least $|\mathcal{C}|-l^f$ cars in $\mathcal{C}$ enter the RL is enforced by adding the following inequality to the RMPs at the root node before BP-CP execution.

%%%Let $\mathcal{C}^{\text{all}}$ be the set of all \textit{maximal cliques} with sizes greater than $|L^f|$, and cars in each ordered by ascending upstream indexes. For each \textit{clique} $\mathcal{C}\in \mathcal{C}^{\text{all}}$, at least $|\mathcal{C}|-|L^f|$ cars in it must use the RL, which is enforced by adding the following inequality to the RMPs of the root node.

\vspace{-5mm}

\begin{equation}\label{ieq1}
    \text{Clique Inequality: }\qquad \sum_{k\in \mathcal{C}}{(1-z^k_0)}\geqslant |\mathcal{C}|-l^f
\end{equation}

\noindent Intuitively, the first $|\mathcal{C}|-l^f$ cars in $\mathcal{C}$ are more likely to use the RL in the optimal solution. These cars are positioned later downstream and stay in the MB-RL longer, allowing them to enter the RL but with little impact on the makespan. Formally, we refer to them as \textit{promising cars} as follows.
%%%Intuitively, the first $|\mathcal{C}|-l^f$ cars in $\mathcal{C}$ are more likely to use the RL in the optimal solution. These cars are positioned later downstream and stay longer in the MB-RL, allowing them to enter the RL with little impact on the makespan. We refer to them as \textit{promising cars} as follows.
\begin{definition}\label{def3}
    For each \textit{clique} $\mathcal{C}\in \mathcal{C}^{\text{all}}$, the first $|\mathcal{C}|-l^f$ cars are defined as \textit{promising cars} and form set $V^{\mathcal{C}}$. %Denote the set of all \textit{promising cars} across $\mathcal{C}^{\text{all}}$ by $V^{\text{pro}}=\bigcup_{\mathcal{C}\in \mathcal{C}^{\text{all}}}{V^{\mathcal{C}}}$.
    Let $V^{\text{pro}}=\bigcup_{\mathcal{C}\in \mathcal{C}^{\text{all}}}{V^{\mathcal{C}}}$ denote the set of all \textit{promising cars} across %all \textit{cliques} in
    $\mathcal{C}^{\text{all}}$.
    %%%%%%%%%%%The first $|\mathcal{C}|-|L^f|$ cars in each \textit{clique} $\mathcal{C}\in \mathcal{C}^{\text{all}}$% with $|\mathcal{C}|>|L^f|$
    %%%%%%%%%%%, denoted by set $\mathcal{C}(|\mathcal{C}|-|L^f|)$, are considered promising cars, defined as those likely to use the RL in the optimal CRSP solution. The set of all \textit{promising cars} is denoted by $V^{\textit{pro}}=\bigcup_{\mathcal{C}\in \mathcal{C}^{\text{all}}}{\mathcal{C}(|\mathcal{C}|-|L^f|)}$.
\end{definition}

\subsection{Overall implementation}\label{sec4-6}
This section details the BP-CP implementation. To reduce branching and accelerate the discovery of CRSP-MS solutions, we incorporate a heuristic that uses \textit{promising cars} to construct \textit{RL-entry plans} both before and during execution. The algorithm flowchart is given in the online supplement.

%%The implementation of BP-CP is detailed in this section. To reduce branching and accelerate the discovery of high-quality CRPS-MS solutions, it also incorporates a heuristic that uses \textit{promising cars} to build \textit{RL-entry plans} both before and during execution. The algorithm flowchart can be found in the online supplement.

Our procedure starts by computing the valid lower bound $\underline{\tau}^e_{\pi_d}$ given in Proposition \ref{pro3} and identifying the set $\mathcal{C}^{\text{all}}$ of all \textit{maximal cliques} of size greater than $l^f$, as described in Section \ref{sec4-5}. From $\mathcal{C}^{\text{all}}$, we derive all related inequalities \eqref{ieq1} and obtain the set $V^{\text{pro}}$ of \textit{promising cars}.

Before building the BB tree, we solve the CP model $\mathcal{M}^{\text{CP}}$ formed by bound $\underline{\tau}^e_{\pi_d}$ and a constructed \textit{RL-entry plan} $u^{\text{con}}_1=\{(k,1)\mid k\in V^{\text{pro}}\}\cup \{(k,0)\mid k\in V\backslash V^{\text{pro}}\}$, where each \textit{promising car} enters the RL once and the others never do, within a short time limit. If $\mathcal{M}^{\text{CP}}$ yields a feasible solution, the algorithm ends with the minimum makespan $\underline{\tau}^e_{\pi_d}$; if not, the BP-CP is run. %approach is truly activated.

%%%%%%Before initializing the BB tree, we solve a CP model $\mathcal{M}^{\text{CP}}$ formed by a makespan equal to the lower bound $\underline{\tau}^e_{\pi_d}$ and a constructed RL-usage plan $u^{\text{con}}_1$, where each promising car enters the RL once and all others not at all, i.e., $u^{\text{con}}_1=\{(k,1)\mid k\in V^{\text{pro}}\}\cup\{(k,0)\mid k\in V\backslash V^{\text{pro}}\}$, within a short time limit. If this model $\mathcal{M}^{\text{CP}}$ yields a feasible solution, then the algorithm ends with the confirmed optimal makespan $\underline{\tau}^e_{\pi_d}$; otherwise, the BP-CP framework is truly activated.

%%%%%%When BP-CP is activated, we first solve the maximization problem $\mathcal{M}^{\text{A}}(k^{\circ})$ for each RL-unused car $k^{\circ}\in V^R_{\times}$ to obtain all the inequalities \eqref{ieq2}. As per Proposition \ref{pro4}, this preprocessing allows us to determine the departure and arrival timestamps for specific cars and then fix the \textit{start} and \textit{end time} values of the related \textit{interval variables} when solving $\mathcal{M}^{\text{CP}}$ in stage $\mathcal{S}_3$.

The BP-CP uses a queue $Q$ that stores all unexplored BB nodes and is initialized with the root node. At the root node, the MP is created using a simple heuristic solution that lets cars loop in the MB-RL until their turn to exit. This solution sets an upper bound ($ub$) on the makespan and, for each car $k\in V$, a downstream-arrival timestamp bound $\tau^e_k$ used to build its MP-specific TSN graph. All derived inequalities \eqref{ieq1} are added to the MP at the root node.
%Once activated, the BP-CP procedure for each RMP uses a queue $Q$ that stores all unexplored BB nodes and is initialized with the root node. At the root node, the RMP is formed using a simple heuristic solution that lets cars loop in the MB-RL until their turn to exit. This solution sets an upper bound ($ub$) on the makespan and, for each car $k\in V$, a downstream-arrival timestamp bound $\tau^e_k$ used to build its RMP-specific TSN graph. All derived inequalities \eqref{ieq1} are added to the root RMP.

%%%%%%Afterward, the BP-CP procedure with each of the three RMPs proceeds using a queue $Q$ that stores all unexplored nodes in the BB tree and is initialized with the root node. To form the RMP of the root node, a rule-based heuristic (see \ref{app7}) creates an initial solution in which all RL-unused cars never enter the RL. This solution provides an upper bound ($ub$) for the CRSP, along with the arrival-timestamp bound $\tau^e_k$ used to build the three types of TSN graphs for each car $k\in V$. Based on the selected graph type, initial columns are identified, and inequalities \eqref{ieq1} and \eqref{ieq2} from preprocessing are added to complete the RMP of the root node.

The algorithm then repeatedly processes nodes from $Q$, always popping the node whose parent has the minimum RMP objective value ($\widehat{m}$). When multiple parents share this minimum value, we select the one with the highest number of integer-valued $\boldsymbol{z}$ variables in its RMP solution.

%%Each popped node first undergoes a CG iteration between $\mathcal{S}_1$ and $\mathcal{S}_2$: solving the RMP in $\mathcal{S}_1$ and adding new columns from $\mathcal{S}_2$. Once the iteration ends with the minimum makespan $\widehat{m}$ of the current RMP, we check whether any \textit{path} in $\Omega^{\prime}_{\pi_d}$ guides car $\pi_d$ to reach downstream before timestamp $\widehat{m}$. If so, that \textit{path} and the related arcs in the graphs of car $\pi_d$ are removed, and CG restarts.

Each popped node first undergoes CG to solve the RMP to optimality. Once completed, we check whether any \textit{path} in $\Omega^{\prime}_{\pi_d}$ directs car $\pi_d$ to reach the downstream shop before the current RMP objective value $\widehat{m}$. If such a \textit{path} exists, it and the related arcs in the TSN graphs of car $\pi_d$ are removed, and CG restarts. If not, the procedure attempts to seek an \textit{RL-entry plan}. It first checks the integrality of the $\boldsymbol{z}$-variable values for non-promising cars. If there is a fractional-valued variable in $\{z^k_n\mid k\in V\backslash V^{\text{pro}}, n\in N_k\}$, then the current node finishes its exploration, and two child nodes are added to $Q$ using one branching rule from Section \ref{sec4-4}. If all are integral, an \textit{RL-entry plan} $u^{\text{con}}_2=\{(k,\widehat{n}_k)\mid k\in V\}$ is constructed: for each car $k\in V\backslash V^{\text{pro}}$, its RL-entry count $\widehat{n}_k=\sum_{n\in N_k}{n\cdot \bar{z}^k_n}$ (integer); for each car $k\in V^{\text{pro}}$, $\widehat{n}_k=1$ if $\sum_{n\in N_k}{n\cdot \bar{z}^k_n}<1$, otherwise $\widehat{n}_k=\big \lceil \sum_{n\in N_k}{n\cdot \bar{z}^k_n} \big \rceil$. This plan is then passed to the embedded heuristic.

To start the heuristic, we first seek the minimum makespan $\widehat{m}^{\prime}$ between $\widehat{m}$ and $ub-1$ such that the pair $(\widehat{m}^{\prime},u^{\text{con}}_2)$ meets all existing cuts in constraints \eqref{m1-con9}. If found, the heuristic enumerates makespans from $\widehat{m}^{\prime}$ to $ub-1$. For each, it solves $\mathcal{M}^{\text{CP}}$ with $u^{\text{con}}_2$ under a short time limit. If $\mathcal{M}^{\text{CP}}$ is feasible, then $ub$ is updated to that makespan, and the enumeration ends. When $ub=\widehat{m}$, the heuristic and node exploration are complete. Otherwise, after this enumeration, %After this enumeration,
if any makespan-$u^{\text{con}}_2$ pairs are proven infeasible, their feasibility cuts \eqref{cut2} are collectively added to the MP, and CG restarts.

%To run the heuristic, we search for a makespan $\widehat{m}^{\prime}$ from $\widehat{m}$ to $ub-1$ such that its combination with $u^{\text{con}}_2$ satisfies all the existing valid cuts in constraints \eqref{m1-con9}. If found, the heuristic proceeds by enumerating each makespan from $\widehat{m}^{\prime}$ up to $ub-1$ to solve $\mathcal{M}^{\text{CP}}$ with $u^{\text{con}}_2$ under a short time limit. If $\mathcal{M}^{\text{CP}}$ yields a feasible solution at a specific makespan, then $ub$ is updated accordingly, and the enumeration stops. After the enumeration ends, if there are makespans along with $u^{\text{con}}_2$ proven infeasible, the corresponding cuts \eqref{cut2} are added to the RMP at a time, and CG restarts.

If no valid $\widehat{m}^{\prime}$ is found or CG restarts, we check all $\boldsymbol{z}$ variables for integrality. If all are integers, an \textit{RL-entry plan} $u_3=\{(k,\widehat{n}_k)\mid k\in V, \widehat{n}_k=\sum_{n\in N_k}{n\cdot \bar{z}^k_n}\}$ is formed and used with makespan $\widehat{m}$ to solve $\mathcal{M}^{\text{CP}}$ in $\mathcal{S}_2$. A feasible $\mathcal{M}^{\text{CP}}$ solution ends node exploration and updates $ub=\widehat{m}$ if $ub>\widehat{m}$.
If $\mathcal{M}^{\text{CP}}$ is infeasible, cut \eqref{cut2} for pair $(\widehat{m},u_3)$ is added to the MP, and CG restarts. If not all $\boldsymbol{z}$ integers, node exploration ends with two child nodes created for $Q$ via the second branching rule.

%%%%%%%%%Otherwise, if neither a valid makespan $\widehat{m}^{\prime}$ is found nor any new cuts are generated within the time limit, we skip the heuristic and verify whether all $\boldsymbol{z}$ variables are integer-valued. If they are, a real RL-usage plan $u^{\text{real}}_3=\{(k,\widehat{n}_k)\mid k\in V\}$ with $\widehat{n}_k=\sum_{k\in V}{\sum_{n\in N_k}{n\cdot \bar{z}^k_n}}$ is obtained and then used in $\mathcal{S}_3$ of BP-CP to solve $\mathcal{M}^{\text{CP}}$ with makespan $\lceil \widehat{m}\rceil$. If $\mathcal{M}^{\text{CP}}$ is feasible, the node exploration concludes, and $ub$ is updated to $\lceil \widehat{m}\rceil$ when $ub>\lceil \widehat{m}\rceil$. Conversely, if $\mathcal{M}^{\text{CP}}$ is infeasible, a valid cut \eqref{cut2} specific to $\lceil \widehat{m}\rceil$ and $u^{\text{real}}_3$ is added to the RMP, and CG continues. If not all $\boldsymbol{z}$ variables are integers, the node exploration ends with the creation of two new BB nodes via branching.

A BB node is pruned if its own $\widehat{m}>ub$ when CG ends. Each completed node exploration updates the global lower bound ($lb$). The algorithm returns the minimum makespan when $ub = lb$.
%%%%%%A BB node is pruned if its parent has m1 > ub or if its own m1 > ub when CG ends. Each completed node exploration updates the global lower bound (lb). The algorithm stops and returns the minimum makespan when ub = lb.

%%%%Each completed node exploration updates the global lower bound ($lb$). The algorithm terminates when $Q$ is empty or the overall time limit is reached, returning the optimal solution when $Q$ is empty.

\section{Computational results}\label{sec5}
All experiments in this section were conducted within one hour on a single 2.4 GHz AMD Rome 7532 processor with 64 GB of RAM. Section \ref{sec5-1} describes the test instances. Section \ref{sec5-2} reports the BP-CP performance for each MP and compares the results with Gurobi. Section \ref{sec5-3} provides managerial insights into the MB-RL by examining the impact of its size and comparing it with the MB in terms of resequencing capability and minimum makespan across various levels of resequencing complexity. All algorithms were coded in C++, using IBM ILOG CP Optimizer 22.1.1 for CP models and Gurobi 12.0.0 for LP and MIP models. All instances and detailed results are available at \href{https://sites.google.com/view/xinyiguo}{https://sites.google.com/view/xinyiguo}.
%%%%All experiments in this section were conducted within one hour on a single 2.40 GHz AMD Rome 7532 processor with 64 GB of RAM. Section \ref{sec5-1} describes the test instances. Section \ref{sec5-2} reports the performance of the BP-CP approach with each RMP compared to Gurobi. Section \ref{sec5-3} compares the MB-RL and the MB in terms of resequencing feasibility and minimum makespan across varying levels of resequencing complexity. All algorithms were coded in C++, using IBM ILOG CP Optimizer 22.1.1 for CP models and Gurobi 12.0.0 for LP and MIP models. All instances, code, and detailed computational results are available at \href{https://sites.google.com/view/xinyiguo}{https://sites.google.com/view/xinyiguo}.

%%%%%%%%%All experiments for the algorithm comparison and sensitivity analysis presented in this section were conducted within one hour on a single 2.40 GHz AMD Rome 7532 processor with 64 GB of RAM. Section \ref{sec5-1} details the problem instances used. Section \ref{sec5-2} compares the performance of our BP-CP approach with three RMPs against Gurobi. Section \ref{sec5-3} investigates how buffer size (with and without an RL) and the complexity of the downstream sequence affect resequencing feasibility and the minimal makespan. All algorithms were coded in C++, using IBM ILOG CP Optimizer 22.1.1 as the CP solver and Gurobi 12.0.0 to solve LP and MIP models. Our instances and detailed computational results are available at \href{https://sites.google.com/view/xinyiguo}{https://sites.google.com/view/xinyiguo.}

\subsection{Instance generation}\label{sec5-1}
Each instance is a downstream sequence $\Pi^{D}=[\pi_1,\pi_2,\dots,\pi_d]$ reshuffled from the upstream sequence $\Pi^{U}=[1,2,\dots,d]$ of $d$ cars. We create two datasets, A and B, that differ in resequencing feasibility. Given the MB-RL size setup, dataset A for the algorithm comparison includes only instances verified to be feasible by our proposed polynomial-time algorithm (see the online supplement). Dataset B contains instances without feasibility guarantees. It serves to examine how resequencing complexity affects the minimum buffer size required for the MB and the MB-RL to restore feasibility, along with the corresponding minimum makespan at that size.
%%%It serves to examine how the \textit{resequencing complexity index} ($\sigma$) impacts the minimum buffer size needed for the MB and the MB-RL to restore feasibility, along with the corresponding minimum makespan at that size.

%%%%%%%%%%%%%Each instance is a downstream sequence $\Pi^{D}=[\pi_1,\dots,\pi_p,\dots,\pi_d]$ rearranged from the upstream one $\Pi^{U}=[1,\dots,k,\dots,d]$ with $d$ cars, where $k\in V$ denotes the $k$-th car in the upstream shop, and $\pi_p\in V$ is the upstream index of the car at downstream position $p\in P$. %where $\pi_p\in V$ and $k\in V$ denote the $p$-th and $k$-th cars in the downstream and upstream shops, respectively.
%%%%%%%%%%%%%We create two datasets that differ in resequencing feasibility. Given the MB-RL size setup, the basic dataset includes only instances verified as feasible by Algorithm \ref{alg1}. It is used for algorithm comparison, analyzing the impact of MB-RL size, and comparing against the buffer without RL (i.e., MB). The extended dataset contains instances not guaranteed to be feasible. It serves to examine how the resequencing complexity $\sigma$ affects the minimum buffer size needed for both MB-RL and MB to achieve feasibility, as well as the corresponding minimum makespan under that size.

%%Following \citet{guo2025logic}, we generate 30 instances for each $d\in \{10, 20, 30, 40, 50, 60, 70, 80,$ $90, 100, 110, 120\}$, yielding 360 instances in the basic dataset.
For dataset A, we generate 30 instances for each $d\in \{20, 30, 40, 50, 60, 70, 80, 90, 100, 110, 120\}$, resulting in 330 instances overall. The number of forward lanes ($l^f$) and lane capacity ($\bar{q}$) of the MB-RL vary with $d$: %The MB-RL has one RL. Its number of forward lanes ($l^f$) and lane capacity ($\bar{q}$) vary with $d$:
for $d\leqslant30$, $l^f=3$, $\bar{q}=4$; for $40\leqslant d\leqslant 60$, $l^f=4$, $\bar{q}=5$; for $70\leqslant d\leqslant 90$, $l^f=5$, $\bar{q}=6$; and for $d\geqslant 100$, $l^f=6$, $\bar{q}=7$. In practice, the last upstream car rarely appears as the first downstream one. Therefore, we first divide $\Pi^{U}$ into $d/10$ segments and randomly permute the cars within each. After recombining these segments, either 5 cars (for $d\leqslant 60$) or 10 cars (for $d\geqslant 70$) are randomly selected for swapping, followed by a feasibility check for resequencing under the corresponding MB-RL size using our polynomial-time algorithm. This process repeats until a resequencing-feasible sequence, i.e., the desired $\Pi^{D}$, is obtained.
%This process repeats until a sequence with confirmed resequencing feasibility, i.e., the desired $\Pi^{D}$, is obtained.

In dataset B, each instance is assigned a level of resequencing complexity measured by \textit{resequencing complexity index} ($\sigma$). By Definition \ref{def1}, $\sigma$ reaches its maximum value $\sigma_{\text{max}}=d\cdot (d-1)/2$ when $\Pi^{D}=[d,d-1,\dots,1]$ is the complete reverse of $\Pi^{U}$. We consider $d\in \{30,50,70\}$ under four complexity levels, where $\sigma$ falls within the intervals: $[0.05\sigma_{\text{max}}, 0.1\sigma_{\text{max}}]$, $[0.15 \sigma_{\text{max}}, 0.2 \sigma_{\text{max}}]$, $[0.25 \sigma_{\text{max}}, 0.3\sigma_{\text{max}}]$, or $[0.35 \sigma_{\text{max}}, 0.4\sigma_{\text{max}}]$. For each $d$-$\sigma$ combination, 30 instances are generated using steps similar to those used for dataset A, but without checking resequencing feasibility. A sequence obtained after segment recombination and car swaps is accepted as $\Pi^{D}$ when its $\sigma$ lies within the specified interval.

\subsection{Algorithm comparison}\label{sec5-2}
Conducted on dataset A, the testing in this section first assesses the heuristic in BP-CP that employs the lower bound $\underline{\tau}^e_{\pi_d}$ and \textit{promising cars} from Section \ref{sec4-5} to explore CRSP-MS solutions. It then evaluates the performance of all the methods we used to solve the CRSP-MS.

To assess the heuristic, we solve instances using two BP-CP setups, both with the second MP (i.e., $\mathcal{M}_2$), which pretests suggested as promising for solving more instances. The first setup (``$M_2$'') runs the full BP-CP including the heuristic, whereas the second (``$M_2\backslash$H'') omits it. In $M_2\backslash$H, after CG terminates, the procedure moves to stage $\mathcal{S}_2$ to solve the CP model $\mathcal{M}^{\text{CP}}$ if all $\boldsymbol{z}$ variables are integers; if not, branching is applied.
%%To assess the heuristic, we solve instances under two algorithmic setups, both based on BP-CP with the second MP (i.e., $\mathcal{M}_2$), which pretests suggested as promising for solving more instances.

\begin{table}[htbp]
    \centering
    \caption{Performance of Heuristic Procedures Embedded in the BP-CP Approach with $\mathcal{M}_2$\label{tab5}}

\fontsize{10.5}{12}\selectfont
\makebox[\textwidth][c]{%
    \begin{tabular}{lcccrrcrrcrrcrr}
    \toprule
    \multirow{2}[4]{*}{$d$} &       & \multirow{2}[4]{*}{\#Ins} &       & \multicolumn{2}{c}{\#Ins.Opt} &       & \multicolumn{2}{c}{Time$^{(\text{Total})}$ (sec.)} &       & \multicolumn{2}{c}{Time$^{(\text{CP})}$ (sec.)} &       & \multicolumn{2}{c}{Time$^{(\text{First})}$ (sec.)} \\
\cmidrule{5-6}\cmidrule{8-9}\cmidrule{11-12}\cmidrule{14-15}          &       &       &       & $M_2$    & $M_2\backslash$H   &       & $M_2$    & $M_2\backslash$H   &       & $M_2$    & $M_2\backslash$H   &       & $M_2$    & $M_2\backslash$H \\
\cmidrule{1-1}\cmidrule{3-3}\cmidrule{5-6}\cmidrule{8-9}\cmidrule{11-12}\cmidrule{14-15}    %10    &       & 0     &       &      &      &       &      &      &       &      &      &       &      &  \\
    20    &       & 27    &       & 27    & 27    &       & 30.0    & 94.4  &       & 0.4   & 6.2   &       & 14.4  & 44.6 \\
    30    &       & 27    &       & 27    & 25    &       & 53.1  & 328.0   &       & 0.4   & 6.3   &       & 4.6   & 126.3 \\
    40    &       & 29    &       & 29    & 28    &       & 132.6 & 145.4 &       & 0.4   & 1.3   &       & 58.3  & 12.4 \\
    50    &       & 30    &       & 30    & 26    &       & 137.2 & 586.5 &       & 1.3   & 5.3   &       & 18.5  & 64.8 \\
    60    &       & 26    &       & 26    & 22    &       & 246.1 & 587.1 &       & 2.9   & 2.2   &       & 23.6  & 24.1 \\
    70    &       & 30    &       & 27    & 16    &       & 450.8 & 1,877.7 &       & 3.3   & 50.3  &       & 30.4  & 245.3 \\
    80    &       & 24    &       & 13    & 12    &       & 1,839.7 & 1,873.7 &       & 2.9   & 14.1  &       & 28.3  & 51.9 \\
    90    &       & 28    &       & 9     & 9     &       & 2,477.5 & 2,533.3 &       & 3.3   & 16.1  &       & 12.3  & 37.5 \\
    100   &       & 24    &       & 6     & 4     &       & 2,760.2 & 3,147.1 &       & 63.3  & 504.4 &       & 18.8  & 608.0 \\
    110   &       & 27    &       & 9     & 3     &       & 2,456.6 & 3,367.4 &       & 112.8 & 589.7 &       & 37.4  & 22.7 \\
    120   &       & 29    &       & 10    & 8     &       & 2,485.7 & 2,689.0  &       & 49.8  & 368.5 &       & 34.9  & 35.4 \\
\cmidrule{1-1}\cmidrule{3-3}\cmidrule{5-6}\cmidrule{8-9}\cmidrule{11-12}\cmidrule{14-15}    Overall &       & 301   &       & 213   & 180   &       & 1,181.1 & 1,566.3 &       & 21.9  & 142.2 &       & 25.6  & 115.7 \\
    \bottomrule
    \end{tabular}
}% end makebox
\end{table}

Table \ref{tab5} presents the computational results. Note that solving $\mathcal{M}^{\text{CP}}$ with bound $\underline{\tau}^e_{\pi_d}$ and the constructed \textit{RL-entry plan} $u^{\text{con}}_1$ allows some instances to be solved before executing BP-CP. Excluding those early-solved instances, column ``\#Ins'' lists the number of instances that require complete BP-CP execution. Among these, columns ``\#Ins.Opt'' and ``Time$^{(\text{Total})}$'' show, respectively, the number solved to optimality and the average runtime in seconds (sec.). Columns ``Time$^{(\text{CP})}$'' and ``Time$^{(\text{First})}$'' compare the average $\mathcal{M}^{\text{CP}}$ solution time per model and the average time to obtain the first improved solution over the initial heuristic one.

The results confirm the effectiveness of the heuristic and its integrated enhancements. %Column ``\#Ins'' shows that the bound $\underline{\tau}^e_{\pi_d}$ from Proposition \ref{pro3} is tight when the number of cars is small relative to the MB-RL size, as all 10-car instances are solved without running the complete BP-CP.
The full algorithm (i.e., $M_2$) solves more instances and requires less runtime than $M_2\backslash$H. In $M_2\backslash$H, $\mathcal{M}^{\text{CP}}$ is solved only for the actual \textit{RL-entry plans} $u_3$ generated after all $\boldsymbol{z}$ variables become integral. %obtained when all $\boldsymbol{z}$ variables are integral.
We observe that these plans tend to use the RL less often than the constructed plans and thus need fewer \textit{interval variables} with larger domains for the same makespan, which makes $\mathcal{M}^{\text{CP}}$ harder to solve. The results in ``Time$^{(\text{First})}$'' suggest that \textit{promising cars} are more likely to enter the RL. Enforcing this in the constructed \textit{RL-entry plans} helps the heuristic probe feasible CRSP-MS solutions earlier.

\begin{table}[htbp]
    \centering
    \caption{Comparison of All Algorithms Used to Solve the CRSP-MS on the Basic Dataset\label{tab1}}
    \setlength{\tabcolsep}{5pt}
    \fontsize{10.5}{12}\selectfont
\makebox[\textwidth][c]{%
    \begin{tabular}{lrrrrrrrrrrrrrrrrrr}
\toprule
    \multirow{2}[3]{*}{$d$} &       &  \multicolumn{4}{c}{\#Ins.Opt} &       & \multicolumn{4}{c}{Time$^{(\text{Total})}$ (sec.)} &       & \multicolumn{3}{c}{\#Nodes.Solved} &       & \multicolumn{3}{c}{Time$^{(\text{Node})}$ (sec.)} \\
\cmidrule{3-6}\cmidrule{8-11}\cmidrule{13-15}\cmidrule{17-19}          &       & \multicolumn{1}{c}{CF} & \multicolumn{1}{c}{$M_1$} & \multicolumn{1}{c}{$M_2$} & \multicolumn{1}{c}{$M_3$} &       & \multicolumn{1}{c}{CF} & \multicolumn{1}{c}{$M_1$} & \multicolumn{1}{c}{$M_2$} & \multicolumn{1}{c}{$M_3$} &       & \multicolumn{1}{c}{$M_1$} & \multicolumn{1}{c}{$M_2$} & \multicolumn{1}{c}{$M_3$} &       & \multicolumn{1}{c}{$M_1$} & \multicolumn{1}{c}{$M_2$} & \multicolumn{1}{c}{$M_3$} \\
\cmidrule{1-1}\cmidrule{3-6}\cmidrule{8-11}\cmidrule{13-15}\cmidrule{17-19}    %10    &      & 30    & 30    & 30    & 30    &       & 0.6   & 0.1   & 0.1   & 0.1   &       &       &       &       &       &       &       &  \\
    20    &         & 24    & 29    & 30    & 29    &       & 1,184.8  & 170.7  & 27.1  & 130.1  &       & 18.5  & 40.6  & 121.3  &       & 23.8  & 1.3   & 5.0  \\
    30      &       & 17    & 26    & 30    & 27    &       & 2,433.9  & 511.4  & 47.8  & 417.2  &       & 20.4  & 66.1  & 1,125.6  &       & 80.0  & 1.0   & 1.6  \\
    40      &       & 1     & 24    & 30    & 22    &       & 3,549.4  & 857.2  & 128.1  & 970.8  &       & 40.7  & 44.5  & 73.3  &       & 79.8  & 2.3   & 36.3  \\
    50    &       & \multirow{8}[1]{*}{0} & 26    & 30    & 24    &       & \multirow{8}[1]{*}{$>3,600$} & 536.2  & 137.2  & 733.2  &       & 9.8   & 69.5  & 81.5  &       & 67.8  & 2.0   & 30.2  \\
    60    &       &       & 23    & 30    & 25    &       &       & 897.0  & 213.4  & 718.5  &       & 10.7  & 47.3  & 1,208.0  &       & 111.9  & 5.2   & 27.4  \\
    70    &       &       & 6     & 27    & 20    &       &       & 3,041.3  & 450.8  & 1,363.3  &       & 23.3  & 10.0  & 649.6  &       & 530.1  & 66.2  & 173.3  \\
    80     &       &       & 12    & 19    & 16    &       &       & 2,307.9  & 1,472.0  & 1,730.9  &       & 9.8   & 5.7   & 153.1  &       & 538.8  & 341.2  & 171.2  \\
    90     &       &       & 7     & 11    & 10    &       &       & 2,848.8  & 2,312.4  & 2,431.2  &       & 10.1  & 5.2   & 48.5  &       & 618.6  & 478.6  & 272.5  \\
    100   &       &       & 7     & 12    & 12    &       &       & 2,896.5  & 2,208.4  & 2,216.7  &       & 6.8   & 5.5   & 181.5  &       & 748.4  & 751.2  & 709.6  \\
    110   &       &       & 6     & 12    & 11    &       &       & 3,109.8  & 2,211.1  & 2,343.4  &       & 6.0   & 5.7   & 75.3  &       & 772.4  & 806.3  & 675.0  \\
    120    &       &       & 5     & 11    & 10    &       &       & 3,186.2  & 2,402.9  & 2,465.9  &       & 6.6   & 8.4   & 480.6  &       & 773.5  & 562.4  & 504.2  \\
\cmidrule{1-1}\cmidrule{3-6}\cmidrule{8-11}\cmidrule{13-15}\cmidrule{17-19}  Overall   &       & 42    & 171   & 242   & 206   &       &       & 1,851.2  & 1,055.6  & 1,411.0  &       & 14.8  & 28.0  & 381.7  &       & 395.0  & 274.3  & 236.9  \\
    \bottomrule
    \end{tabular}
}

\end{table}

We then ran the full BP-CP algorithm with the first and third MPs, denoted $M_1$ and $M_3$, on dataset A and compared their performance with that of the compact formulation solved by Gurobi (``CF''). The results are summarized in Table \ref{tab1}. For all instances, columns ``\#Ins.Opt'' and ``Time$^{(\text{Total})}$'' report the number solved to optimality and the average runtime, respectively. Comparisons among the MPs within BP-CP consider only instances that required complete BP-CP execution (i.e., those counted in column ``\#Ins'' of Table \ref{tab5}). For these, Table \ref{tab1} gives the average number of solved BB nodes per instance (``\#Nodes.Solved'') and the average solution time per node (``Time$^{(\text{Node})}$'').

Columns ``\#Ins.Opt'' and ``Time$^{(\text{Total})}$'' show that our BP-CP approach solves instances with up to 120 cars within one hour and greatly outperforms the solver. Across 330 instances, BP-CP variants $M_1$, $M_2$, and $M_3$ require less runtime than CF and solve 129, 200, and 164 more to optimality, respectively. Among them, $M_2$ performs the best. It solves all 50-car instances (the scale of a real-world one-hour production batch in \citet{guo2025logic}), whereas neither $M_1$ nor $M_3$ can. From $\mathcal{M}_1$ to $\mathcal{M}_3$, successive constraint relaxations ease the RMP but weaken its lower bound, leading to more BB nodes for convergence. In comparison, $M_2$ spends less per-node solution time than $M_1$ and explores fewer nodes overall than $M_3$, balancing bound tightness and model tractability. Columns ``\#Node.Sol'' and ``Time$^{(\text{Node})}$'' also reflect this trade-off: $M_3$ processes the most nodes with the shortest per-node solution time due to the weakest lower bound of the simplest $\mathcal{M}_3$.

\begin{table}[htbp]
    \centering
    \caption{Comparison of the Three RMPs within BP-CP\label{tab2}}
    \setlength{\tabcolsep}{5pt}  % 默认是 6pt，减小后表格会变窄
    \fontsize{10.5}{12}\selectfont
\makebox[\textwidth][c]{%
    \begin{tabular}{lrrrrrrrrrrrrrrrrrrrr}
    \toprule
    \multirow{2}[4]{*}{$d$} &       & \multicolumn{3}{c}{\#Ins.First (Gap$^{\text{F}}$-\%)} &       & \multicolumn{3}{c}{Time$^{(\text{First})}$ (sec.)} &       & \multicolumn{3}{c}{Time$^{(\text{RMP})}$ (sec.)} &       & \multicolumn{3}{c}{\#Cut$^{(\text{CP})}$} &       & \multicolumn{3}{c}{Gap$^{\text{R}}$ (\%)} \\
\cmidrule{3-5}\cmidrule{7-9}\cmidrule{11-13}\cmidrule{15-17}\cmidrule{19-21}          &       & \multicolumn{1}{c}{$M_1$} & \multicolumn{1}{c}{$M_2$} & \multicolumn{1}{c}{$M_3$} &       & \multicolumn{1}{c}{$M_1$} & \multicolumn{1}{c}{$M_2$} & \multicolumn{1}{c}{$M_3$} &       & \multicolumn{1}{c}{$M_1$} & \multicolumn{1}{c}{$M_2$} & \multicolumn{1}{c}{$M_3$} &       & \multicolumn{1}{c}{$M_1$} & \multicolumn{1}{c}{$M_2$} & \multicolumn{1}{c}{$M_3$} &       & \multicolumn{1}{c}{$M_1$} & \multicolumn{1}{c}{$M_2$} & \multicolumn{1}{c}{$M_3$} \\
\cmidrule{1-1}\cmidrule{3-5}\cmidrule{7-9}\cmidrule{11-13}\cmidrule{15-17}\cmidrule{19-21}    20    &       & 26(2.5) & 27(1.5) & 27(2.7) &       & 27.2  & 14.4  & 2.1   &       & 0.2   & 0.0   & {0.0\,\ }   &       & 8.6   & 52.1  & 40.1  &       & 0.6   & 1.4   & 2.7  \\
    30    &       & 24(1.6) & 27(1.1) & 27(2.1) &       & 9.2   & 4.6   & 1.1   &       & 0.2   & 0.0   & {0.0\,\ }   &       & 4.5   & 30.2  & 73.0  &       & 1.3   & 2.1   & 3.2  \\
    40    &       & 25(1.7) & 29(1.2) & 29(2.1) &       & 50.1  & 58.3  & 5.5   &       & 0.8   & 0.1   & {0.0\,\ }   &       & 13.1  & 49.4  & 45.4  &       & 1.1   & 1.2   & 1.6  \\
    50    &       & 27(1.2) & 30(0.6) & 29(2.1) &       & 32.0  & 18.5  & 3.3   &       & 0.8   & 0.1   & {0.0\,\ }   &       & 3.9   & 15.2  & 32.5  &       & 0.7   & 0.9   & 1.6  \\
    60    &       & 20(1.6) & 25(1.6) & 25(1.3) &       & 59.1  & 23.6  & 2.8   &       & 1.5   & 0.2   & {0.0\,\ }   &       & 3.7   & 15.3  & 33.3  &       & 0.5   & 0.8   & 1.4  \\
    70    &       & 7(0.7) & 28(0.3) & 30(0.5) &       & 352.3  & 30.4  & 55.5  &       & 8.6   & 0.3   & {0.0\,\ }   &       & 3.6   & 4.0   & 156.4  &       & 1.1   & 0.6   & 0.2  \\
    80    &       & 6(0.4) & 14(0.1) & 19(0.6) &       & 616.5  & 28.3  & 17.2  &       & 7.2   & 0.4   & {0.0\,\ }   &       & 3.9   & 3.3   & 93.6  &       & 0.0   & 0.2   & 0.1  \\
    90    &       & 5(0.7) & 9(0.2) & 16(1.0) &       & 393.6  & 12.3  & 9.9   &       & 7.6   & 0.5   & {0.0\,\ }   &       & 3.6   & 4.2   & 55.7  &       & 0.7   & 0.3   & 0.2  \\
    100   &       & 1{(--)}\textcolor{white}{.0}  & 5{(0)}\textcolor{white}{.0}  & 13{(0)}\textcolor{white}{.0} &       & 1,189.1  & 18.8  & 43.4  &       & 16.7  & 0.5   & {0.0\,\ }   &       & 1.5   & 2.1   & 360.5  &       & 0.0   & 0.0   & 0.0  \\
    110   &       & 2(0.4) & 10(0)\textcolor{white}{.0} & 13(0.8) &       & 1,258.1  & 37.4  & 29.7  &       & 18.4  & 0.5   & {0.0\,\ }   &       & 1.6   & 7.4   & 189.3  &       & 0.7   & 1.0   & 0.1  \\
    120   &       & 4(0)\textcolor{white}{.0}  & 10(0)\textcolor{white}{.0} & 14(1.0) &       & 590.1  & 34.9  & 18.8  &       & 18.2  & 0.6   & {0.0\,\ }   &       & 1.7   & 12.7  & 85.1  &       & 0.0   & 0.0   & 0.0  \\
\cmidrule{1-1}\cmidrule{3-5}\cmidrule{7-9}\cmidrule{11-13}\cmidrule{15-17}\cmidrule{19-21}    Overall &       & 147(1.1) & 214(0.6) & 242(1.3) &       & 416.1  & 25.6  & 17.2  &       & 7.3   & 0.3   & {0.0\,\ }   &       & 4.5   & 17.8  & 105.9  &       & 0.6   & 0.8   & 1.0  \\
    \bottomrule
    \end{tabular}
}

\end{table}

Table \ref{tab2} further compares the three MPs within BP-CP, excluding early-solved instances.  %(e.g., all 10-car cases).
The first six columns concern the first improved solution found by BP-CP relative to the initial heuristic solution. Column ``\#Ins.First'' counts the instances where such a solution is found. The makespan gap between the first improved solution ($\widehat{m}^{\text{Fir}}$) and the optimal ($\widehat{m}^{*}$) is computed as $100\cdot (\widehat{m}^{\text{Fir}}-\widehat{m}^{*})/\widehat{m}^{*}$, and the values in parentheses under ``Gap$^{\text{F}}$'' show its average over instances where both $\widehat{m}^{\text{Fir}}$ and $\widehat{m}^{*}$ are available; ``--'' indicates that no instances are solved to optimality. Column ``Time$^{(\text{First})}$'' gives the average time to find the first improved solution across the instances counted in ``\#Ins.First''. For all BP-CP-executed instances, columns ``Time$^{(\text{RMP})}$'' and ``\#Cut$^{(\text{CP})}$'' report the average solution time per RMP and the average number of feasibility cuts \eqref{cut2} added per instance. Among those solved optimally, column ``Gap$^{\text{R}}$'' shows the average gap $100\cdot (\widehat{m}^{*}-\widehat{m}^{\text{root}})/m^{*}$ between the root-node lower bound $\widehat{m}^{\text{root}}$ and the optimal makespan $\widehat{m}^{*}$.%, calculated as $100\cdot (\widehat{m}^{*}-\widehat{m}^{\text{root}})/m^{*}$.

%%%%%%%%%%%Table \ref{tab2} provides a more detailed comparison of the three RMPs used in BP-CP. Instances solved without activating BP-CP are excluded from this analysis. The first six columns report the results of the first improved solution found during BP-CP execution. Column ``\#Ins.First'' shows the number of instances for which such an improved solution was obtained. To assess solution quality, the values in parentheses under ``Gap$^{\text{F}}$'' give the average percentage gap between the makespan $\widehat{m}^{\text{Fir}}$ of the first improved solution and that $\widehat{m}^*$ of the optimal one for those instances solved to optimality, computed as $100\cdot(\widehat{m}^{\text{Fir}}-\widehat{m}^*)/\widehat{m}^*$. Across all instances considered in ``\#Ins.First'', the ``Time$^{(\text{First})}$'' columns list the average time to find the first improved solution. The average time to solve each RMP appears in the columns under ``Time$^{(\text{RMP})}$''. For instances solved to optimality by BP-CP, column ``Gap$^{\text{R}}$'' reports the average percentage gap between the optimal makespan $\widehat{m}^*$ and the lower bound $\widehat{m}^{\text{root}}$ obtained at the root node, calculated as $100\cdot(\widehat{m}^*-\widehat{m}^{\text{root}})/\widehat{m}^*$. Lastly, column ``\#Cut$^{(\text{CP})}$'' displays the average number of valid cuts \eqref{cut2} generated per instance.

Based on the first six columns, the simplest $\mathcal{M}_3$ enables BP-CP to find an improved solution for the most instances in the shortest time. In terms of solution quality, those provided by $\mathcal{M}_2$ are the best, with an average gap of only 0.6\% from the optimal makespan. The long $\mathcal{M}^{\text{LP}}_1$ solution time shown in the ``Time$^{(\text{RMP})}$'' results suggests that a very tight MP degrades BP-CP performance. The ``\#Cut$^{(\text{CP})}$'' results, as measured by feasibility-cut counts, reveal the drawback of an over-relaxed MP: the frequent generation of infeasible partial integer solutions. Even with a 0.2\% increase in the makespan gap between the root node and the optimal solution, $\mathcal{M}_2$ is preferable to $\mathcal{M}_1$ %adopting $\mathcal{M}^{\text{LP}}_2$ over $\mathcal{M}^{\text{LP}}_1$ is worthwhile
because it substantially reduces the RMP solution time. Instead, replacing $\mathcal{M}_2$ with $\mathcal{M}_3$ is not advisable because the marginal time savings from $\mathcal{M}_3$ cannot offset the loss in bound quality.

Collectively, the BP-CP approach benefits most from a moderately relaxed MP (e.g., $\mathcal{M}_2$). This in-between formulation avoids both the high solution time of a tighter but intractable MP and the excessive growth of the BB tree caused by a weaker MP. It strikes an effective balance between the number of BB nodes created and the solution time per node.

%%%%%Collectively, the BP-CP approach performs best with a moderately relaxed RMP (e.g., $\mathcal{M}^{\text{LP}}_2$). Such an in-between formulation reduces the high solution time of the tighter RMP that comes directly from the compact formulation, and avoids excessive growth of the BB tree associated with the weakest RMP when proving optimality. It achieves a balance between the number of BB nodes created and the solving time per node.

%%%%\subsection{Sensitivity analysis (\textcolor{red}{Comparison of the MB-RL and the MB})}\label{sec5-3}
%%%\subsection{Comparison of the MB-RL and the MB}\label{sec5-3}
\subsection{Managerial Insights into the MB-RL}\label{sec5-3}
This section first examines how MB-RL size affects both BP-CP performance and the minimum makespan. It then evaluates the RL by comparing MB-RL with traditional MB on resequencing capability and efficiency across various difficulty levels, followed by managerial insights for adopting the MB-RL from the results. For each test instance, the minimum makespan under the MB-RL is obtained by running full BP-CP with $\mathcal{M}_2$ for one hour; under the MB, it is derived from the last conclusion of Proposition \ref{pro3}. The MB-RL size settings given in Section \ref{sec5-1}, defined by the total number of lanes ($\bar{l}$) and lane capacity ($\bar{q}$), are considered the baseline buffer size (``Base.Size'').

\subsubsection*{Impact of MB‑RL Size on Algorithm Performance and Makespan}
Tests to examine the effects of varying the MB-RL size were conducted on instances with $d\in \{30, 50, 70\}$ from dataset A. Keeping one RL, we independently added 1 to 4 extra forward lanes and cells per lane (including the RL) to the baseline-sized MB-RL for each $d$, then solved the CRSP-MS for each enlarged MB-RL.

Table \ref{tab4} reports the results. Columns ``$\Delta\bar{q}$'' and ``$\Delta l^f$'' list, respectively, the extra cells added per lane and the number of additional forward lanes compared to the baseline. The BP-CP performance is measured by the number of instances solved to optimality (``\#Ins.Opt'') and the average total runtime (``Time''). To assess the effects on makespan, we calculate the relative change $100\cdot (\widehat{m}^* - \widehat{m}^*_B)/\widehat{m}^*_B$, where $\widehat{m}^*$ is the minimum makespan under each enlarged MB-RL and $\widehat{m}^*_B$ under the baseline. For all instances solved optimally under both sizes, column ``$\Delta \widehat{m}^*$'' shows the average percentage change, with ``$+$''/``$-$'' indicating an increase/decrease in the average makespan relative to the baseline-sized MB-RL.

\begin{table}[htbp]
    \centering
    \caption{Effects of Changing the MB-RL Size on BP-CP Performance and the Minimum Makespan\label{tab4}}
    \setlength{\tabcolsep}{5pt}
    \fontsize{11}{12.5}\selectfont
\makebox[\textwidth][c]{%

\begin{tabular}{ccrccrrcccrr}
    \toprule
    \multirow{2}[4]{*}{$d$} & \multirow{2}[4]{*}{\makecell{Base.Size\\($\bar{l}$,$\bar{q}$)}} &       & \multicolumn{4}{c}{Increase $\bar{q}$} &       & \multicolumn{4}{c}{Increase $l^f$} \\
\cmidrule{4-7}\cmidrule{9-12}          &       &       & $\Delta\bar{q}$    & \#Ins.Opt & \multicolumn{1}{c}{Time (sec.)} & \multicolumn{1}{c}{$\Delta \widehat{m}^{*}$ (\%)} &       & $\Delta l^f$    & \#Ins.Opt & \multicolumn{1}{c}{Time (sec.)} & \multicolumn{1}{c}{$\Delta \widehat{m}^{*}$ (\%)} \\
\cmidrule{1-2}\cmidrule{4-7}\cmidrule{9-12}    \multirow{4}[2]{*}{30} & \multirow{4}[2]{*}{$\bar{l}=4, \bar{q}=4$} &       & 1     & 30    & 30.6  & +3.3 &       & 1     & 30    & 1.3   & $-3.1$ \\
          &       &       & 2     & 27    & 511.6 & +6.8 &       & 2     & 30    & 0.2   & $-4.4$ \\
          &       &       & 3     & 24    & 800.8 & +11.8 &       & 3     & 30    & 0.1   & $-4.3$ \\
          &       &       & 4     & 21    & 1,103.8 & +18.8 &       & 4     & 30    & 0.1   & $-4.5$ \\
\cmidrule{1-2}\cmidrule{4-7}\cmidrule{9-12}    \multirow{4}[2]{*}{50} & \multirow{4}[2]{*}{$\bar{l}=5, \bar{q}=5$} &       & 1     & 26    & 578.5 & +3.0 &       & 1     & 30    & 1.7   & $-3.0$ \\
          &       &       & 2     & 23    & 978.9 & +6.0 &       & 2     & 30    & 0.4   & $-3.4$ \\
          &       &       & 3     & 22    & 1,001.5 & +9.9 &       & 3     & 30    & 0.4   & $-3.5$ \\
          &       &       & 4     & 19    & 1,640.5 & +15.6 &       & 4     & 30    & 0.4   & $-3.5$ \\
\cmidrule{1-2}\cmidrule{4-7}\cmidrule{9-12}    \multirow{4}[2]{*}{70} & \multirow{4}[2]{*}{$\bar{l}=6, \bar{q}=6$} &       & 1     & 26    & 534.9 & +1.3 &       & 1     & 25    & 608.8 & $-1.1$ \\
          &       &       & 2     & 23    & 867.0 & +2.2 &       & 2     & 29    & 124.1 & $-1.2$ \\
          &       &       & 3     & 23    & 847.6 & +3.2 &       & 3     & 30    & 0.8   & $-1.2$ \\
          &       &       & 4     & 22    & 988.6 & +9.3 &       & 4     & 30    & 0.9   & $-0.8$ \\
    \bottomrule
    \end{tabular}
}
\end{table}

%%%The results show that adding forward lanes more effectively improves BP-CP performance and reduces the minimum makespan than increasing lane capacity. Extra lanes help separate \textit{conflicting cars} across different forward lanes, reducing reliance on the RL and bringing the makespan closer to the lower bound in Proposition \ref{pro3}. Consequently, more instances can be solved without complete BP-CP execution. In contrast, larger lane capacity offers no benefits. It increases intra-lane travel time without reducing RL usage, which raises the makespan and requires more computational effort for BP-CP to converge.

The results show that adding forward lanes more effectively improves BP-CP performance and reduces the minimum makespan than increasing lane capacity. Extra lanes help separate \textit{conflicting cars} across different forward lanes, reducing reliance on the RL and bringing the makespan closer to the lower bound in Proposition \ref{pro3}. Consequently, more instances can be solved without complete BP-CP execution. In contrast, larger lane capacity increases intra-lane travel time without decreasing RL usage, thereby raising makespan and requiring more computational effort for BP-CP to converge.

\subsubsection*{MB‑RL versus MB: Resequencing Capability and Efficiency}
%\subsubsection*{Resequencing Capability}
%\subsubsection*{Resequencing Efficiency}

Before comparing the MB-RL with the MB, we define them as identical in size if they share the same total number of lanes ($\bar{l}$) and the same lane capacity ($\bar{q}$). In other words, the RL of the MB-RL is replaced by a forward lane in the same-sized MB.

Using dataset B, the comparison between the two buffers aims to evaluate their resequencing performance across different levels of resequencing complexity, which are reflected in the \textit{resequencing complexity index} ($\sigma$) for each instance. Because some instances in dataset B may be infeasible when using the baseline-sized buffer to resequence, they were first identified for each buffer using Proposition \ref{pro1}. For each infeasible instance, we incrementally increased the number of forward lanes from the baseline, keeping lane capacity fixed (as expansion in $\bar{q}$ previously proved ineffective), until feasibility was restored or up to 20 extra lanes were added. Table \ref{tab3} shows, for each buffer, by $d$ and $\sigma$ range, the number of infeasible instances (``Ins.Infea'') and the average number of extra forward lanes required per such instance to regain feasibility (``$\Delta l^f$'').
%%%Conducted on the extended dataset, the comparative test evaluates the resequencing capability and efficiency of the MB-RL and the MB under varying levels of resequencing complexity, which are reflected in the \textit{resequencing complexity index} ($\sigma$) for each instance. Because instances in this dataset may be infeasible when using the baseline-sized buffer to resequence, we first identified these for each buffer using Proposition \ref{pro1}. For each infeasible instance, we incrementally added forward lanes to the baseline-sized buffer, keeping lane capacity $\bar{q}$ unchanged (as mentioned before, expansion in $\bar{q}$ is ineffective), until feasibility was restored or 20 extra lanes were added. %For each buffer,
%%%Table \ref{tab3} reports, by $d$ and $\sigma$, the number of infeasible instances (``Ins.Infea'') and the average minimum number of extra forward lanes required per such instance to regain feasibility (``$\Delta l^f$'').

\begin{table}[htbp]
    \centering
    \caption{Comparison between the MB-RL and the MB Across Different Levels of Resequencing Complexity\label{tab3}}
\setlength{\tabcolsep}{5pt}  % 默认是 6pt，减小后表格会变窄
    \fontsize{10.5}{12}\selectfont
\makebox[\textwidth][c]{%
    \begin{tabular}{ccccccccccccccccccc}
    \toprule
    $d$     &       & \multicolumn{5}{c}{30 (Base.Size $\bar{l}=4$, $\bar{q}=4$)}                &       & \multicolumn{5}{c}{50 (Base.Size $\bar{l}=5$, $\bar{q}=5$)}                &       & \multicolumn{5}{c}{70 (Base.Size $\bar{l}=6$, $\bar{q}=6$)} \\
\cmidrule{3-7}\cmidrule{9-13}\cmidrule{15-19}          &       & \multicolumn{2}{c}{Ins.Infea} &       & \multicolumn{2}{c}{$\Delta l^f$} &       & \multicolumn{2}{c}{Ins.Infea} &       & \multicolumn{2}{c}{$\Delta l^f$} &       & \multicolumn{2}{c}{Ins.Infea} &       & \multicolumn{2}{c}{$\Delta l^f$} \\
\cmidrule{3-4}\cmidrule{6-7}\cmidrule{9-10}\cmidrule{12-13}\cmidrule{15-16}\cmidrule{18-19}    $\sigma$    &       & MB-RL & MB    &       & MB-RL & MB    &       & MB-RL & MB    &       & MB-RL & MB    &       & MB-RL & MB    &       & MB-RL & MB \\
\cmidrule{1-1}\cmidrule{3-4}\cmidrule{6-7}\cmidrule{9-10}\cmidrule{12-13}\cmidrule{15-16}\cmidrule{18-19}    $[0.05\sigma_{\text{max}},0.1\sigma_{\text{max}}]$  &       & 0     & 8     &       & 0     & 1.0     &       & 0     & 9     &       & 0     & 1.2   &       & 5     & 15    &       & 1.8   & 2.1 \\
    $[0.15\sigma_{\text{max}},0.2\sigma_{\text{max}}]$  &       & 4.0     & 27    &       & 1.0     & 1.6   &       & 30    & 30    &       & 2.9   & 4.0     &       & 30    & 30    &       & 4.9   & 6.3 \\
    $[0.25\sigma_{\text{max}},0.3\sigma_{\text{max}}]$  &       & 29    & 30    &       & 2.0     & 3.4   &       & 30    & 30    &       & 4.1   & 5.6   &       & 30    & 30    &       & 5.5   & 7.4 \\
    $[0.35\sigma_{\text{max}},0.4\sigma_{\text{max}}]$  &       & 30    & 30    &       & 3.0     & 4.9   &       & 30    & 30    &       & 4.7   & 6.8   &       & 30    & 30    &       & 5.7   & 8.6 \\
    \bottomrule
    \end{tabular}
}
\end{table}

We observe that the MB-RL consistently outperforms the MB in resequencing capability across all levels of complexity. At the baseline size, it yields fewer infeasible instances and needs fewer additional lanes to restore feasibility. At the highest complexity ($d=70$, $\sigma\in [0.35\sigma_{\text{max}},0.4\sigma_{\text{max}}]$), the MB-RL restores feasibility for all instances with about 34\% fewer extra lanes than the MB (5.7 versus 8.6), and this benefit increases as complexity decreases.
%%%It can be seen that the MB-RL consistently outperforms the MB in resequencing capability across all levels of resequencing complexity. Of the baseline size, it yields fewer infeasible instances and requires fewer added forward lanes to restore feasibility. At the highest resequencing complexity ($d=70$, $\sigma$ at 35\% to 50\% of $d\cdot(d-1)/2$), the MB-RL achieves full feasibility with about 34\% fewer extra lanes than the MB (5.7 versus 8.6), and this advantage increases as complexity decreases.

Managerially, incorporating an RL can replace several forward lanes to achieve the same resequencing capability. It effectively reduces buffer construction costs and enhances operational flexibility, which is a key benefit in JIT production with frequent and complex sequence adjustments.

%%%%%%%%%It can be observed that the MB-RL consistently outperforms MB in resequencing feasibility under the same baseline size. It results in fewer infeasible instances for resequencing and requires fewer extra forward lanes to restore feasibility. At the highest complexity level ($d = 70$, $\sigma=35-40\%$ of $d\cdot (d-1)/2$), the MB-RL achieves full feasibility with at least 34\% fewer extra lanes than the MB (5.7 versus 8.6). This advantage grows as complexity decreases. For $d\leqslant 50$ and $\sigma< 0.15\cdot d\cdot (d-1)/2$, the MB-RL attains complete feasibility without extra lanes, whereas the MB still requires them. Thus, replacing one forward lane with an RL greatly improves buffer robustness in terms of resequencing feasibility across different complexity levels.

\begin{figure}[htbp]
    \centering

 \noindent\makebox[\linewidth][c]{
    \centering
    \scalebox{0.95}{
 \begin{tikzpicture}
    \begin{axis}[
        axis y line*=left,
        width=6.6cm,
		height=6.2cm,
        ybar,
		ybar=4*\pgflinewidth,
		major x tick style = transparent,
		title={(a) $d=30$},
		title style={yshift=-2mm},
		enlargelimits=0.15,
		ylabel={\normalsize \textbf{Avg. $\widehat{m}^{*}$ (TU)}},
		ylabel style={at={(axis description cs:-0.22,0.5)},anchor=center},
		symbolic x coords={5-10\%, 15-20\%, 25-30\%, 30-35\%},
		grid=major,major grid style={dashed},
		xlabel = {\normalsize Percentage of $\sigma_{\text{max}}$},
                    xlabel style = {yshift=-1mm},
		tick label style={font=\small},
		legend style={at={(0.42,0.98)},
						anchor=north,legend columns=-1, font=\fontsize{8}{8}\selectfont
                        },
					nodes near coords,
					every node near coord/.append style={font=\footnotesize},
					nodes near coords align={vertical},
					]
        \addplot[black,fill=white,%pattern=north west lines, pattern color=red,
        every node near coord/.append style=
					{red,font=\fontsize{8}{12}\selectfont, xshift=-0.9mm, yshift=-1mm}] coordinates {(5-10\%, 42.0) (15-20\%, 47.9) (25-30\%, 55.4) (30-35\%, 60.5)};
        \addplot[black,fill=black!50,every node near coord/.append style=
					{black,font=\fontsize{8}{12}\selectfont, yshift=-1mm, xshift=+0.8mm}] coordinates {(5-10\%,  40.6) (15-20\%, 47.5) (25-30\%, 55.4) (30-35\%, 60.4)};
					\legend{MB-RL, MB}
        %\node[anchor=north west] at (axis description cs:1,0) {x1};
    \end{axis}
    \begin{axis}[
                    height = 6.2cm,
					axis y line*=right,
					axis x line=none,
					width=6.6cm,
					axis line style={-},
					nodes near coords,
					every node near coord/.append style={
						rotate=20,
						font=\bfseries\fontsize{11}{8}\selectfont,
                        anchor = east,
						%anchor=north west,
						%yshift=+2mm,
                        xshift=+0.6mm
					},
					ylabel = {},
					yticklabel style={
						font=\small\bfseries, rotate=20,
					},
					symbolic x coords={5-10\%, 15-20\%, 25-30\%, 30-35\%},
					tick label style={font=\small},
					scaled y ticks = false,
					legend cell align=left,
					legend style={at={(0,1.38)},
						anchor=north,legend columns=1, %font=\scriptsize
                        font=\fontsize{6}{8}\selectfont}
					]
					\addplot[color=black,mark=o,sharp plot] coordinates {(5-10\%,30) (15-20\%,13) (25-30\%,15) (30-35\%,13)};
					%\legend{\textbf{\#Ins.Opt}}
				\end{axis}
\end{tikzpicture}
\hspace{-3mm}
\begin{tikzpicture}
    \begin{axis}[
					axis y line*=left,
                    width=6.6cm,
		              height=6.2cm,
					ybar,%=8pt, % configures ‘bar shift’
					ybar=4*\pgflinewidth,
					major x tick style = transparent,
					title={(b) $d=50$},
					title style={yshift=-2mm},
					enlargelimits=0.15,
					ylabel={},
					symbolic x coords={5-10\%, 15-20\%, 25-30\%, 30-35\%},
					grid=major,major grid style={dashed},
					xlabel = {\normalsize Percentage of $\sigma_{\text{max}}$},
                    xlabel style = {yshift=-1mm},
					tick label style={font=\small},
					%x tick label style={rotate=10},
					legend cell align=left,
					legend style={at={(0.5,-0.335))},
						anchor=north,legend columns=2, font=\scriptsize},
					nodes near coords,
					every node near coord/.append style={font=\footnotesize},
					nodes near coords align={vertical},
					]
            \addplot[black,fill=white,%pattern=north west lines, pattern color=red,
            every node near coord/.append style=
					{red,font=\fontsize{8}{12}\selectfont, xshift=-1mm, yshift=-1mm}] coordinates {(5-10\%, 66.8) (15-20\%, 89.3) (25-30\%, 97.7) (30-35\%, 100.7)};
			\addplot[black,fill=black!50,,every node near coord/.append style=
					{black,font=\fontsize{8}{12}\selectfont, yshift=-1mm, xshift=+1mm}] coordinates {(5-10\%,  65.9) (15-20\%, 89.3) (25-30\%, 97.7) (30-35\%, 100.7)};
    \end{axis}
    \begin{axis}[
                    height=6.2cm,
					axis y line*=right,
					axis x line=none,
					width=6.6cm,
					enlarge x limits=0.2,
					axis line style={-},
					nodes near coords,
					every node near coord/.append style={
						rotate=20,
						font=\bfseries\fontsize{11}{8}\selectfont,
                        anchor = east,
						%anchor=north west,
						yshift=+2mm,
                        xshift=-0.2mm
					},
					ylabel = {},
					yticklabel style={
						font=\small\bfseries, rotate=20,
					},
					symbolic x coords={5-10\%, 15-20\%, 25-30\%, 30-35\%},
					tick label style={font=\small},
					scaled y ticks = false,
					legend cell align=left,
					legend style={at={(0.26,1.11)},
						anchor=north,legend columns=-1, font=\bfseries\scriptsize}
					]
        \addplot[color=black,mark=o,sharp plot] coordinates {(5-10\%,30) (15-20\%,18) (25-30\%,11) (30-35\%,6)};
    \end{axis}
\end{tikzpicture}
\hspace{-3mm}
\begin{tikzpicture}
    \begin{axis}[
					axis y line*=left,
                    width=6.6cm,
		              height=6.2cm,
					ybar,%=8pt, % configures ‘bar shift’
					ybar=4*\pgflinewidth,
					major x tick style = transparent,
					title={(c) $d=70$},
					title style={yshift=-2mm},
					enlargelimits=0.15,
					ylabel={},
					symbolic x coords={5-10\%, 15-20\%, 25-30\%, 30-35\%},
					grid=major,major grid style={dashed},
					xlabel = {\normalsize Percentage of $\sigma_{\text{max}}$},
                    xlabel style = {yshift=-1mm},
					tick label style={font=\small},
					%x tick label style={rotate=10},
					legend style={at={(0.5,-0.38)},
						anchor=north,legend columns=-1, font=\footnotesize},
					nodes near coords,
					every node near coord/.append style={font=\footnotesize},
					nodes near coords align={vertical}
					]
        \addplot[black,fill=white,%pattern=north west lines, pattern color=red,
        every node near coord/.append style=
					{red,font=\fontsize{8}{12}\selectfont, xshift=-0mm, yshift=-1mm}] coordinates {(5-10\%, 95.8) (15-20\%, 135) (25-30\%, 143) (30-35\%, 137)};
		\addplot[black,fill=black!40,,every node near coord/.append style=
					{black,font=\fontsize{8}{12}\selectfont, yshift=-1mm, xshift=+1.5mm}] coordinates {(5-10\%,  95.4) (15-20\%, 135) (25-30\%, 143) (30-35\%, 137)};
    \end{axis}
    \begin{axis}[
					axis y line*=right,
					axis x line=none,
					width=6.6cm,
		              height=6.2cm,
					enlarge x limits=0.2,
					axis line style={-},
					nodes near coords,
					every node near coord/.append style={
						rotate=20,
						font=\bfseries\fontsize{11}{8}\selectfont,
						anchor=east,
					},
					ylabel = {\normalsize \textbf{\#Ins.Opt}},
					yticklabel style={
						font=\small\bfseries, rotate=20,
					},
					ylabel style={at={(axis description cs:1.22,0.5)},
						anchor=center},
					symbolic x coords={5-10\%, 15-20\%, 25-30\%, 30-35\%},
					tick label style={font=\small},
					scaled y ticks = false,
					legend cell align=left,
					legend style={at={(0.35,0.98)},
						anchor=north,legend columns=1, font=\fontsize{8}{8}\selectfont}
					]
        \addplot[color=black,mark=o,sharp plot] coordinates {(5-10\%,25) (15-20\%,7) (25-30\%,1) (30-35\%,1)};
		\legend{\#Ins.Opt}
    \end{axis}
\end{tikzpicture}}}

    \caption{The Minimum Number of Extra Forward Lanes Needed by the MB-RL and the MB for Resequencing Feasibility and the Corresponding Minimum Makespan. \label{fig4}}

\end{figure}
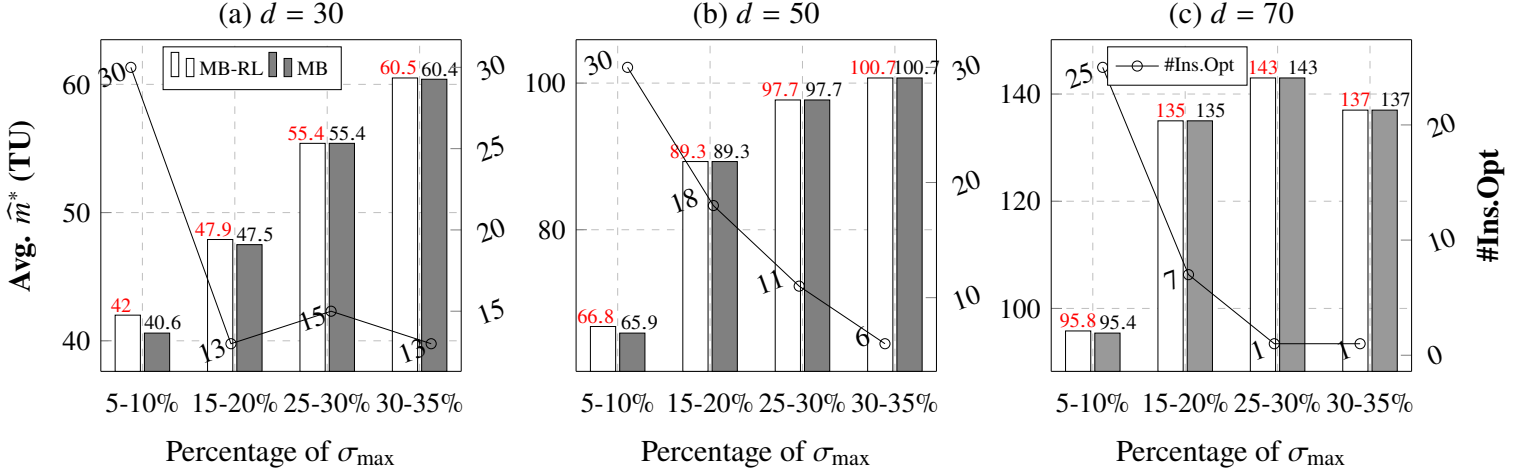

Then, for each instance in dataset B, we computed the minimum \textit{resequencing makespan} under both buffers. If the instance was feasible at the baseline buffer size, the makespan was computed there; otherwise, under the minimally expanded buffer that required the fewest extra forward lanes for feasibility. Figure \ref{fig4} shows the number of instances solved by BP-CP within one hour (``\#Ins.Opt'') and the average minimum makespan under each buffer across these instances (``Avg. $\widehat{m}^*$'').
%%%%Then, for each instance, we computed the minimum \textit{resequencing makespan} under both buffers: using BP-CP with $\mathcal{M}^{\text{LP}}_2$ for the MB-RL and following the last conclusion in Proposition \ref{pro3} for the MB. If the instance was feasible at the baseline buffer size, the makespan was computed there; otherwise, under the minimally expanded buffer that required the fewest extra forward lanes for feasibility. Figure \ref{fig4} shows the number of instances solved by BP-CP within one hour (``\#Ins.Opt'') and the average minimum makespan under each buffer across these instances (``Avg. $\widehat{m}^*$'').

\vspace{-1mm}

%\vspace{-4mm}

%%%%Then, for each instance, we computed the minimum \textit{resequencing makespan} under both buffers. If the instance was feasible at the baseline buffer size, we computed it at that size; otherwise, at the minimally expanded buffer that required the fewest extra forward lanes for feasibility. Using BP-CP with $\mathcal{M}^{\text{LP}}_2$ or following the last conclusion in Proposition \ref{pro3}, the makespans for the MB-RL and the MB were obtained, respectively. Figure 1 shows the number of instances solved by BP-CP within one hour (``\#Ins.Opt'') and the average minimum makespan under each buffer across these instances (``Avg.$\widehat{m}^*$'').

%%%%%%Then, for each instance in the extended dataset, the minimum makespan under the MB-RL and MB was determined using BP-CP and Proposition \ref{pro3}, respectively. If an instance was feasible with the baseline buffer size, the makespan was calculated at that size; otherwise, it was obtained under the smallest enlarged buffer that required the fewest extra forward lanes for feasibility. Figure \ref{fig2} shows the number of instances solved by BP-CP within one hour and the average minimum makespan under each buffer across these instances.

The results indicate that BP-CP solves fewer instances as resequencing complexity increases. At low complexity ($\sigma\leqslant 0.2\sigma_{\text{max}}$), the MB attains slightly shorter makespans than the MB-RL, but at the cost of requiring more forward lanes to maintain resequencing feasibility (as shown in Table \ref{tab3}). When $\sigma>0.2\sigma_{\text{max}}$ (the upstream and downstream sequences are nearly reversed), the MB-RL achieves makespans comparable to those of the MB while using fewer lanes, highlighting its robustness in achieving competitive makespans across various resequencing complexities, especially with limited buffer size. For managers, the RL supports a more compact and cost-effective buffer design without sacrificing resequencing efficiency.

\section{Conclusion}\label{sec6}
This study explores a \textit{car resequencing problem} where an advanced but less-studied buffer placed between two shops is used to rearrange the car sequence from one shop to the next. This buffer, called the \textit{mix bank with a return lane} (MB-RL), comprises several first-in-first-out forward lanes and one \textit{return lane} (RL) that operates in the opposite direction, allowing cars to loop within it. By scheduling car movements in space and time, the goal is to minimize the total time needed to change the sequence.

We prove the strong NP-completeness %strong NP-hardness
of this \textit{resequencing makespan} minimization problem and provide methods to check resequencing feasibility. Using a \textit{time-space network} (TSN), a two-stage \textit{branch-and-price approach integrated with constraint programming} (BP-CP) is developed for exact solutions. The first stage uses a \textit{column-generation} (CG) scheme to iteratively solve the \textit{master problem} (MP) with \textit{pricing subproblems}. Once CG converges, branching enforces integrality on only a subset of variables. Their integer values define a \textit{constraint-programming} model in the second stage to find a feasible solution to the original problem. To reduce exhaustive branching due to problem symmetry, we propose two reduced MPs based on compressed TSN graphs. A valid lower bound and a preprocessing technique further improve BP-CP performance.

Computational results show that the BP-CP approach significantly increases the scale of instances solvable for a one-hour production batch to 120 cars, nearly double the real‑world scale. It performs best with a moderately relaxed MP, which strikes an excellent balance between solution time and lower-bound quality. Structurally, adding more forward lanes to the MB-RL eases the problem and reduces the minimum makespan, whereas increasing lane capacity has the opposite effect. Compared to the traditional buffer without an RL, the MB-RL exhibits superior resequencing capability and achieves competitive makespans with fewer forward lanes across various resequencing complexities. Managerially, incorporating an RL can nearly match the resequencing capability of several forward lanes. It enables a smaller buffer footprint and provides a cost-effective, robust buffer design for \textit{just-in-time} production, where resequencing demands are frequent and diverse.

\section*{Acknowledgment}
We acknowledge financial support from the Canadian Natural Sciences and Engineering Research Council (NSERC) under Grant 2026-05294. We also thank the Digital Research Alliance of Canada for providing high-performance computing facilities.

\bibliographystyle{plainnat}
\bibliography{myref}

\newpage

%\begin{center}
%{\LARGE\bfseries Online Supplementary}
%\end{center}
\noindent{\Large\bfseries Online Supplementary}
\vspace{-0.5cm}

\appendix

\section{Proof of the Strong NP-completeness of the CRSP-MS}
\begin{theorem}\label{app-theo1}
    The CRSP-MS is NP-complete in the strong sense, even if no car enters the RL.
\end{theorem}
\begin{proof}

    We establish the strong NP-completeness of the CRSP-MS by considering a special case in which no car enters the RL. This case is a subproblem of the CRSP-MS with all RL-entry decisions already fixed to zero. Consequently, the RL becomes redundant, and the problem reduces to resequencing via an MB consisting only of forward lanes. The proof is based on a reduction from the \textit{numerical 3-dimensional matching problem} (N3DM), which is a feasibility problem known to be strongly NP-complete \citep{garey2002computers}.
    %We establish the strong NP-completeness of the CRSP-MS by considering a special case in which no car enters the RL. This case corresponds to a subproblem of the CRSP-MS with the RL-entry decisions for all cars already fixed to zero. The proof is based on a reduction from the \textit{numerical 3-dimensional matching} problem (N3DM), which is a feasibility problem known to be NP-complete in the strong sense \citep{garey2002computers}.
    %The result follows from a reduction from the 3-Partition Problem (3PP), a feasibility problem known to be NP-hard in the strong sense.

    %We prove Theorem \ref{app-theo1} by reducing the \textit{3-Partition Problem}, well known to be NP-hard in the strong sense, to the CRSP-MS. The \textit{3-Partition Problem} is a feasibility problem whose answer is either \textit{yes} or \textit{no}, and it is defined as follows:

    Given an integer $B$ and three multisets $X=\{x_1,x_2,\dots,x_n\}$, $Y=\{y_1,y_2\dots,y_n\}$, and $z=\{z_1,z_2,\dots,$ $z_n\}$, each containing $n$ elements, the N3DM asks whether there exists a set $M\subseteq X\times Y\times Z$ such that each element of $X$, $Y$, and $Z$ belongs to exactly one triple and every triple $(x,y,z)\in M$ satisfies $x+y+z=B$.

    We now describe the reduction from the N3DM to the special case of the CRSP-MS. For a given N3DM instance, we construct a CRSP-MS instance with $d=1+nB$ cars to be resequenced via an MB consisting of $l^f=n+1$ forward lanes, each with capacity $\bar{q}=B$. Note that such constructed $d$ and $\bar{q}$ make our reduction depend on the integer values of the N3DM and, thus, pseudo-polynomial. The last upstream car $d$ is designated as the first car in the downstream shop and therefore has downstream position $p_d=1$. The remaining $nB$ cars are partitioned into $3n$ consecutive blocks. For each $i = 1,2,\dots,n$, integer $x_i\in X$, $y_i\in Y$, and $z_i\in Z$ induce an \textit{X-car} block $V^X_i$, a \textit{Y-car} block $V^Y_i$, and a \textit{Z-car} block $V^Z_i$, which contain $x_i$, $y_i$, and $z_i$ consecutive cars, respectively. The downstream position $p_k\in \{2,\dots,nB+1\}$ of each car $k\in \{1,\dots,nB\}$ in these blocks is defined as follows:

    \begin{itemize}[topsep=2pt]
        %\item $p_k=1$, for the \textit{last car} $k=d$,

        \item $p_k=1+c+\sum_{i^{\prime}=i+1}^{n}{x_{i^{\prime}}}$, for the \textit{X-car} $k=c+\sum_{i^{\prime}=1}^{i-1}{x_{i^{\prime}}}$, where $i=1,2,\dots, n$, $c=1,2,\dots,x_i$,

        \item $p_k=1+c+\sum_{x\in X}{x}+\sum_{i^{\prime}=i+1}^{n}{y_{i^{\prime}}}$, for the \textit{Y-car} $k=c+\sum_{x\in X}{x}+\sum_{i^{\prime}=1}^{i-1}{y_{i^{\prime}}}$, where $i=1,2,\dots,n$, $c=1,2,\dots,y_i$,

        \item $p_k=1+c+\sum_{x\in X}{x}+\sum_{y\in Y}{y}+\sum_{i^{\prime}=i+1}^{n}{z_{i^{\prime}}}$, for the \textit{Z-car} $k=c+\sum_{x\in X}{x}+\sum_{y\in Y}{y}+\sum_{i^{\prime}=1}^{i-1}{z_{i^{\prime}}}$, where $i=1,2,\dots,n$, $c=1,2,\dots,z_i$.

        %%\item $p_k=\sum_{i^{\prime}=1}^{i-1}{a_{i^{\prime}}}+i\cdot (n-1)-(j-1)+1$, for the \textit{delimiter cars} $k=\sum_{i^{\prime}=1}^{i}{a_{i^{\prime}}}+(i-1)\cdot (n-1)+j$, where $i=1,2,\dotsm3n$, $j=1,2,\dots,n-1$,

        %%\item $p_k=k+n$, for the \textit{non-conflicting cars} $k=1+\sum_{i^{\prime}=1}^{i-1}{a_{i^{\prime}}}+(n-1)\cdot (i-1),\dots,\sum_{i^{\prime}=1}^{i}{a_{i^{\prime}}}+(n-1)\cdot (i-1)$, where $i=1,2,\dots,3n$.

    \end{itemize}

    %\begin{itemize}
    %    \item \textcolor{red}{$p_k=d$, for $k=d$}
    %    \item \textcolor{red}{$p_k=k+n$, for $k=1+\sum_{i^{\prime}=1}^{i-1}{a_{i^{\prime}}+(n-1)(i-1)},\dots,\sum_{i^{\prime}=1}^{i}{a_{i^{\prime}}}+(n-1)(i-1)$ (non-critical cars)}
    %    \item \textcolor{red}{$p_k=\sum_{i^{\prime}=1}^{i-1}{a_{i^{\prime}}}+i(n-1)-(j-1)+1$, for $k=\sum_{i^{\prime}=1}^{i}{a_{i^{\prime}}}+(i-1)(n-1)+j$ (``delimiter cars''), where $i=1,2,\dots,3n$, $j=1,2,\dots,n-1$}
    %\end{itemize}

    The corresponding car blocks are given by
    \begin{itemize}[topsep=2pt]
        \item $V^X_i=\{k=c+\sum_{i^{\prime}=1}^{i-1}{x_{i^{\prime}}}\mid c=1,2,\dots,x_i\}$,
        \item $V^Y_i=\{k=\sum_{x\in X}{x}+c+\sum_{i^{\prime}=1}^{i-1}{y_{i^{\prime}}}\mid c=1,2,\dots,y_i\}$,
        \item $V^Z_i=\{k=\sum_{x\in X}{x}+\sum_{y\in Y}{y}+c+\sum_{i^{\prime}=1}^{i-1}{z_{i^{\prime}}}\mid c=1,2,\dots,z_i\}$, $i=1,2,\dots,n$.
    \end{itemize}

    Figure \ref{fig1-1} illustrates this reduction, and Figure \ref{fig1-2} presents the corresponding feasible solutions for both instances. The answer to the N3DM instance is \textit{yes} if and only if the CRSP-MS instance admits a feasible schedule, or car-to-lane assignment, with a makespan no greater than $1+(2n+1)B$ TUs.

    \vspace{-2mm}

    \begin{figure}[htp]
        \centering
        \makebox[\textwidth][c]{\includegraphics[width=1.14\linewidth]{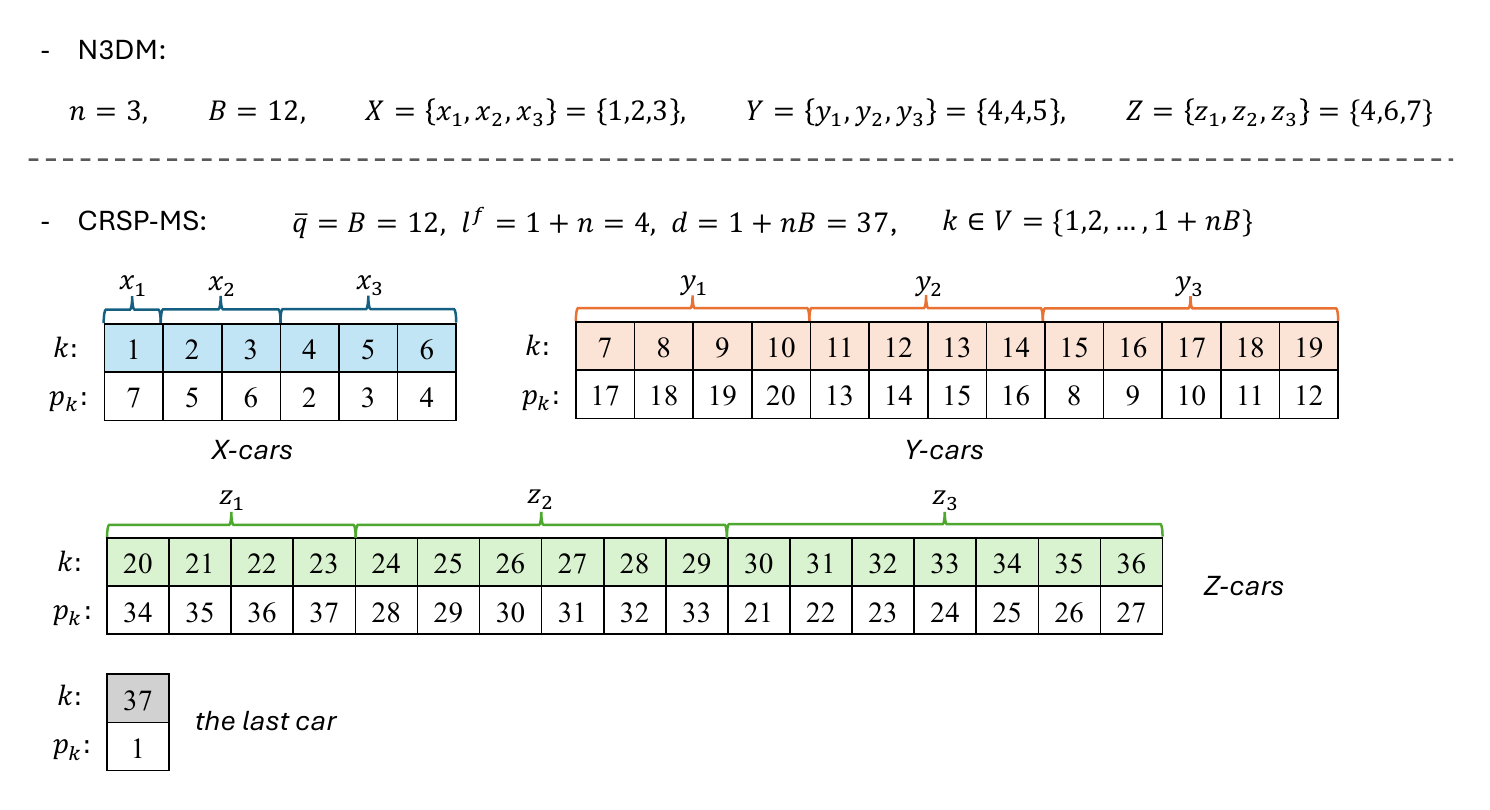}}
        \caption{Illustration of the Reduction from the N3DM to the CRSP-MS.}
        \label{fig1-1}
    \end{figure}

    \begin{figure}[htp]
        \centering
        \makebox[\textwidth][c]{\includegraphics[width=0.84\linewidth]{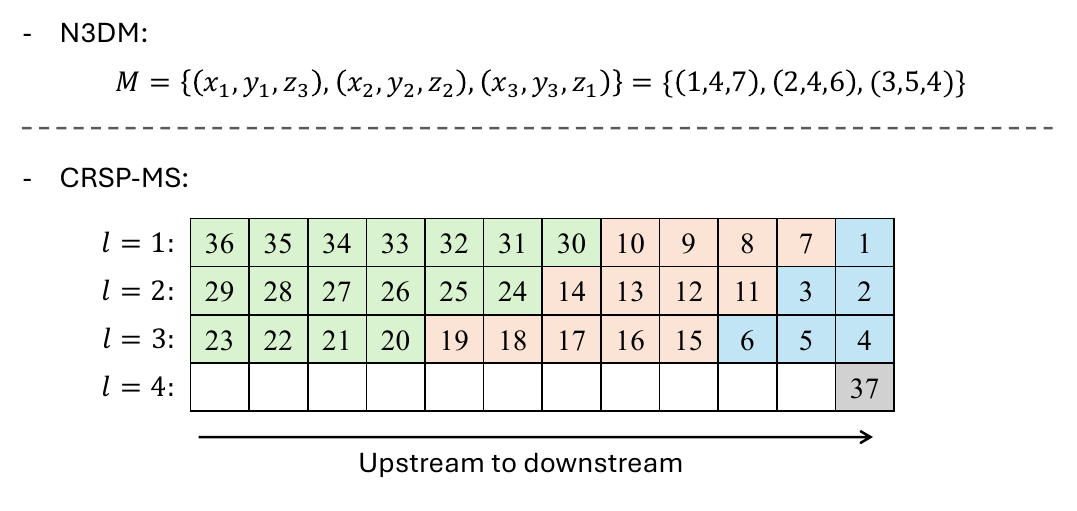}}
        \caption{Illustration of the Feasible Solutions for the N3DM and the CRSP-MS Instances.}
        \label{fig1-2}
    \end{figure}

    By construction, the order of downstream positions within each car block is consistent with that of upstream indexes. Therefore, all cars from the same group can be assigned to a common lane. However, for any two distinct \textit{X-car} blocks $V^X_i$ and $V^X_j$ with $i<j$, every car in $V^X_i$ has a smaller upstream index but a larger downstream position than every car in $V^X_j$. The two blocks are thus completely conflicting and cannot share a lane under the FIFO rules. The same property also applies to \textit{Y-car} groups and \textit{Z-car} groups. As a result, each lane contains at most one \textit{X-car} group, one \textit{Y-car} group, and one \textit{Z-car} group.

    $\Longrightarrow$ A feasible matching $M$ of the N3DM instance induces a feasible solution to the CRSP-MS instance with a makespan of at most $1+(2n+1)B$ TUs. By construction, the last car $d$ is the first to enter the downstream shop and thus occupies one lane alone. Before it leaves the MB, all preceding $nB$ cars must be stored in the other $n$ lanes. For each triple $(x_i,y_j,z_s)\in M$ ($i,j,s\in \{1,\dots,n\}$), assign all cars in $V^X_i$, $V^Y_j$, and $V^Z_s$ to the same lane in sequence. Since $x_i+y_j+z_s=B$, this lane is filled exactly to its capacity. Because every integer in $X$, $Y$, and $Z$ appears in exactly one triple of $M$, every car block is assigned to exactly one lane, and all $n$ lanes are fully occupied. The earliest timestamp at which car $d$ reaches the downstream is $1+(n+1)B$. Thereafter, the remaining $nB$ cars can continuously follow it to the downstream without idle time. The resulting schedule achieves a makespan of $1+(2n+1)B$ TUs.

    $\Longleftarrow$ A feasible schedule with a makespan no greater than $1+(2n+1)B$ TUs for the CRSP-MS instance implies a \textit{yes}-certificate for the N3DM instance. %induces a feasible matching for the N3DM instance.
    In any feasible assignment, as stated previously, the $nB$ cars preceding car $d$ must be stored in the $n$ lanes not used by $d$ before $d$ leaves the MB. These $n$ lanes have a total capacity of $nB$, and each lane must be filled to its full capacity. Moreover, each lane can contain at most one \textit{X-car} block, one \textit{Y-car} block, and one \textit{Z-car} block. As there are exactly $n$ lanes and $n$ blocks of each type, every block must appear in exactly one lane, and each lane stores exactly one block of each type. Consider an arbitrary lane that stores the blocks $V^X_i$, $V^Y_j$, and $V^Z_s$. The numbers of these blocks satisfy $x_i + y_j + z_s = B$ because the lane is full, which then defines a valid triple $(x_i, y_j, z_s)$. Because every block is assigned to exactly one lane, the $n$ lanes induce $n$ disjoint triples that collectively cover all integers of $X$, $Y$, and $Z$. These triples constitute a feasible solution to the N3DM instance.

    Because the N3DM is strongly NP-complete, $B$ is polynomially bounded in the input size. Therefore, the constructed CRSP-MS instance has polynomial size, and the reduction is actually polynomial. This proves that the CRSP-MS is NP-complete in the strong sense.

    %%$\Longleftarrow$ A feasible scheduling scheme for the CRSP-MS instance with a makespan no greater than $1+(2n+1)\cdot[B+3(n-1)]$ TUs implies a \textit{yes}-certificate for the 3PP instance. %or yields a feasible partition of the 3PP instance.
    %A feasible car-to-lane assignment that ensures the completion time for rearranging all cars does not exceed $1+(2\cdot n+1)\cdot \left[B +3\cdot(n-1) \right] $ time units implies a feasible partition scheme, i.e., a \textit{yes}-certificate for the \textit{3-Partition Problem}.

\end{proof}

\section{Proof of Proposition 1}
\begin{proposition}
Consider car $\pi_{p}\in V$ in the downstream sequence $\Pi^{D}=[\pi_1,\dots,\pi_p,\dots,\pi_{d}]$. If $\pi_p$ has the largest upstream index among the first $p\in P$ cars in $[\pi_1,\pi_2,\dots,\pi_p]$, i.e., $\pi_p=\max\{\pi_1,\pi_2,\dots,\pi_p\}$, then there exists an optimal solution to the CRSP-MS in which car $\pi_d$ never enters the RL and is said to be RL-unused.
%then requiring car $\pi_p$ to enter the RL is a dominated decision. Hence, there exists an optimal solution to the CRSP-MS in which car $\pi_p$ never enters the RL and is said to be RL-unused.
\end{proposition}
\begin{proof}
    Define $A_p=\{\pi_1,\pi_2,\dots,\pi_p\}$. Because $\pi_p=\max A_p$, every car $k\in A_p\backslash\{\pi_p\}$ has an upstream index smaller than $\pi_p$ and appears earlier in the downstream shop, i.e., $k<\pi_p$ and $p_k<p$. Car $\pi_p$ therefore satisfies the FIFO rules with all cars in $A_p\backslash \{\pi_p\}$. None of these cars that enter the MB-RL earlier will overtake car $\pi_p$, so there is no need for car $\pi_p$ to enter the RL to allow them to reach downstream first.

    Now consider any car $k>\pi_p$. Its downstream position $p_k$ must satisfy $p_k>p$. Otherwise, if $p_k<p$, then car $k$ would belong to $A_p$, contradicting $\pi_p=\max A_p$. Hence, $p_k>p$ holds for all cars $k>\pi_p$. As a result, these cars also respect FIFO rules with $\pi_p$. None of these later cars need to overtake $\pi_p$, and requiring car $\pi_p$ to enter the RL cannot create additional resequencing opportunities for them.

    Consequently, no car forces car $\pi_p$ to change its relative order during resequencing. Any feasible solution with car $\pi_p$ entering the RL can be transformed into another feasible solution in which car $\pi_p$ either delays its entry into the MB or remains in a forward lane throughout the process. This transformation still allows for the rearrangement of other conflicting cars without increasing the makespan. Therefore, there always exists an optimal solution in which car $\pi_p$ never enters the RL.
    %Consequently, no car requires car $\pi_p$ to change its relative order during the resequencing process. Any feasible solution in which car $\pi_p$ enters the RL can be transformed into another feasible solution in which car $\pi_p$ remains in a forward lane throughout the process. This latter option apparently leads to no longer makespan. Hence, requiring car $\pi_p$ to enter the RL is a dominated decision, and there always exists an optimal solution in which car $\pi_p$ never enters the RL.
\end{proof}

\section{The Resequencing Feasibility Problem (RSFP)}\label{app-sec2}
Formally, given an upstream sequence $\Pi^U$, a downstream sequence $\Pi^D$, and a physical buffer (e.g., an MB or an MB-RL), the \textit{resequencing feasibility problem} (RFSP) determines whether $\Pi^U$ can be adjusted to $\Pi^D$ using the buffer.

%\subsection{Strong NP-hardness}
%\subsection{A Polynomial-time Algorithm for the MB-RL Case}
\subsection{MB-RL Case: A Polynomial-time Algorithm}
The RSFP with an MB-RL can be solved in $\mathcal{O}(d^2)$ time using Algorithm \ref{alg1}, where $d$ is the number of cars.
\begin{algorithm}[htp]\label{alg1}
    \LinesNumbered
    \caption{Resequencing Feasibility Check for the \textit{RSFP} with an MB-RL}
    \Input{\textit{ins} - an \textit{RSFP} with an MB-RL ($\bar{l}$ lanes, each with $\bar{q}$ cells) problem instance.}
    \Output{\textit{flag} - feasibility status of the instance; \textit{$c^*$} - the minimum MB-RL size required.}

    $boolean$ $flag \gets true$; $int$ $c^{*} \gets 0$\;

    \For{$k=2$ \KwTo $d$}{
        $\overline{V}(k) \gets \varnothing$\;
        \For{$v=1$ \KwTo $k-1$}{
           \lIf{$p_v>p_k$}{$\overline{V}(k) \gets \overline{V}(k)\cup \{v\}$}}
        $c^{*}\gets\max\{c^{*},
        \big|\overline{V}(k)\big| \}$\;
        \lIf{$flag=false$}{continue}
        \lIf{$\big|\overline{V}(k)\big|\geqslant \bar{l}\times \bar{q}$}{$flag \gets false$}
    }
    \Return{$flag$, $c^{*}$}
\end{algorithm}

Notably, Algorithm \ref{alg1} also determines the minimum MB-RL size $c^*$ required to transform from $\Pi^D$ to $\Pi^D$. If $c^*$ is undesired, the ``continue'' in line \textbf{7} can be replaced with ``break''.

\subsection{MB Case: the Assignment Model of \citet{guo2025logic}}
\begin{proposition}\label{app-pro2}
    The RSFP with an MB is NP-complete in the strong sense.
\end{proposition}
\begin{proof}
Based on a reduction from the N3DM, we construct an RFSP instance in exactly the same way as the special case of the CRSP-MS in the proof of Theorem \ref{app-theo1}. The N3DM instance is feasible if and only if the RFSP instance admits a feasible car-to-lane assignment, or equivalently, if the sequence change from $\Pi^U$ to $\Pi^D$ can be achieved via the MB. %The correctness of the reduction follows directly from the proof of Theorem \ref{app-theo1}.
The proof therefore follows directly from that of Theorem \ref{app-theo1}.
\end{proof}

The RSFP with an MB can be solved using the car-to-lane assignment model of \citet{guo2025logic}. In their model, the binary variable $h^k_l$ if car $k\in V$ is assigned to forward lane $l\in L^f$ and 0 otherwise. The set $\mathcal{I}=\{(k_1,k_2)\in V\times V \mid \text{car } k_1 \text{ conflicts with car } k_2, k_1\neq k_2\}$ contains all pairs of \textit{conflicting} cars, and the binary parameter $\rho^k_v$ indicates whether car $v\in V, v\leqslant k$ is stored in the MB when car $k\in V$ enters, with its value derived from $\Pi^D$. With these definitions, their assignment model, which is a feasibility problem to assign each car $k\in V$ to a forward lane subject to FIFO rules and lane capacity, is given by: %With these definitions, their assignment model, which determines the forward lane entered by each upstream car $k\in V$, is given by:
\begin{subequations}\label{m1}
    \begin{flalign}
        \label{con1}\qquad \qquad \qquad &\sum_{l\in L^f}{h^k_l}=1, &&k\in V, \\
        \label{con2}&h^{k_1}_l+h^{k_2}_l\leqslant 1, &&l\in L^f, (k_1,k_2)\in \mathcal{I}, \qquad \\
        \label{con3}&\sum_{v=1}^{k}{\rho^{k}_v \cdot h^v_l}\leqslant \bar{q}, &&l\in L^f, k\in V, \\
        \label{con4}&h^k_l\in \{0,1\}, &&l\in L^f, k\in V.
    \end{flalign}
\end{subequations}

Constraints (\ref{con1}) ensure that each car enters exactly one forward lane. Constraints (\ref{con2}) require cars to comply with the FIFO rules. When car $k\in V$ enters lane $l$, one constraint in (\ref{con3}) ensures that the number of cars still in lane $l$ cannot exceed capacity $\bar{q}$.

After solving model \ref{m1}, if it yields a feasible assignment, the given MB can successfully achieve the sequence change from $\Pi^U$ to $\Pi^D$. If it is infeasible, the sequence change cannot be accomplished.

\section{A Multi-Commodity Network Flow Model}\label{app-sec3}
Based on the \textit{time-space network} (TSN) $\mathcal{G}=(\mathcal{N},\mathcal{A})$, we formulate the CRSP-MS as a multi-commodity network flow model. For each car $k\in V$, let binary variable $x^k_{itjt^{\prime}}$ indicate whether it traverses arc $(i,t;j,t^{\prime})\in \mathcal{A}$. The compact formulation is as follows:

\begin{subequations}
    \label{M1}
    \begin{flalign}
    \label{m1-obj1}\min& \sum_{(i,t;e,t^{\prime})\in \mathcal{A}^{e}}{t^{\prime}\cdot x^{\pi_d}_{itet^{\prime}}} \\
    s.t.&\quad \nonumber \\
    \label{app-m1-con1}& \sum_{(s,t;j,t^{'})\in \mathcal{A}^{s}}{x^{k}_{stjt^{'}}}=1, &&k\in V, \\
    \label{app-m1-con2}& \sum_{(i,t;e,t^{\prime})\in \mathcal{A}^{e}}{x^k_{itet^{\prime}}}=1, &&k\in V, \\
    \label{app-m1-con3}&\sum_{(s,t;j,t^{'})\in \mathcal{A}^{s}}{x^{k}_{stjt^{'}}}\leqslant \sum_{t^{'}=0}^{t}{\sum_{(s,t^{'};j,t^{''})\in \mathcal{A}^{s}}{x^{k-1}_{st^{'}jt^{''}}}}, &&k\in V\backslash \{1\}, t\in \{0,\dots,\tau-\tau^m\}, \\
    \label{app-m1-con4}& \sum_{(i,t^{\prime};e,t)\in \mathcal{A}^{e}}{x^v_{it^{\prime}et}}\leqslant \sum_{t^{\prime}=\tau^m}^{t}{\sum_{(i,t^{''};e,t^{\prime})\in \mathcal{A}^{e}}{x^k_{it^{''}et^{'}}}}, &&k,v\in V, p_k<p_v\in P, t\in \{\tau^m,\dots,\tau\}, \\
    \label{app-m1-con5}& \sum_{(i,t;j,t^{\prime})\in \mathcal{A}}{x^{k}_{itjt^{\prime}}}-\sum_{(j,t^{\prime};i,t)\in \mathcal{A}}{x^{k}_{jt^{\prime}it}}=0, &&k\in V, (i,t)\in \mathcal{N}^{f}\cup \mathcal{N}^{r}, \\
    \label{app-m1-con6}&\sum_{k\in V}{\sum_{(i,t^{\prime};j,t)\in \mathcal{A}}{x^k_{it^{\prime}jt}}}\leqslant 1, &&(j,t)\in \mathcal{N}^{f}\cup \mathcal{N}^{r}, \\
    \label{app-m1-con7}&\sum_{k\in V}{\big(
    \sum_{(s,t;j,t^{\prime})\in \mathcal{A}^{s}}{x^k_{stjt^{\prime}}}+\sum_{(i,t;j,t^{\prime})\in \mathcal{A}^{fr}}{x^k_{itjt^{\prime}}}\big)} \leqslant 1, &&t\in \{0,\dots,\tau-\tau^{m}\}, \\
    \label{app-m1-con8}&\sum_{k\in V}{\big(\sum_{(i,t^{\prime};e,t)\in \mathcal{A}^{e}}{x^k_{it^{\prime}et}}+\sum_{(i,t^{\prime};j,t)\in \mathcal{A}^{rf}}{x^k_{it^{\prime}jt}} \big)}\leqslant 1, &&t\in \{\tau^m,\dots,\tau\}, \\
    \label{app-m1-con9}&x^{k}_{itjt^{\prime}}\in \{0,1\}, &&k\in V, (i,t;j,t^{\prime})\in \mathcal{A}.
    \end{flalign}
\end{subequations}

The objective function \eqref{m1-obj1} minimizes the \textit{resequencing makespan}. Constraints \eqref{app-m1-con1} and \eqref{app-m1-con2} ensure that all cars depart from the upstream shop and arrive at the downstream shop. The given departure and arrival orders (i.e., $\Pi^U$ and $\Pi^D$) are guaranteed by constraints \eqref{app-m1-con3} and \eqref{app-m1-con4}. Flow conservation at each \textit{FL node} in $\mathcal{N}^f$ and each \textit{RL node} in $\mathcal{N}^r$ is enforced by constraints \eqref{app-m1-con5}, and constraints \eqref{app-m1-con6} further restrict the inflows to such a node to at most one, so as to satisfy both FIFO rules and the capacity limit for each lane. Constraints \eqref{app-m1-con7} and \eqref{app-m1-con8} respectively ensure that at any timestamp, at most one car can enter and leave any forward lane. Constraints \eqref{app-m1-con9} define the variable domains.

\section{Detailed Arc and Variable Definitions}\label{app-sec4}
For each car $k\in V$, given its earliest departure time $\tau^s_k$ from the upstream shop and its latest arrival timestamp $\tau^e_k$ at the downstream shop, this section formally defines the arc set $\mathcal{A}^1_k$ of the TSN graph $\mathcal{G}^1_k$ and specifies the domains of all decision variables in the \textit{constraint programming} (CP) model $\mathcal{M}^{\text{CP}}$.

\subsection{Arc Set $\mathcal{A}^1_k$}
The arc set $\mathcal{A}^1_k$ of $\mathcal{G}^1_k$ for car $k\in V$ is defined as $\mathcal{A}^1_k=\mathcal{A}^s_k\cup \mathcal{A}^e_k\cup \mathcal{A}^w_k\cup \mathcal{A}^m_k$, where

\begin{enumerate}[topsep=1pt, itemsep=1pt]
    \renewcommand{\labelenumi}{(\roman{enumi})}
    \item \textit{pull-in arcs} $\mathcal{A}^s_k=\{(s,t;c_{l1},t+1)\mid (s,t)\in \mathcal{N}^s_k, (c_{l1},t+1)\in \mathcal{N}^f_k\}$,
    \item \textit{pull-out arcs} $\mathcal{A}^e_k=\{(c_{l\bar{q}},t-1;e,t)\mid (c_{l\bar{q}},t-1)\in \mathcal{N}^f_k,(e,t)\in \mathcal{N}^e_k\}$,
    \item \textit{waiting arcs} $\mathcal{A}^w_k=\{(c_{lq},t;c_{lq},t+1)\mid (c_{lq},t),(c_{lq},t+1)\in \mathcal{N}^f_k\cup \mathcal{N}^r_k\}$,
    \item \textit{moving arcs} $\mathcal{A}^m_k=\mathcal{A}^{mc}_k\cup \mathcal{A}^{fr}_k\cup \mathcal{A}^{rf}_k$,
    $\mathcal{A}^{mc}_k=\{(c_{lq},t;c_{l,q+1},t+1)\mid (c_{lq},t),(c_{l,q+1},t+1)\in \mathcal{N}^f_k\}\cup \{(c_{\bar{l}q},t;c_{\bar{l},q-1},t+1)\mid (c_{\bar{l}q},t),(c_{\bar{l},q-1},t+1)\in \mathcal{N}^r_k\}$, $\mathcal{A}^{fr}_k=\{(c_{l\bar{q}},t;c_{\bar{l}\bar{q}},t+1)\mid (c_{l\bar{q}},t)\in \mathcal{N}^f_k,(c_{\bar{l}\bar{q}},t+1)\in \mathcal{N}^r_k\}$, and $\mathcal{A}^{rf}_k=\{(c_{\bar{l}1},t;c_{l1},t+1)\mid (c_{\bar{l}1},t)\in \mathcal{N}^r_k,(c_{l1},t+1)\in \mathcal{N}^f_k\}$.
\end{enumerate}

\subsection{Variable Domains in $\mathcal{M}^{\text{CP}}$}
The decision variables in $\mathcal{M}^{\text{CP}}$ are grouped into \textit{interval variables} $y^k_n$, $\widetilde{y}^{l}_{kn}$, $g^k_n$ and integer variables $\phi^{in}_{kn}$, $\phi^{out}_{kn}$. For each \textit{interval variable} in $\{y^k_n\}\cup \{\widetilde{y}^l_{kn}\mid l\in L^f\}$, its \textit{start time}, representing the arrival timestamp of \textit{dummy car} $o^k_n\in \mathcal{D}$ at the leftmost cell of a forward lane, shares the same domain as the integer variable $\phi^{in}_{kn}$. Likewise, its \textit{end time}, indicating the departure timestamp of $o^k_n$ from the rightmost cell of a forward lane, shares the domain with the integer variable $\phi^{out}_{kn}$. Therefore, for each \textit{dummy car} $o^k_n\in \mathcal{D}^k$ ($\forall k\in V$), we specify below how to determine the domains of only the \textit{start} and \textit{end times} for the \textit{interval variables} $y^k_n$ and $g^k_n$, where $g^k_n$ denotes the $n$-th ($n\in \{1,\dots,\widehat{n}_k\}$) entry of \textit{dummy car} $o^k_n$ into the RL.

We first determine the range of the arrival timestamp $\tau^e_k$ at the downstream shop for each original car $k\in V$. Given the makespan $\widehat{m}$, i.e., the arrival timestamp $\tau^e_{\pi_d}$ of the last car $\pi_d$ ($\widehat{m}=\tau^e_{\pi_d}$), the arrival timestamp $\tau^e_k$ for each car $k\in V\backslash\{\pi_d\}$ lies in $\left[\underline{t}^e_k,\overline{t}^e_k\right]$, where
\begin{equation*}
    \underline{t}^e_k = \max \left\{\underline{\tau}^e_{k}, k + (2\cdot\widehat{n}_k + 1)\cdot \bar{q}\right\}, \qquad \overline{t}^e_k=\widehat{m}-(d-p_k).
\end{equation*}
\noindent Here, $\underline{\tau}^e_{k}$ is the lower bound on $\tau^e_k$ obtained from Proposition 3, $\widehat{n}_k$ is the number of times car $k$ enters the RL, and $p_k$ is its downstream position.

Then, the departure timestamp $\tau^s_k$ from the upstream shop for car $k\in V$ falls in $\left[\underline{\tau}^s_k,\overline{\tau}^s_k\right]$ with
\begin{equation*}
    \underline{\tau}^s_k = k - 1, \qquad \overline{\tau}^s_k=\overline{\tau}^e_k-(2\cdot \widehat{n}_k + 1) \cdot \bar{q} - 1.
\end{equation*}
By Corollary 1, if $k\in \{1,2,\dots, \pi_1\}$, then $\overline{\tau}^s_k=\underline{\tau}^s_k=k-1$.

After that, for each \textit{interval variable} $y^k_n$ ($k\in V, n\in \{1,\dots,\widehat{n}_k+1\}$), let $\left[\underline{t}_{kn}^{1},\overline{t}^{1}_{kn}\right]$ and $\left[\underline{t}_{kn}^{2},\overline{t}^{2}_{kn}\right]$ denote the domains of its \textit{start} and \textit{end times}, respectively, as determined by the following procedure:

%%After that, for each \textit{interval variable} $y^k_n$ ($k\in V, n\in \{1,\dots,\widehat{n}_k+1\}$), let $\left[\underline{t}_{kn}^{1},\overline{t}^{1}_{kn}\right]$ and $\left[\underline{t}_{kn}^{2},\overline{t}^{2}_{kn}\right]$ denote the domains of its \textit{start} and \textit{end times}, respectively. For each \textit{interval variable} $g^k_n$, the corresponding domains are $\left[\underline{t}_{kn}^{3},\overline{t}^{3}_{kn}\right]$ (\textit{start time}) and $\left[\underline{t}_{kn}^{4},\overline{t}^{4}_{kn}\right]$ (\textit{end time}). These domains can be determined using the following procedures:

\begin{algorithm}[htp]\label{alg2}
    \LinesNumbered
    \caption{Procedure to Obtain the Domains $[\underline{t}_{kn}^{1},\overline{t}^{1}_{kn}]$ and $[\underline{t}_{kn}^{2},\overline{t}^{2}_{kn}]$}
    \Input{$\left[\underline{t}^e_k,\underline{t}^e_k \right]$ - domain of $\tau^e_k$; $\left[\underline{t}^s_k,\underline{t}^s_k \right]$ - domain of $\tau^s_k$.}
    \Output{$[\underline{t}_{kn}^{1},\overline{t}^{1}_{kn}]$; $[\underline{t}_{kn}^{2},\overline{t}^{2}_{kn}]$.}

    %$boolean$ $flag \gets true$; $int$ $c^{*} \gets 0$\;

    \For{$k=1$ \KwTo $d$}{
        \For{$n=1$ \KwTo $\widehat{n}_k+1$}{
        $\underline{t}_{kn}^{1}=\underline{\tau}^s_k+1+2\cdot (n-1) \cdot \bar{q}$\;
        $\underline{t}_{kn}^{2}=\underline{t}_{kn}^{1}+\bar{q}-1$\;
        \lIf{$n=\widehat{n}_k+1$}{$\underline{t}_{kn}^{2}=\max \{\underline{t}_{kn}^{2}, \,\ \underline{\tau}^e_k - 1\}$}
        $\overline{t}_{kn}^{2}=\overline{\tau}^e_k-1-2\cdot (\widehat{n}_k + 1 - n) \cdot \bar{q}$\;
        $\overline{t}_{kn}^{1}=\overline{t}_{kn}^{2}-\bar{q}+1$\;
        \lIf{$n=1$}{$\overline{t}_{kn}^{1}=\min \{\overline{t}_{kn}^{1}, \,\ \overline{\tau}^s_k+1\}$}
        }
    }
    \Return{$[\underline{t}_{kn}^{1},\overline{t}^{1}_{kn}]$, $[\underline{t}_{kn}^{2},\overline{t}^{2}_{kn}]$}
\end{algorithm}

\noindent Based on $\left[\underline{t}_{kn}^{1},\overline{t}^{1}_{kn}\right]$ and $\left[\underline{t}_{kn}^{2},\overline{t}^{2}_{kn}\right]$, the corresponding domains for each \textit{interval variable} $g^k_n$ ($k\in V$) are $\left[\underline{t}_{kn}^{3},\overline{t}^{3}_{kn}\right]$ (\textit{start time}) and $\left[\underline{t}_{kn}^{4},\overline{t}^{4}_{kn}\right]$ (\textit{end time}), where
\begin{align*}
    &\underline{t}_{kn}^{3}= \underline{t}^2_{kn}+1,\qquad &&\overline{t}_{kn}^{3}=\overline{t}^{2}_{kn}+1, \\
    &\underline{t}^{4}_{kn}=\underline{t}^{1}_{k,n+1}-1, \qquad &&\overline{t}^{4}_{kn}=\overline{t}^1_{k,n+1}-1.
\end{align*}

\section{Reduced Cost Calculation and Labeling Algorithm}\label{app-sec5}
Because no new constraints are introduced when reducing $\mathcal{M}_2$ to $\mathcal{M}_3$, only the following dual variables related to $\mathcal{M}_1^{\text{LP}}$ and $\mathcal{M}_2^{\text{LP}}$ are considered when calculating the reduced cost of a \textit{path} $r\in \Omega_k$ ($\forall k\in V$).

\vspace{+2mm}

\begin{itemize}[topsep = 0.4pt, itemsep = 0.8pt, parsep = 0pt, labelwidth=0pt, leftmargin=*]
    \item $\mathcal{M}^{\text{LP}}_1$:
    \begin{itemize}[topsep = 0.4pt, itemsep = 0.8pt, parsep = 0pt, labelwidth=0pt, leftmargin=*]
    \item[] $\alpha_{k}$: dual variable of constraint (1b) for car $k\in V$.
    \item[] $\beta_{k}$: dual variable of constraint (1c) for car $k\in V\backslash \{1\}$, $\beta_k\geqslant 0$.
    \item[] $\gamma_p$: dual variable of constraint (1d) for downstream position $p\in P\backslash \{d\}$, $\gamma_p\geqslant 0$.
    \item[] $\lambda^1_{it}$: dual variable of constraint (1e) for node $(i,t)\in \mathcal{N}^{f}\cup \mathcal{N}^{r}$, $\lambda_{it}\leqslant 0$.
    \item[] $\theta^1_t$: dual variable of constraint (1f) for timestamp $t\in \{1,\dots, \tau-\bar{q}\}$, $\theta_{t}\leqslant 0$.
    \item[] $\omega^1_{t}$: dual variable of constraint (1g) for timestamp $t\in \{\tau^m, \dots, \tau\}$, $\omega_t\leqslant 0$.
    \item[] $\mu_k$: dual variable of constraint (1i) for car $k\in V$.
    \item[] $\xi^{t}_u$: dual variable of valid cut (1j) for makespan equal $t$ and the \textit{RL-entry plan} $u$, $\xi^t_u\leqslant 0$.
    \end{itemize}

    \item $\mathcal{M}^{\text{LP}}_2$:
    \begin{itemize}[topsep = 0.4pt, itemsep = 0.8pt, parsep = 0pt, labelwidth=0pt, leftmargin=*]
        \item[] $\lambda^2_{t}$ : dual variable of constraint (2a) for timestamp $t\in \{1,\dots, \tau-\bar{q}\}$, $\lambda^2_t\leqslant 0$.
        \item[] $\theta^2_t$: dual variable of constraint (2b) for timestamp $t\in \{\tau^m, \dots, \tau\}$, $\theta^2_t\leqslant 0$.
        \item[] $\omega^2_t$: dual variable of constraint (2c) for timestamp $t\in \{\tau^m, \dots, \tau-\tau^m\}$, $\omega^2_t\leqslant 0$.
        \item[] $\eta^l_t$: dual variable of constraint (2d) for lane $l\in L^f$ and timestamp $t\in \{1, \dots, \tau-1\}$, $\eta^l_t\leqslant 0$.
    \end{itemize}
\end{itemize}

\vspace{+2mm}

Let $\bar{\alpha}_k$, $\bar{\beta}_k$, $\bar{\gamma}_p$, $\bar{\lambda}^1_{it}$, $\bar{\theta}^1_t$, $\bar{\omega}^1_t$, $\bar{\mu}_k$, $\bar{\xi}^{t}_{u}$, $\bar{\lambda}^2_t$, $\bar{\theta}^2_t$, $\bar{\omega}^2_t$, and $\bar{\eta}^{l}_{t}$ denote their values of these variables. For each of the three RMPs, the reduced cost of \textit{path} $r\in \Omega_k$ in the corresponding TSN graph for each car $k\in V$ is calculated as follows:
\begin{itemize}[topsep = 0.4pt, itemsep = 0.8pt, parsep = 0pt, labelwidth=0pt, leftmargin=*]
    \item For $\mathcal{M}^{\text{LP}}_1$, the reduced cost $\bar{c}^{1k}_r$ of \textit{path} $r$ in $\mathcal{G}^1_k$ is:
\end{itemize}
\begin{equation}\label{rc1}
    %\begin{split}
        \bar{c}^{1k}_r=%&
        \bar{c}^k_r -\sum_{(i,t)\in \mathcal{N}^{f}\cup \mathcal{N}^{r}}{\rho^{kr}_{it}\cdot \bar{\lambda}^1_{it}} - \sum_{t=1}^{\tau-\bar{q}}{\sum_{l\in L^f}{\tilde{\rho}^{kr}_{c_{l1},t}\cdot \bar{\theta}^1_t}} - \sum_{t=\tau^m}^{\tau}{(\rho^{kr}_{et}+\tilde{\rho}^{kr}_{c_{\bar{l}\bar{q}},t})\cdot \bar{\omega}^1_t} -\zeta^k_r\cdot \bar{\eta}_k% \\
        %&+t^k_r-\sum_{t}{\sum_{(s,t)}{\bar{\xi}^{\bar{m}}_c}}\;\ \big(\text{if }p_k=d\big)
        %\end{split}
\end{equation}
\noindent if $p_k=d$, then
\begin{align}\label{rc1-1}
    \bar{c}^{1k}_r= &\bar{c}^k_r -\sum_{(i,t)\in \mathcal{N}^{f}\cup \mathcal{N}^{r}}{\rho^{kr}_{it}\cdot \bar{\lambda}^1_{it}} - \sum_{t=1}^{\tau-\bar{q}}{\sum_{l\in L^f}{\tilde{\rho}^{kr}_{c_{l1},t}\cdot \bar{\theta}^1_t}} - \sum_{t=\tau^m}^{\tau}{(\rho^{kr}_{et}+\tilde{\rho}^{kr}_{c_{\bar{l}\bar{q}},t})\cdot \bar{\omega}^1_t} -\zeta^k_r\cdot \bar{\eta}_k\nonumber \\
    +&t^k_r-\sum_{t\geqslant \bar{m}_r}{\sum_{u\in \mathcal{U}(t)}{\bar{\xi}^{t}_u}}
\end{align}
where $\bar{m}_r$ is the downstream arrival timestamp of car $k$ along \textit{path} $r$, and $\mathcal{U}(t)$ is the set of \textit{RL-entry plans} that are proven infeasible when paired with makespan $t$.
\begin{itemize}[topsep = 0.4pt, itemsep = 0.8pt, parsep = 0pt, labelwidth=0pt, leftmargin=*]
    \item For $\mathcal{M}^{\text{LP}}_2$, the reduced cost $\bar{c}^{2k}_r$ of \textit{path} $r$ in $\mathcal{G}^2_k$ is:
\end{itemize}
\begin{equation}\label{rc2}
    \bar{c}^{2k}_r=\bar{c}^k_r
        -\sum_{t=1}^{\tau-\bar{q}}{\rho^{kr}_{It}\cdot \bar{\lambda}^2_t} -\sum_{t=\tau^m}^{\tau}{\rho^{kr}_{Ot}\cdot \bar{\theta}^2_t} - \sum_{t=\tau^m}^{\tau-\tau^m}{\rho^{kr}_{Rt}\cdot \bar{\omega}^2_{t}} - \sum_{l\in L^f}{\sum_{t=1}^{\tau-1}{\rho^{kr}_{lt}}\cdot \bar{\eta}^l_t} - \zeta^k_r\cdot \bar{\eta}_k
\end{equation}
\noindent if $p_k=d$, then
\begin{align}\label{rc2-1}
    \bar{c}^{2k}_r=&\bar{c}^k_r
        -\sum_{t=1}^{\tau-\bar{q}}{\rho^{kr}_{It}\cdot \bar{\lambda}^2_t} -\sum_{t=\tau^m}^{\tau}{\rho^{kr}_{Ot}\cdot \bar{\theta}^2_t} - \sum_{t=\tau^m}^{\tau-\tau^m}{\rho^{kr}_{Rt}\cdot \bar{\omega}^2_{t}} - \sum_{l\in L^f}{\sum_{t=1}^{\tau-1}{\rho^{kr}_{lt}}\cdot \bar{\eta}^l_t} - \zeta^k_r\cdot \bar{\eta}_k \nonumber \\
        +&t^k_r-\sum_{t\geqslant \bar{m}_r}{\sum_{u\in \mathcal{U}(t)}{\bar{\xi}^{t}_u}}
\end{align}
where $\bar{m}_r$ is the downstream arrival timestamp of car $k$ along \textit{path} $r$, and $\mathcal{U}(t)$ is the set of \textit{RL-entry plans} that are proven infeasible when paired with makespan $t$.

\begin{itemize}[topsep = 0.4pt, itemsep = 0.8pt, parsep = 0pt, labelwidth=0pt, leftmargin=*]
    \item For $\mathcal{M}^{\text{LP}}_3$, the reduced cost $\bar{c}^{3k}_r$ of \textit{path} $r$ in $\mathcal{G}^3_k$ is:
\end{itemize}
\begin{equation}\label{rc3}
    \bar{c}^{3k}_{r}=\bar{c}^{2k}_r+\sum_{l\in L^f}{\sum_{t=1}^{\tau-1}{\rho^{kr}_{lt}\cdot\bar{\eta}^l_t}}
\end{equation}
\noindent where
\begin{equation*}
    \bar{c}^{k}_r=-\bar{\alpha}_k -
        \begin{cases}
        -\bar{\beta}_2\cdot s^k_r, &k=1 \\
        (\bar{\beta}_k-\bar{\beta}_{k+1})\cdot s^k_r, &k\neq 1,d \\
        \bar{\beta}_d\cdot s^k_r, &k=d
        \end{cases} -
        \begin{cases}
        -\bar{\gamma}_1\cdot t^k_r, &p_k=1 \\
        (\bar{\gamma}_{p_k-1}-\bar{\gamma}_{p_k})\cdot t^k_r, &p_k\neq 1,d \\
        \bar{\gamma}_{d-1}\cdot t^k_r, &p_k=d
        \end{cases}
\end{equation*}

In our \textit{labeling algorithm}, the label $(nl_1, nl_2, nl_3, nl_4, nl_5)$ for each node contains: a fixed RL-entry count ($nl_1$), a binary flag ($nl_2$) indicating whether the label is visited, the shortest distance ($nl_3$) from the \textit{source} under that RL-entry count, the index of the incoming arc ($nl_4$) that realizes this distance, and the index of the parent label ($nl_5$) at the predecessor node connected by arc $nl_4$. We assign a label $(0, 1, 0, -1, -1)$ to the \textit{source} and $\bar{n}_k+1$ labels to all others. The $n$-th label ($n\in N_k$), initialized as $(n, 0, +\infty, -1, -1)$, indicates that car $k$ enters the RL $n$ times along a partial \textit{path} from the \textit{source}.

The algorithm extends partial \textit{paths} and updates labels via a queue of zero-in-degree nodes, which initially contains the \textit{source}. Each time the head node is dequeued, the in-degrees of its successor nodes along all available outgoing arcs are decremented by one, and any successor reaching zero in-degree is enqueued. For each label $(nl_1, nl_2, nl_3, nl_4, nl_5)$ of the dequeued node with $nl_2 = 1$, the label at each successor node connected by an available outgoing arc $a$ is updated. In the case where $a$ represents entry into the RL, i.e., \textup{(i)} a \textit{moving arc} in $\mathcal{A}^{fr}_k$ of $\mathcal{G}^1_k$, \textup{(ii)} an \textit{enter-RL arc} in $\mathcal{A}^{FR}_k$ of $\mathcal{G}^2_k$, or \textup{(iii)} a \textit{use-RL arc} in $\mathcal{A}^{sR}_k$ or a \textit{re-RL arc} in $\mathcal{A}^{RR}_k$ of $\mathcal{G}^3_k$, the $nl_1+1$-th label of the successor node, $(nl_{1}+1, nl_2^{\prime}, nl_3^{\prime}, nl_4^{\prime}, nl_5^{\prime})$, is updated by setting the visited flag $nl_2^{\prime}=1$. If the current shortest distance $nl_3^{\prime}$ exceeds $nl_3+\bar{c}_a$, where $\bar{c}_a$ is the cost of arc $a$, then $nl_3^{\prime}$ is set to $nl_3+\bar{c}_a$ and $nl_4^{\prime}$ is set to $a$. In the other case, where $a$ does not involve RL-entry, we update the $nl_1$-th label of the successor node analogously.

The above process continues until the queue is empty. When finished, for each RL-entry count $n\in N_k$, if the $n$-th label at the \textit{terminal} is marked as visited, i.e., $nl_2 = 1$, then the corresponding shortest \textit{path} can be retrieved by backtracking from the stored incoming arc indexes $nl_4$ and parent label indexes $nl_5$.

Accounting for arc availability affected by branching rules, we perform a depth-first search to compute the in-degree of each node before the \textit{labeling algorithm}. If the \textit{terminal} has zero in-degree, i.e., it is unreachable from the \textit{source}, then SP$_k$ is confirmed infeasible, allowing us to prune the current BB node.

\newpage

\section{Flowchart of the BP-CP Algorithm}\label{app-sec6}
\begin{figure}[htb]
\noindent\makebox[\linewidth][c]{
    \centering
    \scalebox{0.75}{
    \begin{tikzpicture}[font=\small,thick]
        \filldraw[fill=gray!10, draw=black,loosely dashed, thin] (-20.2,18.35) -- (-9.8,18.35) -- (-9.8,11.85)
		-- (-20.2,11.85) -- (-20.2,18.35);
        \node[font=\large] at (-14,12.6) {\textbf{Column Generation}};
        \node[font=\large, blue, dashed] at (-10.5,7.8) {\textbf{Heuristic Procedures}};
        %node1
		\node[draw, rounded rectangle, minimum width=1.8cm, minimum height=0.9cm] (n1) {START};
        %node2
		\node[draw, minimum height=1cm, left=of n1, xshift=-2mm] (n2)
		{\begin{minipage}[]{2.8cm}
				\centering
				Compute the lower bound $\underline{\tau}^e_{\pi_d}$
		\end{minipage}};
        %node3
		\node[draw, minimum height=1cm, left=of n2, xshift=+2mm] (n3)
		{\begin{minipage}[]{2.9cm}
				\centering
				Find all \textit{maximal cliques} $\mathcal{C}^{\text{all}}$ and inequalities (5)
		\end{minipage}};
        %node4
		\node[draw, minimum height=1cm, left=of n3, xshift=+2mm] (n4)
		{\begin{minipage}[]{3.1cm}
				\centering
				Solve $\mathcal{M}^{\text{CP}}$ with pair $(\underline{\tau}^e_{\pi_d},u^{\text{con}}_1)$ under TimeLimit$_1$
		\end{minipage}};
        %node5
		\node[draw, diamond, left=of n4,  aspect=2.5,fill=white, xshift=+2mm] (n5)
		{\begin{minipage}[]{1.45cm}
				\centering
                Is $\mathcal{M}^{\text{CP}}$ feasible?
		\end{minipage}};
        %node6
		\node[draw, minimum height=1cm, left=of n5, xshift=+1mm] (n6)
		{\begin{minipage}[]{3.5cm}
				\centering
				Initialize $ub$, $lb$, and queue $Q=\{\text{root}\}$
		\end{minipage}};
        %node7
		\node[draw, rounded rectangle, minimum width=1.8cm, minimum height=0.9cm, below = of n5,yshift=+2.5mm] (n7) {END};
        %node8
		\node[draw, diamond, above=of n6,  aspect=1.5,fill=white, inner sep=0.5pt, inner sep=0.5pt, yshift=-5mm] (n8)
		{\begin{minipage}[]{2.05cm}
				\centering
                %$Q$ is non-empty \&\&
                $ub=lb$ $\parallel$
                RunTime $\geq$ TimeLimit$_2$?
		\end{minipage}};
        %node9
		\node[draw, minimum height=1cm, above=of n8, yshift=+10.7cm] (n9)
		{\begin{minipage}[]{1.7cm}
				\centering
				Pop the next node
		\end{minipage}};
        %node10
		\node[draw, minimum height=1cm, right=of n9, xshift=-3mm, fill=white] (n10)
		{\begin{minipage}[]{1.7cm}
				\centering
				Solve the RMP
		\end{minipage}};
        %node11
		\node[draw, minimum height=1cm, right=of n10, xshift=+2mm, fill=white] (n11)
		{\begin{minipage}[]{1.7cm}
				\centering
				Solve the SPs
		\end{minipage}};
        %node12
		\node[draw, diamond, right=of n11,  aspect=1.3,fill=white, inner sep=0.5pt, xshift=-2mm] (n12)
		{\begin{minipage}[]{1.55cm}
				\centering
                Find any new columns?
		\end{minipage}};
        %node13
		\node[draw, minimum height=1cm, below=of n12, yshift=+5mm, fill=white] (n13)
		{\begin{minipage}[]{1.7cm}
				\centering
				Add new columns
		\end{minipage}};
        %node14
		\node[draw, minimum height=1cm, above=of n11, yshift=-5mm, fill=white] (n14)
		{\begin{minipage}[]{3.6cm}
				\centering
				Remove \textit{incompatible} columns and arcs
		\end{minipage}};
        %node15
		\node[draw, diamond, right=of n12,  aspect=1.5,fill=white, inner sep=0.5pt,xshift=-1mm] (n15)
		{\begin{minipage}[]{1.46cm}
				\centering
                Is $\widehat{m}$ integral?
		\end{minipage}};
        %node16
		\node[draw, dashed, blue, diamond, right=of n15,  aspect=2.2,fill=white, inner sep=0.5pt,xshift=-2mm] (n16)
		{\begin{minipage}[]{3.5cm}
				\centering
                Are the $\mathbf{z}$ variables for all \textit{non-promising} cars integers?
		\end{minipage}};
        %node17
		\node[draw, dashed, blue, minimum height=1cm, below=of n16, yshift=+2mm] (n17)
		{\begin{minipage}[]{3.2cm}
				\centering
				Construct \textit{RL-entry plan} $u^{\text{con}}_2$
		\end{minipage}};
        %node18
        \node[draw, diamond, dashed, blue, below=of n17, aspect=2.8, fill=white, yshift=+2mm,
        text width=3.52cm,  % 直接设置文本宽度
        align=center,  % 文本居中对齐
        inner sep=0.5pt]  % 内部间距
      (n18)
        {Can find $\widehat{m}^{\prime}$ between $\widehat{m}$ and $ub-1$ such that $(\widehat{m}^{\prime}, u^{\text{con}}_2)$ satisfies (1j)?
        %Do all pairs $(\widehat{m}^{\prime},u^{\text{con}}_2)$, where $\widehat{m}^{\prime}=\widehat{m},\dots,ub-1$, violate (1j)?
        };
        %node19
        \node[draw, diamond, below=of n18, aspect=2.5, fill=white, yshift=+4mm,
        text width=2.6cm,  % 直接设置文本宽度
        align=center,  % 文本居中对齐
        inner sep=0.5pt]  (n19)
        {Are the $\mathbf{z}$ variables for all cars integers?};
        %node20
		\node[draw, minimum height=1cm, below=of n19, yshift=+4mm] (n20)
		{\begin{minipage}[]{3.7cm}
				\centering
				Get \textit{RL-entry plan} $u_3$
		\end{minipage}};
        %node21
		\node[draw, minimum height=1cm, left=of n20, xshift=+3mm] (n21)
		{\begin{minipage}[]{2.75cm}
				\centering
				Solve $\mathcal{M}^{\text{CP}}$ with pair $(\widehat{m},u_3)$
		\end{minipage}};
        %node22
		\node[draw, diamond, left=of n21,  aspect=2.9,fill=white, xshift=+3mm] (n22)
		{\begin{minipage}[]{1.42cm}
				\centering
                Is $\mathcal{M}^{\text{CP}}$ feasible?
		\end{minipage}};
        %node23
		\node[draw, dashed, blue, minimum height=1cm, left=of n18, xshift=+2mm] (n23)
		{\begin{minipage}[]{3.1cm}
				\centering
				Solve $\mathcal{M}^{\text{CP}}$ with each $(\widehat{m}^{''},u^{\text{con}}_2)$ under TimeLimit$_3$ ($\widehat{m}^{''}=\widehat{m}^{\prime},\dots,ub-1$)
		\end{minipage}};
        %node24
		\node[draw, diamond, dashed, blue, left=of n23,  aspect=1.9,fill=white, inner sep=0.5pt, xshift=+3mm] (n24)
		{\begin{minipage}[]{1.55cm}
				\centering
                Find a feasible solution?
		\end{minipage}};
        %node25
		\node[draw, dashed, blue, minimum height=1cm, left=of n24, xshift=+2mm] (n25)
		{\begin{minipage}[]{1.25cm}
				\centering
				Update $ub$
		\end{minipage}};
        %node26
		\node[draw, diamond, dashed, blue, below=of n24,  aspect=3,fill=white, inner sep=0.5pt, yshift=-4mm] (n26)
		{\begin{minipage}[]{2.3cm}
				\centering
                Find any new cuts (4)?
		\end{minipage}};
        %node27
		\node[draw, minimum height=1cm, below=of n22, yshift=+3mm,inner sep=0.5pt] (n27)
		{\begin{minipage}[]{1.9cm}
				\centering
				Update $ub$
		\end{minipage}};
        %node28
		\node[draw, minimum height=1cm, left=of n27, xshift=0mm,inner sep=0.5pt] (n28)
		{\begin{minipage}[]{1.8cm}
				\centering
				Update $lb$
		\end{minipage}};
        %node29
		\node[draw, minimum height=1cm, below=of n10, yshift=-1.1cm, fill=white] (n29)
		{\begin{minipage}[]{1.8cm}
				\centering
				Add new cut(s) (4)
		\end{minipage}};
        %node30
		\node[draw, minimum height=1cm, right=of n16, xshift=-1mm] (n30)
		{\begin{minipage}[]{1.7cm}
				\centering
				Branching
		\end{minipage}};
        \draw[-Stealth,
		black,
		line width=1.1pt] (n1) -- (n2);
        \draw[-Stealth,
		black,
		line width=1.1pt] (n2) -- (n3);
        \draw[-Stealth,
		black,
		line width=1.1pt] (n3) -- (n4);
        \draw[-Stealth,
		black,
		line width=1.1pt] (n4) -- (n5);
        \draw[-Stealth,
		black,
		line width=1.1pt] (n5) -- node[above,pos=0.38]{No}(n6);
        \draw[-Stealth,
		black,
		line width=1.1pt] (n5) -- node[left,pos=0.38]{Yes}(n7);
        \draw[-Stealth,
		black,
		line width=1.1pt] (n6) -- (n8);
        \draw[-Stealth,
		black,
		line width=1.1pt] (n8) -- node[left,pos=0.5]{No}(n9);
        \draw[-Stealth,
		black,
		line width=1.1pt] (n8.west) -- ++(-0.5cm,0cm) node[yshift=-4.15cm, xshift=+3cm]{Yes} |- (n7.west);
        \draw[-Stealth,
		black,
		line width=1.1pt] (n9) -- (n10);
        \draw[-Stealth,
		black,
		line width=1.1pt] (n10) -- (n11);
        \draw[-Stealth,
		black,
		line width=1.1pt] (n11) -- (n12);
        \draw[-Stealth,
		black,
		line width=1.1pt] (n12) -- node[left,pos=0.4]{Yes}(n13);
        \draw[-Stealth,
		black,
		line width=1.1pt] (n13.west) -| (n10);
        \draw[-Stealth,
		black,
		line width=1.1pt] (n12) -- node[above,midway]{No}(n15);
        \draw[-Stealth,
		black,
		line width=1.1pt] (n29) -- (n10);
        \draw[-Stealth,
		black,
		line width=1.1pt] (n15) -- node[above,pos=0.35]{Yes}(n16);
        \draw[-Stealth,
		black,
		line width=1.1pt] (n15.north) |- node[below,pos=0.675]{No}(n14);
        \draw[-Stealth,
		black,
		line width=1.1pt] (n14) -| (n10);
        \draw[-Stealth,
		black,
		line width=1.1pt] (n16) -- node[above,pos=0.45]{No}(n30);
        \draw[-Stealth,
		black,
		line width=1.1pt] (n28) -- (n8);
        \draw[-Stealth,
		black,
		line width=1.1pt] (n16) -- node[left,pos=0.38]{Yes}(n17);
        \draw[-Stealth,
		black,
		line width=1.1pt] (n17) -- (n18);
        \draw[-Stealth,
		black,
		line width=1.1pt] (n18) -- node[left,pos=0.38]{Yes}(n19);
        \draw[-Stealth,
		black,
		line width=1.1pt] (n19) -- node[left,pos=0.38]{Yes}(n20);
        \draw[-Stealth,
		black,
		line width=1.1pt] (n20) -- (n21);
        \draw[-Stealth,
		black,
		line width=1.1pt] (n21) -- (n22);
        \draw[-Stealth,
		dashed, blue,
		line width=1.1pt] (n18) -- node[above,pos=0.38]{No}(n23);
        \draw[-Stealth,
		dashed, blue,
		line width=1.1pt] (n23) -- (n24);
        \draw[-Stealth,
		dashed, blue,
		line width=1.1pt] (n24) -- node[above,pos=0.38]{Yes}(n25);
        \draw[-Stealth,
		dashed, blue,
		line width=1.1pt] (n24) -- node[left,pos=0.38]{No}(n26);
        \draw[-Stealth,
		dashed, blue,
		line width=1.1pt] (n26) -- node[above,midway]{No}(n19);
        \draw[-Stealth,
		dashed, blue,
		line width=1.1pt] (n25) |- (n26);
        \draw[-Stealth,
		dashed, blue,
		line width=1.1pt] (n26.south) -- ++(0cm,-0.5cm) node[yshift=+0.3cm, xshift=-3.5cm]{Yes} -| (n29.south);
        \draw[-Stealth,
		black,
		line width=1.1pt] (n22) -- node[left,pos=0.4]{Yes}(n27);
        \draw[-Stealth,
		black,
		line width=1.1pt] (n27) -- (n28);
        \draw[-Stealth,
		black,
		line width=1.1pt] (n19) -| node[above,pos=0.28]{No}(n30);
        \draw[-Stealth,
		black,
		line width=1.1pt] (n30.east) -- ++(+0.2cm,0cm) --(+3.75cm,+1.2cm) -| (n28.south);
        \draw[-Stealth,
		black,
		line width=1.1pt] (n22) -| node[above,pos=0.3]{No}(n29);
    \end{tikzpicture}
}}
    \caption{Flowchart of the BP-CP Algorithm.}
    \label{app-fig2}
\end{figure}
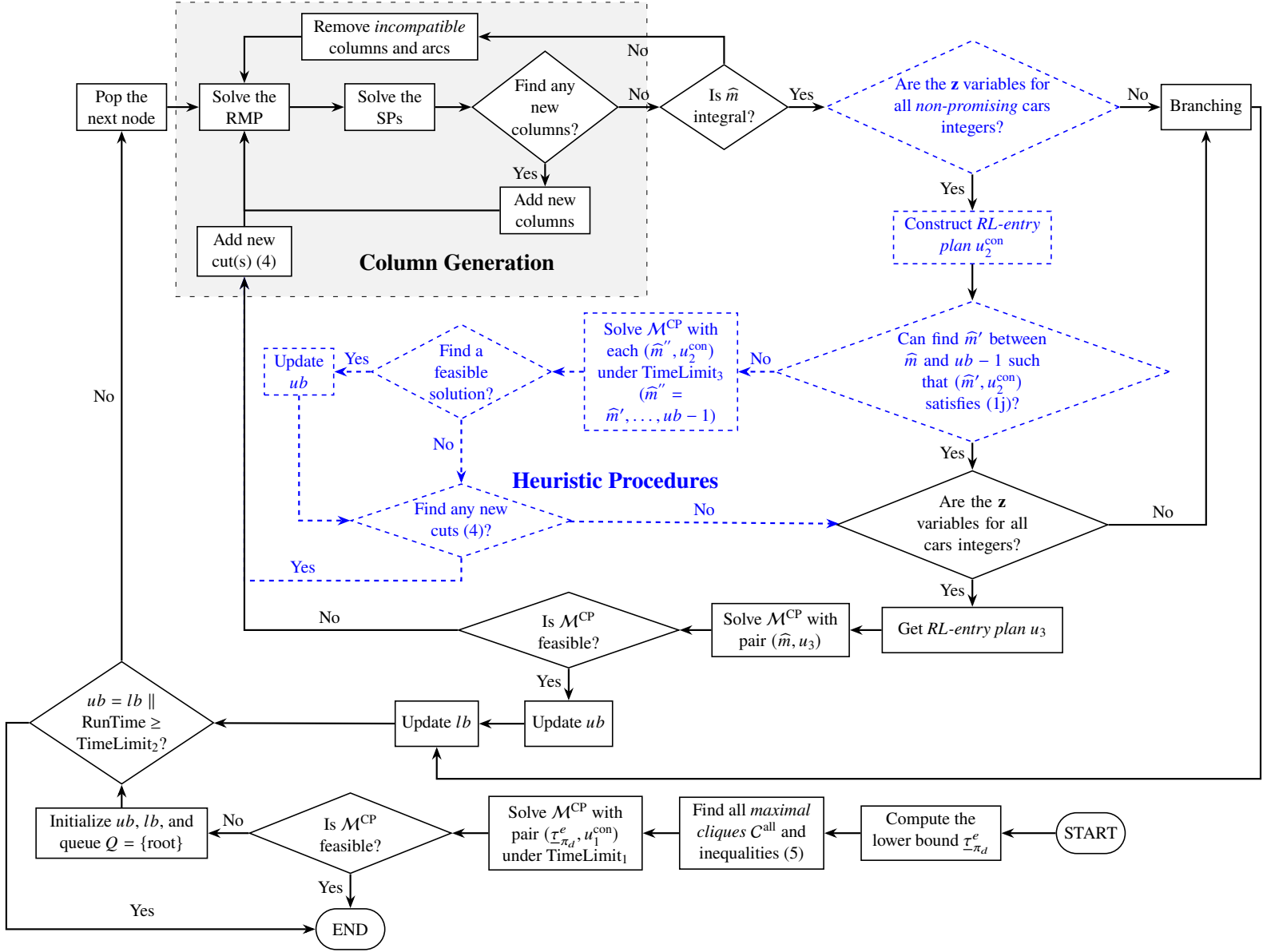
We set
\begin{itemize}
    \item[-] TimeLimit$_1$=TimeLimit$_3$=2 seconds for $d\leqslant 50$,
    \item[-]
    TimeLimit$_1$=TimeLimit$_3$=5 seconds for $d> 50$,
    \item[-] TimeLimit$_2$=3,600 seconds.
\end{itemize}

\bibliographystyleapp{plainnat}
\bibliographyapp{myref}

%%%%%%%%%%%%%%%%%
\end{document}